%% file: main.tex
\documentclass[gdplain]{geradwp}

\PassOptionsToPackage{hyphens}{url}

\usepackage[english]{babel}
\usepackage[utf8]{inputenc}
\usepackage[T1]{fontenc}
\usepackage{lmodern}
\usepackage{csquotes}
\usepackage[square,numbers,comma]{natbib}
\usepackage{lastpage}
\usepackage{xspace}
\usepackage{fancyhdr}
\usepackage{geometry}
\usepackage{setspace}
\usepackage{array}
\usepackage{lscape}
\usepackage{enumitem}
\usepackage{listings}
\usepackage{hyperref}
\usepackage{caption}
\usepackage{graphicx}
\usepackage{float}
\usepackage{amsmath}
\usepackage{amsthm}
\usepackage{dsfont}
\usepackage{amssymb}
\usepackage{stmaryrd}
\usepackage{mathtools}
\usepackage[ruled]{algorithm}
\usepackage[noend]{algpseudocode}
\usepackage{textcomp}
\usepackage{gensymb}
\usepackage{eurosym}
\usepackage[textsize=scriptsize]{todonotes}
\usepackage{indentfirst}
\usepackage{xstring}

\GDcoverpagewhitespace{6.8cm}
\hypersetup{colorlinks,
    citecolor={black},
    urlcolor={blue},
    linkcolor={blue},
    breaklinks={true}
}

\makeatletter
\ifthenelse{\isundefined{\ALG@name}}{}{\renewcommand{\ALG@name}{\sffamily\footnotesize Algorithm}}
\makeatother
\ifthenelse{\isundefined{\SetAlCapNameFnt}}{}{\SetAlCapNameFnt{\footnotesize}\SetAlCapFnt{\sffamily\footnotesize}}

\newcolumntype{L}[1]{>{\raggedright\let\newline\\\arraybackslash\hspace{0pt}}m{#1}}
\newcolumntype{C}[1]{>{\centering\let\newline\\\arraybackslash\hspace{0pt}}m{#1}}
\newcolumntype{R}[1]{>{\raggedleft\let\newline\\\arraybackslash\hspace{0pt}}m{#1}}

\newcommand{\refbis}[2]{\ref{#1}.\ref{#1/#2}}

\newcommand{\Acal}{\mathcal{A}}   \newcommand{\Abb}{\mathbb{A}}
   \renewcommand{\Bbb}{\mathbb{B}}
\newcommand{\Ccal}{\mathcal{C}}   
\newcommand{\Dcal}{\mathcal{D}}

\newcommand{\Hcal}{\mathcal{H}}   
\newcommand{\Ical}{\mathcal{I}}

\newcommand{\Lcal}{\mathcal{L}}   
   
 \newcommand{\Nbf}{\mathbf{N}}  \newcommand{\Nbb}{\mathbb{N}}
   
 \newcommand{\Pbf}{\mathbf{P}}  
   
   \newcommand{\Rbb}{\mathbb{R}}
   
\newcommand{\Tcal}{\mathcal{T}}

\newcommand{\Xcal}{\mathcal{X}}   \newcommand{\Xbb}{\mathbb{X}}
\newcommand{\Ycal}{\mathcal{Y}}   \newcommand{\Ybb}{\mathbb{Y}}
   \newcommand{\Zbb}{\mathbb{Z}}

\newcommand{\algoname}[1]{\texttt{#1}\xspace}
\newcommand{\dfo}     {\algoname{DFO}}

\newcommand{\pof}     {\algoname{POf}}
\newcommand{\pom}     {\algoname{DFPOm}}
\newcommand{\dsm}     {\algoname{DSM}}
\newcommand{\cdsm}    {\algoname{cDSM}}
\newcommand{\poll}    {\algoname{poll}}
\newcommand{\covering}{\algoname{covering}}
\newcommand{\search}  {\algoname{search}}
\newcommand{\update}  {\algoname{update}}
\newcommand{\pmp}     {\algoname{Pmp}}
\newcommand{\DcalS}{\Dcal_{\tt{S}}} \newcommand{\TcalS}{\Tcal_{\tt{S}}} \newcommand{\tS}{t_{\tt{S}}} 
\newcommand{\DcalC}{\Dcal_{\tt{C}}} \newcommand{\TcalC}{\Tcal_{\tt{C}}} \newcommand{\tC}{t_{\tt{C}}} 
\newcommand{\DcalP}{\Dcal_{\tt{P}}} \newcommand{\TcalP}{\Tcal_{\tt{P}}} \newcommand{\tP}{t_{\tt{P}}} 
 \newcommand{\TcalA}{\Tcal_{\tt{A}}} \newcommand{\tA}{t_{\tt{A}}} 

\newcommand{\solvername}[1]{\textsc{#1}\xspace}
\newcommand{\nomad}{\solvername{Nomad}}
\newcommand{\prima}{\solvername{Prima}}

\algrenewcommand{\textproc}{\textit}
\algrenewcommand{\Return}{\State\algorithmicreturn\xspace}
\algnewcommand{\algorithmicelif    }{\algorithmicelse\ \algorithmicif}
\algnewcommand{\algorithmicendif   }{\algorithmicend\ \algorithmicif}
\algnewcommand{\algorithmicendfor  }{\algorithmicend\ \algorithmicfor}
\algnewcommand{\algorithmicendwhile}{\algorithmicend\ \algorithmicwhile}
\algnewcommand{\algorithmicendloop }{\algorithmicend\ \algorithmicloop}

\usepackage{forloop}
\newcounter{ct}
\newcommand{\markdent}[1]{\hspace{-\algorithmicindent}\forloop{ct}{0}{\value{ct} < #1}{\hspace{1em}}}
\newcommand{\IState}[1]{\Statex\markdent{#1}}

\DeclareMathOperator*{\argmin}{argmin}
\DeclareMathOperator*{\argmax}{argmax}
\DeclareMathOperator*{\minimize}{minimize}

\let\integ\int
\newcommand{\dsum}[2]       {\sum\limits_{#1}^{#2}}

\newcommand{\dinteg}[2]{\displaystyle\integ_{#1}^{#2}}
\newcommand{\defequal}{\triangleq}
\newcommand{\minusinfty}{{-}\infty}
\newcommand{\1}             {\mathds{1}}
\newcommand{\abs}[1]        {\left\lvert{#1}\right\rvert}
\newcommand{\textabs}[1]    {\lvert{#1}\rvert}
\newcommand{\norm}[1]       {\left\lVert{#1}\right\rVert}
\newcommand{\textnorm}[1]   {\lVert{#1}\rVert}
\newcommand{\floor}[1]      {\left\lfloor{#1}\right\rfloor}

\newcommand{\ceil}[1]       {\left\lceil{#1}\right\rceil}
\newcommand{\textceil}[1]   {\lceil{#1}\rceil}
\newcommand{\floorceil}[1]  {\left\lfloor{#1}\right\rceil}

\newcommand{\sqrtp}[2]      {\sqrt[\leftroot{-2}\uproot{5}#1]{#2}}
\newcommand{\Nbbinf}        {\overline{\Nbb}}
\newcommand{\Rbbinf}        {\overline{\Rbb}}
\newcommand{\llb}           {\llbracket}
\newcommand{\rrb}           {\rrbracket}
\renewcommand{\int}{\mathrm{int}}
\newcommand{\cl}{\mathrm{cl}}
\newcommand{\acc}{\mathrm{acc}}
\newcommand{\ACC}{\mathcal{ACC}}
\newcommand{\partitioncup}  {\sqcup}

\newcommand{\dist}          {\mathrm{dist}}
\newcommand{\pointplusset}[2]{\{#1\}+#2}
\newcommand{\barrier}[2]{\overline{#1_{#2}}}

\newcommand{\fct}[5]{
    #1:\left\{\begin{array}{ccl}
        #3 & \to     & #5\\
        #2 & \mapsto & #4
    \end{array}\right.}

\newcommand{\compactarray}[1]{\begingroup\renewcommand{\arraystretch}{1.00}\begin{array}{c}#1\end{array}\endgroup}
\newcommand{\problemoptimfree}[3]{\underset{\compactarray{#2}}{#1} \quad #3}
\newcommand{\problemoptimoneline}[4]{\problemoptimfree{#1}{#2}{#3 \quad \mbox{subject to} \quad #4}}
\newcommand{\problemoptim}[4]{
    \begin{array}{cl}
        \underset{\compactarray{#2}}{#1} & \begin{array}{l}#3\end{array} \\ [1ex]
        \mbox{subject to} & \begin{array}[t]{ll}#4\end{array}
    \end{array}
}

\newtheorem{assumption}{Assumption}
\newtheorem{proposition}{Proposition}
\newtheorem{lemma}{Lemma}
\newtheorem{definition}{Definition}
\newtheorem{theorem}{Theorem}
\newtheorem{remark}{Remark}

\newtheorem{framework}{Framework}

\newlength{\myeqskip}  
\AtBeginDocument{%
    \setlength\abovedisplayskip{\myeqskip}%
    \setlength\belowdisplayskip{\myeqskip}%
    \setlength\abovedisplayshortskip{\myeqskip-\baselineskip}%
    \setlength\belowdisplayshortskip{\myeqskip}}

\begin{document}

\GDtitle{A Partitioned Optimization Framework for Structure-Aware Problems}
\GDmonth{Octobre}{October}
\GDyear{2025}
\GDnumber{XX}
\GDauthorsShort{C. Audet, P.-Y. Bouchet, L. Bourdin}
\GDauthorsCopyright{P.-Y. Bouchet}
\GDpostpubcitation{Hamel, Benoit, Karine H\'ebert (2021). \og Un exemple de citation \fg, \emph{Journal of Journals}, vol. X issue Y, p. n-m}{https://www.gerad.ca/fr}
\GDrevised{XX}{XX}{20XX}


\begin{GDtitlepage}
    \begin{GDauthlist}
        \GDauthitem{Pierre-Yves Bouchet \ref{affil:polymtl}\GDrefsep\ref{affil:gerad}}
        \GDauthitem{Charles Audet \ref{affil:polymtl}\GDrefsep\ref{affil:gerad}}
        \GDauthitem{Loïc Bourdin \ref{affil:unilim}}
    \end{GDauthlist}
    \begin{GDaffillist}
        \GDaffilitem{affil:polymtl}{École Polytechnique de Montréal, Montr\'eal (Qc), Canada, H3T 1J4}
        \GDaffilitem{affil:gerad}{GERAD, Montr\'eal (Qc), Canada, H3T 1J4}
        \GDaffilitem{affil:unilim}{XLIM Research Institute, University of Limoges, France}
    \end{GDaffillist}
    \begin{GDemaillist}
        \GDemailitem{pierre-yves.bouchet@polymtl.ca}
        \GDemailitem{charles.audet@gerad.ca}
        \GDemailitem{loic.bourdin@unilim.fr}
    \end{GDemaillist}
\end{GDtitlepage}

\GDabstracts
\begin{GDabstract}{Abstract}
    This work tackles a class of optimization problems in which fixing some well-chosen combinations of the variables makes the problem substantially easier to solve.
    We consider that the variables space may be partitioned into subsets that fix these combinations to given values, so that the restriction of the problem to any of the partition sets admits a tractable solution.
    Then, we exhibit a reformulation of the problem that consists in searching for the partition set index that minimizes the objective value of the solution to the restricted problem.
    We name \textit{partitioned optimization framework} (\pof) the formalization of this class of problems and this reformulation process.
    As we prove in this work, the \pof allows solving the original problem by focusing on the reformulated problem: all solutions to the reformulated problem are partition indices for which the solution to the associated restricted problem is also a solution to the original problem.
    Second, we introduce a \textit{derivative-free partitioned optimization method} (\pom) to efficiently solve problems that fit in the \pof.
    We prove that the reformulated problem is nicely handled by \textit{derivative-free optimization} (\dfo) algorithms with a \textit{\covering step}.
    Then the \pom consists in solving the reformulated problem using such \dfo algorithm with a \covering step to obtain an optimal partition index, and to return the solution to the associated restricted problem as a solution to the initial problem.
    Finally, we illustrate how the \pof allows solving some classes of problems.
    We first focus on an infinite-dimensional case, by solving analytically an optimal control problem that challenges standard methods from the literature.
    Then, we apply the \pom on a class of finite-dimensional problems called \textit{composite greybox problems}, and we highlight the gain in numerical performance provided by the \pom by comparing it to two popular \dfo solvers.
\end{GDabstract}

\begin{GDacknowledgements}
    This work was supported by the first NSERC discovery grant RGPIN-2020-04448 of Audet.
\end{GDacknowledgements}

\clearpage \newpage
\GDarticlestart

\section{Introduction}
\label{section:intro}

In this work, we consider the optimization problem
\begin{equation}
    \label{problem:P_initial}
    \tag{$\Pbf_\texttt{ini}$}
    \problemoptimoneline{\minimize}{y \in \Ybb}{\varphi(y)}{y \in \Omega,}
\end{equation}
where~$\Ybb$ is a (possibly infinite-dimensional) Banach space, and~$\varphi: \Ybb \to \Rbbinf \defequal \Rbb \cup \{\pm \infty\}$ (possibly discontinuous) has a nonempty effective domain~$Y \defequal \{y \in \Ybb : \varphi(y) \neq \pm\infty\}$ intersecting~$\Omega \subseteq \Ybb$ (which is possibly not closed).
We consider that~$\Ybb$ may be partitioned so that restricting Problem~\eqref{problem:P_initial} to any partition set makes it significantly simpler to solve.
We introduce a \textit{partitioned optimization framework} (\pof) to study this context, and a \textit{derivative-free partitioned optimization method} (\pom) to efficiently solve Problem~\eqref{problem:P_initial} when it fits in the \pof.
Section~\ref{section:intro/illustrative_example} illustrates the \pof, Section~\ref{section:intro/motivation} discusses our motivation to develop it, and Section~\ref{section:intro/contribution_outline} outlines this work and our contribution.

\subsection{Illustrative example with a composite greybox problem}
\label{section:intro/illustrative_example}

This section illustrates the \pof with two variables ($\Ybb \defequal \Rbb^2$), no constraints ($\Omega \defequal \Ybb$), and
\begin{equation*}
    \fct{\varphi}{(y_1,y_2)}{\Ybb}{(y_2-\sigma(y_1))^2+\varepsilon(y_1)}{\Rbb}
\end{equation*}
shown in Figure~\ref{figure:intro_example} (this example is solved in Section~\ref{section:applications/monovariable}).
Note that~$\varphi$ is explicitly defined with respect to the functions~$\sigma$ and~$\varepsilon$, while~$\sigma$ and~$\varepsilon$ are \textit{blackboxes}~\cite{AuHa2017} since their expressions are not provided.

\begin{figure}[!hb]
    \centering
    \includegraphics[width=0.43\linewidth, trim = 0 20  0 40, clip]{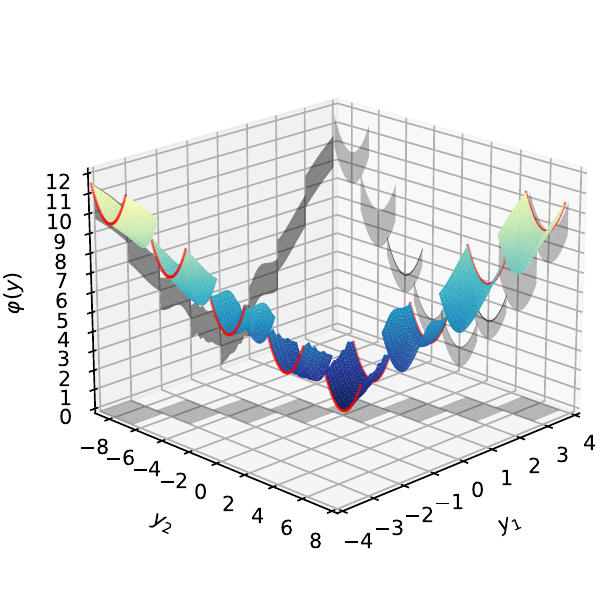}
    \includegraphics[width=0.41\linewidth, trim = 0  0  0  0, clip]{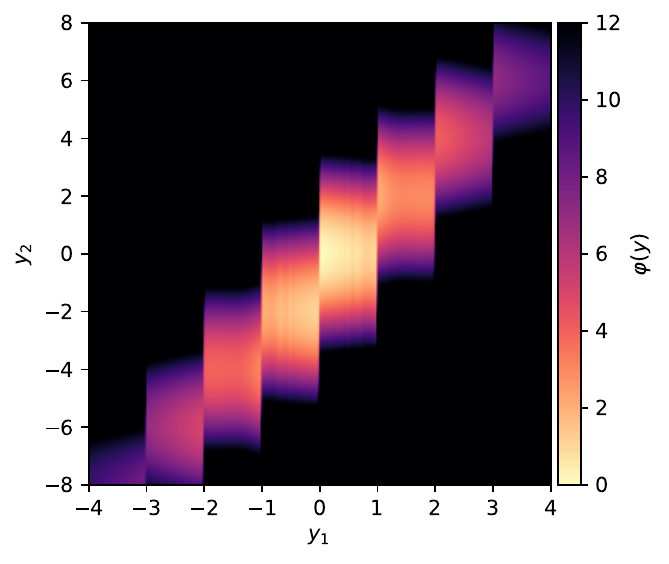}
    \hfill
    \caption{The function~$\varphi$ from Sections~\ref{section:intro/motivation} and~\ref{section:applications/monovariable}, minimized efficiently via the \pof. For readability, at each~$y_1 \in [-4,4]$ the two graphs actually plot only the values of the quadratic~$y_2 \mapsto \varphi(y_1,y_2)$ for~$y_2 \in [\sigma(y_1)-1.3,\ \sigma(y_1)+1.3]$.}
    \label{figure:intro_example}
\end{figure}

Minimizing~$\varphi$ over~$\Ybb$ is difficult.
Yet,~$\varphi$ has a simpler expression when~$y_1$ is fixed: for all~$y_1 \in \Rbb$, the function~$y_2 \mapsto \varphi(y_1, y_2)$ is quadratic.
We exploit this property by partitioning~$\Ybb$ as~$\Ybb = \partitioncup_{x \in \Xbb} \Ybb(x)$, where for all partition set index~$x \in \Xbb \defequal \Rbb$, the partition set~$\Ybb(x) \defequal \{x\} \times \Rbb$ is such that the restriction of~$\varphi$ to~$\Ybb(x)$ (denoted by~$\varphi_{|\Ybb(x)}$) is simple to minimize.
We define an \textit{oracle function}~$\gamma: \Xbb \to \Ybb$ that, for all~$x \in \Xbb$, computes the minimizer of~$\varphi_{|\Ybb(x)}$.
Its associated objective function value is denoted by~$\Phi(x) \defequal \varphi(\gamma(x))$.
As we highlight in this work, the \textit{optimal partition set index}~$x^*$ minimizing~$\Phi$ provides a solution to the initial problem as the point~$\gamma(x^*)$.
In this example, for all~$x \in \Xbb$, the minimizer of~$\varphi_{|\Ybb(x)}$ is~$\gamma(x) \defequal (x,\sigma(x))$, and its objective function value equals~$\Phi(x) = \varepsilon(x)$.
The minimizer~$x^* \defequal 0$ of~$\Phi$ yields~$\gamma(x^*) = (0,0)$, which is indeed the global solution to the problem.

The \pof formalizes the class of optimization problems where a similar approach is possible.
We also introduce our \pom to compute the optimal partition set index~$x^*$, since the problem of minimizing~$\Phi$ is not trivial in general.
For completeness, let us also introduce another function that the \pof relies on, the \textit{index function}~$\chi$, defined for all~$y \in \Ybb$ as the index~$x \in \Xbb$ of the unique partition set~$\Ybb(x)$ containing~$y$.
In the example above,~$\chi(y) \defequal y_1$ for all~$y \defequal (y_1,y_2) \in \Ybb$, since~$y \in \Ybb(x)$ only when~$x = y_1$.

\subsection{Motivation for this work}
\label{section:intro/motivation}

This work introduces the \pof as a new optimization framework that leverages inherent structure of some optimization problems to describe their potential solutions, and the \pom to exploit the \pof in practice.
Several applications may enjoy the \pof and be solved efficiently through the \pom.

A first broad field of applications is optimal control problems~\cite{Vinter10Controle} with discontinuous Mayer cost.
An example consists in optimizing the flight of a parachute so that it lands on a target made of concentric rings with different costs, while minimizing overall efforts.
The \textit{Mayer cost} of the flight is the cost associated to the ring the parachute ends in, which is therefore discontinuous.
However, if we impose the landing point of the parachute, the Mayer cost becomes fixed, so the problem becomes smooth, and then solvable (at least numerically with standard methods).
Thus, we may partition the variables space (the full trajectory of the parachute plus the control to ensure it) according to the landing point.
This context was our initial motivation in developing the \pof and to consider~$\Ybb$ infinite-dimensional.
An example of discontinuous optimal control problem is provided in Section~\ref{section:example_optimal_control}.
It challenges standard methods but is analytically solved using the \pof.
A broader application of the \pof in discontinuous contexts of optimal control theory is given in the PhD thesis~\cite[Chapter~7]{BouchetPhD} (in French), and in our technical report~\cite{BoAuBo21PWMayerCost}.

Another rich field is that of composite greybox problems~\cite{LaMe24structureAware,LaMeZh21ManifoldComposite} where, as in Section~\ref{section:intro/illustrative_example}, the inputs of the blackbox components of the problem are low-dimensional.
Several examples of such problems are discussed in Section~\ref{section:applications}.

Finally, we have initiated applications to \textit{counterfactuals}~\cite{VeBoHoHiDiSh22CounterfactualReview,VAFoPaVi24CFOPT} for \textit{contextual optimization}~\cite{SaChDeFoFrVi25SurveyContextualCO}.
Computing such counterfactuals is a trending topic in machine learning, and its relation to the \pof is natural.
A contextual problem consists in minimizing an objective function parameterized by a \textit{context vector}, that is usually predicted from data by some machine learning model.
Given a context~$x$, its associated contextual solution~$y^*(x)$, and another element~$y \neq y^*(x)$, a \textit{counterfactual of~$x$ with respect to~$y$} is a context~$x' \neq x$ close to~$x$ and such that~$y^*(x') = y$.

Other possible applications may likely connect to the \pof, such as almost \textit{separable problems}~\cite{stefanov2021Separable} and problems that become smooth when a constraint linking the variables is added.

\subsection{Contribution and outline of this work}
\label{section:intro/contribution_outline}

Our contribution is fourfold.
First, we state the \pof, and we show that we may recover a solution to Problem~\eqref{problem:P_initial} from an optimal partition set index.
Second, we propose the \pom to exploit the \pof, using a class of \textit{derivative-free optimization} (\dfo) algorithms to compute an optimal partition set index.
Third, we substantiate our infinite-dimensional considerations by solving with the \pof a discontinuous optimal control problem that challenges usual techniques from the literature.
Fourth, we empirically assert the performance of the \pom through numerical experiments.
They ensure that our theoretical analysis is reflected in practice, and compare the \pom against two state-of-the-art \dfo solvers solving Problem~\eqref{problem:P_initial} directly.

This work is organized as follows.
A literature review is provided in Section~\ref{section:literature}.
Then, Section~\ref{section:results} addresses our first two goals: it formalizes the \pof and highlights its properties in Theorem~\ref{theorem:connections_solutions}, and it derives the \pom to solve Problem~\eqref{problem:P_initial} in Theorem~\ref{theorem:solving_P}.
Our third goal follows in Section~\ref{section:example_optimal_control}, which uses Theorem~\ref{theorem:connections_solutions} to solve analytically a discontinuous version of a classical optimal control problem on which standard methods cannot be applied directly.
Finally, our fourth goal is addressed in Sections~\ref{section:applications} and~\ref{section:applications_heavy}.
Section~\ref{section:applications} shows how the \pom solves some finite-dimensional composite optimization problems, and that numerical observations match our theoretical claims from Theorem~\ref{theorem:solving_P}.
Section~\ref{section:applications_heavy} considers more challenging problem instances (in higher dimension), and shows that solving these problems directly with two \dfo solvers from the literature is less efficient than our strategy.
Finally, Section~\ref{section:discussion} discusses our work and its possible extensions.

\clearpage \newpage
\section{Review of related literature}
\label{section:literature}

The \pof relates to the frameworks of \textit{parametric optimization}~\cite{St18Parametric} and \textit{bilevel optimization}~\cite{DempeZemkoho2020bilevel}.
Indeed, parametric optimization studies problems where the objective function and the feasible set depend on a tuneable parameter.
In our case, the parameterized problem is to minimize~$\varphi$ over~$\Ybb(x) \cap \Omega$, where the parameter of the problem is the partition set index~$x \in \Xbb$.
The local continuity of the parametric solution in discontinuous contexts is studied in~\cite{FeKaKr21Param}, but, to the best of our knowledge, no reference in parametric optimization claims results similar to ours.
Moreover, the terminology \textit{partitioned optimization} highlights that Problem~\eqref{problem:P_initial} is non-parametric and that the nature of the so-called parameter in the restricted problem is left to the user (as the partition set indices result from the chosen partition of~$\Ybb$).
Similarly, bilevel optimization describes problems in which some of the variables are fixed as a solution to a nested subproblem.
However, the difference in terminology highlights a contextual difference since Problem~\eqref{problem:P_initial} has only one level and, unlike in usual bilevel problems, there is no conflictual min-max structure.
Some references study bilevel nonsmooth problems~\cite{Mordukhovich2020} and apply \dfo algorithms~\cite{DiKuRiZe23DSMforBilevel}, but, to the best of our knowledge, none provides results similar to ours.

An important aspect of the \pof is its exploitation through the \pom.
The problem of computing an optimal partition set index is akin to a usual \textit{blackbox optimization}~\cite{AlAuGhKoLed2020} problem.
To solve it, we use \textit{\dfo algorithms with a \covering step}~\cite{AuBoBo24Covering} since, to our best knowledge, they are the only algorithms with guarantees to return a local solution in discontinuous contexts under mild assumptions.

Although the \pom leverages \dfo, we remark that the \pof itself targets problems that differ from usual use-cases from \dfo.
Indeed, the dimension of~$\Ybb$ may be high in Problem~\eqref{problem:P_initial}, while \dfo algorithms usually handle up to a few dozens of variables~\cite[Section~1.4]{AuHa2017}.
Moreover, the \pof requires some information on the problem to define a relevant partition of~$\Ybb$ and ensure the tractability of the minimizer of~$\varphi$ over each partition set, which is not available in the blackbox optimization setting that is popular in \dfo.
The \dfo contexts that match the most with the \pof are \textit{greybox optimization}~\cite{AuFr21GreyBoxTuto} and large-dimension concerns.
Greybox optimization considers problems where partial information on the problem is available, as in our example in Section~\ref{section:intro/illustrative_example}.
Regarding large-dimensional \dfo, the \pof has complementarity with the literature.
Indeed, usual techniques rely on dimension-selection~\cite{DzWi22} that assume \textit{active subspaces}~\cite{ConstantineBook15,WyBiWi21ActiveSubspaces}, or random directions~\cite{CaSc18,GrRoViZh2015,RoRo23ReducedSpaces}.
These techniques likely outperform the \pof when the objective function has a small sensitivity with respect to most of the variables but a high sensitivity with respect to the few others.
Yet, they likely fail on cases where the objective function mostly depends on a few complex combinations of the variables instead of all variables directly.
In contrast, this last context is favourable to the \pof.
As illustrated in Section~\ref{section:applications_heavy}, the \pom outperforms usual \dfo algorithms on four problems with this property.

As detailed in Section~\ref{section:applications}, the main application of the \pof that we consider in this work relates to greybox composite optimization problems.
Although the associated literature is rich,~\cite{LaMe24structureAware,LaMeZh21ManifoldComposite} claim that most papers assume specific structures for a composite function~$\varphi \defequal \widetilde{\varphi}\circ\sigma+\varepsilon$.
We observe that our approach has complementarity with those in~\cite{LaMe24structureAware,LaMeZh21ManifoldComposite}.
Indeed,~\cite{LaMe24structureAware} considers cases where~$\widetilde{\varphi}$ is cheap to evaluate, but with a possibly complex structure and few inputs, and~$\sigma$ and~$\varepsilon$ expensive with possibly many inputs (around~$65$ at most).
In contrast, we better tackle cases where~$\widetilde{\varphi}$ has a simple structure and numerous inputs but~$\sigma$ and~$\varepsilon$ have few inputs.
Moreover,~\cite{LaMeZh21ManifoldComposite} also considers~$\widetilde{\varphi}$ well-structured and cheap to evaluate, but unlike us, it assumes that~$\sigma$ and~$\varepsilon$ are smooth.

The terminology \textit{partitioned optimization} also describes the \dfo methods~\cite{JoMa21DIRECT,LiLuPi10}.
However, these methods consider a \textit{discrete} partition which does not simplify the restricted problem, as they partition a hyperrectangle into smaller ones and recursively partition the most promising ones.
Then, we believe that our \textit{continuous} and structure-simplifying partition fits in the same terminology without confusion.
The \pof may also be related to \textit{optimization on manifold}~\cite{boumal23Manifold}, which considers problems with a constraint inducing an explicit manifold on the variables space.
This happens in our case to the problem restricted to any partition set, when each partition set of~$\Ybb$ is a manifold.
In practice, when~$\Ybb$ has a finite dimension, it is likely that~$\Ybb$ is indeed partitioned into a continuum of manifolds.

\section{Formal framework and related theoretical results}
\label{section:results}

This section considers Problem~\eqref{problem:P_initial} in its general form (that is,~$\varphi$ is not limited to its expression from Section~\ref{section:intro/illustrative_example}).
Section~\ref{section:results/framework} formalizes the \pof and its associated properties, and
Section~\ref{section:results/algorithm} details the \pom that we use to exploit the \pof and highlights the resulting convergence analysis.
The proofs of our claims are left to Section~\ref{section:results/proofs}.

\textbf{Notation:}
We denote by~$\Rbbinf \defequal \Rbb \cup \{\pm\infty\}$,~$\Nbbinf \defequal \Nbb \cup \{+\infty\}$,~$\Rbb^* \defequal \Rbb \setminus \{0\}$ and~$\Nbbinf^* \defequal \Nbbinf \setminus \{0\}$.
Consider two normed spaces~$\Abb$ and~$\Bbb$.
A set~$A \subseteq \Abb$ is said to be \textit{ample}~\cite{AuBoBo24Covering} if~$A \subseteq \cl(\int(A))$ (where~$\cl$ and~$\int$ denote respectively the closure and interior operators), and to be a \textit{continuity set} of a function~$f: \Abb \to \Rbbinf$ if~$f(a)$ is finite for all~$a \in A$ and~$f_{|A}$ is continuous.
For all collections~$(A_i)_{i=1}^{N}$ of subsets of~$\Abb$, with~$N \in \Nbbinf^*$, their union is denoted by~$\partitioncup_{i=1}^{N} A_i$ when the sets are pairwise disjoint.
For all sequences~$(b^k)_{k \in \Nbb} \in \Bbb^\Nbb$, we denote by~$\acc((b^k)_{k \in \Nbb})$ the set of all accumulation points of~$(b^k)_{k \in \Nbb}$.
For all functions~$g: \Abb \to \Bbb$, all sets~$A \subseteq \Abb$ and all points~$a \in \Abb$, we denote by~$g(A) \defequal \{g(a): a \in A\}$ and by~$\ACC(g;a)$ the union of~$\acc((g(a^k))_{k \in \Nbb})$ over all sequences~$(a^k)_{k \in \Nbb} \in \Abb^\Nbb$ converging to~$a$.
For all functions~$f: \Abb \to \Rbbinf$ and all sets~$A \subseteq \Abb$, we denote by~$\barrier{f}{A}: \Abb \to \Rbbinf$ the \textit{extreme barrier function}~\cite{AuDe2006} mapping all~$a \in \Abb$ to~$f(a)$ if~$a \in A$ and to~$+\infty$ if~$a \notin A$.
In particular,~$\barrier{\varphi}{\Omega}$ denotes the extreme barrier function associated to~$\varphi$ and~$\Omega$ from Problem~\eqref{problem:P_initial}.

\subsection{Formal definition of the \pof and theoretical properties}
\label{section:results/framework}

This section states the \pof, given in Framework~\ref{framework:PO}, and its properties, enclosed in Theorem~\ref{theorem:connections_solutions}.

\begin{framework}[partitioned optimization framework (\pof)]
    \label{framework:PO}
    Let $\Xbb$ be a finite-dimensional normed space, and $\Ybb = \partitioncup_{x \in \Xbb} \Ybb(x)$ be a partition such that the \textit{subproblem associated to Problem~\eqref{problem:P_initial} at~$x$},
    \begin{equation}
        \label{problem:P_subproblem}
        \tag{$\Pbf_\texttt{sub}(x)$}
        \problemoptimoneline{\minimize}{y \in \Ybb(x)}{\varphi(y)}{y \in \Omega,}
    \end{equation}
    admits a nonempty set~$\Gamma(x) \subseteq \Ybb(x)$ of global solutions for all~$x \in X \defequal \{x \in \Xbb : \Ybb(x) \cap Y \cap \Omega \neq \emptyset\}$, that we name the \textit{set of effective partition set indices}.
    Define an \textit{oracle function}~$\gamma: X \to Y$, assumed tractable and satisfying~$\gamma(x) \in \Gamma(x)$ for all~$x \in X$.
    Define also the \textit{index function}~$\chi: \Ybb \to \Xbb$ that, for all~$y \in \Ybb$, maps~$y$ to the index~$x \in \Xbb$ of the unique partition set~$\Ybb(x)$ containing~$y$.
    Finally, define the problem of computing an optimal partition set index as the \textit{reformulation of Problem~\eqref{problem:P_initial}},
    \begin{equation}
        \label{problem:P_reformulated}
        \tag{$\Pbf_\texttt{ref}$}
        \problemoptimfree{\minimize}{x \in \Xbb}{\Phi(x),}
        \qquad \mbox{where} \qquad
        \fct{\Phi}{x}{\Xbb}{\left\{\begin{array}{ll}\varphi(\gamma(x)) & \mbox{if}~ x \in X, \\ +\infty & \mbox{otherwise}.\end{array}\right.}{\Rbbinf}
    \end{equation}
\end{framework}

Let us stress that ensuring the tractability of the oracle function~$\gamma$ is not within the scope of this work.
In practice, we consider that we possess an optimization method to tackle Subproblem~\eqref{problem:P_subproblem} for any partition set index~$x$.
Then~$\gamma$ consists in the function that, given any~$x$, calls this method and returns its output.
We state the \pof in Framework~\ref{framework:PO} by assuming that the oracle function~$\gamma$ computes, for all~$x$, exactly a global solution to Subproblem~\eqref{problem:P_subproblem}.
Our theoretical results are derived from this assumption.
However, we show in Sections~\ref{section:applications/dim2} and~\ref{section:applications_heavy/dim2_heavy} that the \pom remains efficient when~$\gamma(x)$ computes an approximation of a global solution to Subproblem~\eqref{problem:P_subproblem}.
The \pof may even retain theoretical properties when~$\gamma$ returns only an approximation of a local solution.
We announce in Section~\ref{section:discussion/perspectives} a future work that will study an alteration of Framework~\ref{framework:PO} allowing for this relaxation of the definition of~$\gamma$.
Similarly, how to efficiently partition~$\Ybb$ in general is not within our scope either.
We believe that finding a relevant partition is often simple, as illustrated in an infinite-dimensional problem in Section~\ref{section:example_optimal_control} and a class of finite-dimensional composite problems in Section~\ref{section:applications}.

To study the \pof, we define first a notion of \textit{generalized local solution}.
This notion, hinted in~\cite{AuBoBo24Covering} but not formally stated, is suited for problems with a possibly nonclosed feasible set and a possibly not lower semicontinuous objective function.
We introduce it in Definition~\ref{definition:generalized_local_solutions} and discuss it in Remark~\ref{remark:generalized_local_solutions}, for a generic problem since we will apply it to either Problem~\eqref{problem:P_initial} and Reformulation~\eqref{problem:P_reformulated}.

\begin{definition}[Generalized local solution]
    \label{definition:generalized_local_solutions}
    Consider a normed space~$\Abb$, and~$A \subseteq \Abb$, and~$f: \Abb \to \Rbbinf$.
    Any point~$a^* \in \cl(A)$ is said to be a generalized local solution to the problem
    \begin{equation*}
        \problemoptimoneline{\minimize}{a \in \Abb}{f(a)}{a \in A}
    \end{equation*}
    if there exists a neighbourhood~$\Acal^* \subseteq \Abb$ of~$a^*$ such that~$\inf \barrier{f}{A}(\Acal^*) = \inf \ACC(\barrier{f}{A}, a^*)$.
\end{definition}

We may now claim the main property of the \pof, in Theorem~\ref{theorem:connections_solutions}: under Assumption~\ref{assumption:pof_minimal}, solutions to Reformulation~\eqref{problem:P_reformulated} (either global, local, or generalized) yield solutions to Problem~\eqref{problem:P_initial} (of same nature).
We discuss Assumption~\ref{assumption:pof_minimal} in Remark~\ref{remark:discussion_assumption_minimal}, and we prove Theorem~\ref{theorem:connections_solutions} in Section~\ref{section:results/proofs}.

\begin{assumption}
    \label{assumption:pof_minimal}
    In Framework~\ref{framework:PO},~$\chi$ is continuous; and~$X$ admits a partition~$X = \partitioncup_{i=1}^{N} X_i$, with~$N \in \Nbbinf^*$, such that for all~$x \in \cl(X)$ there exists a neighbourhood~$\Xcal \subseteq \Xbb$ of~$x$ such that~$\{i \in \llb1,N\rrb: X_i \cap \Xcal \neq \emptyset\}$ is finite, and~$\gamma_{|X_i}$ is uniformly continuous for all~$i \in \llb1,N\rrb$.
\end{assumption}

\begin{theorem}
    \label{theorem:connections_solutions}
    In Framework~\ref{framework:PO} under Assumption~\ref{assumption:pof_minimal}, the next three statements hold:
    \begin{enumerate}[label=\alph*),nosep,topsep=-\parskip]
        \item
            \label{theorem:connections_solutions/global}
            for all global solutions~$x^*$ to Reformulation~\eqref{problem:P_reformulated},~$\gamma(x^*)$ is a global solution to Problem~\eqref{problem:P_initial};
        \item
            \label{theorem:connections_solutions/local}
            for all local solutions~$x^* \in X$ to Reformulation~\eqref{problem:P_reformulated},~$\gamma(x^*)$ is a local solution to Problem~\eqref{problem:P_initial};
        \item
            \label{theorem:connections_solutions/generalized}
            for all generalized local solutions~$x^* \in \cl(X)$ to Reformulation~\eqref{problem:P_reformulated}, the set~$\ACC(\gamma;x^*)$ contains a generalized local solution to Problem~\eqref{problem:P_initial}.
    \end{enumerate}
\end{theorem}

\begin{remark}
    \label{remark:generalized_local_solutions}
    Let us illustrate Definition~\ref{definition:generalized_local_solutions} via two problems (with~$\Abb \defequal \Rbb$) it is suited for.
    First, let~$A \defequal \Rbb$ and~$f: a \mapsto a+1$ if~$a \geq 0$ and~$\textabs{a}$ otherwise.
    Second, let~$A \defequal \Rbb_+^*$ and~$f: a \mapsto a$.
    In both cases, the problem has no solution, but the point~$a^* \defequal 0$ has a remarkable property: the infimum of~$\barrier{f}{A}$ is found as the limit of~$(f(a^k))_{k \in \Nbb}$ for some well-chosen sequence~$(a^k)_{k \in \Nbb}$ converging to~$a^*$.
    Definition~\ref{definition:generalized_local_solutions} captures this property, and names~$a^*$ a \textit{generalized local solution} (in these examples, it is a global one).
\end{remark}

\begin{remark}
    \label{remark:discussion_assumption_minimal}
    Assumption~\ref{assumption:pof_minimal} is likely nonstringent in practice.
    First, the index function~$\chi$ is continuous when the partition set~$\Ybb(x)$ varies continuously with its index~$x \in \Xbb$.
    This is controllable since the partition of~$\Ybb$ is user-defined, and it holds when, e.g., the partition of~$\Ybb$ is made by fixing some variables as in Sections~\ref{section:intro/illustrative_example} and~\ref{section:example_optimal_control}.
    Second, the partition of~$X$ allows~$\gamma$ to be piecewise uniformly continuous.
    This rejects only instances where~$\gamma$ is locally unbounded, and is weaker than the Lipschitz continuity of~$\gamma$ that is commonly assumed in parametric optimization and bilevel optimization~\cite{DempeZemkoho2020bilevel,St18Parametric}.
\end{remark}

\subsection{The \pom and its related convergence certifications}
\label{section:results/algorithm}

We now discuss our method to exploit the \pof.
It consists in solving Reformulation~\eqref{problem:P_reformulated} using a \dfo algorithm, and apply Theorem~\ref{theorem:connections_solutions} to the (theoretically exact) solution it returns.
Accordingly, we name this method \textit{derivative-free partitioned optimization method} (\pom).
Before going through this section, let us stress that we introduce the \pom for cases where solving Reformulation~\eqref{problem:P_reformulated} is challenging.
The \pom is not needed when a solution to Reformulation~\eqref{problem:P_reformulated} is available via specific techniques, since Theorem~\ref{theorem:connections_solutions} is compatible with any solution.
Section~\ref{section:example_optimal_control} shows such an example, where Reformulation~\eqref{problem:P_reformulated} is solved analytically.
We also highlight that in practice, a \dfo algorithm returns only an approximation of a solution.
Yet, our numerical tests in Sections~\ref{section:applications} and~\ref{section:applications_heavy} prove that this discrepancy with our theoretical analysis does not alter the performance of the \pom.

Reformulation~\eqref{problem:P_reformulated} is a usual blackbox optimization problem: its variables space~$\Xbb$ likely has a low dimension (even when~$\Ybb$ has infinite dimension), and, for all~$x \in \Xbb$, the objective value~$\Phi(x)$ is computed on-the-fly by calling the oracle function~$\gamma$ at~$x$ and evaluating~$\Phi(x) \defequal \varphi(\gamma(x))$ (or~$\Phi(x) \defequal +\infty$ if~$\gamma(x)$ is not defined).
Then \dfo algorithms are suited for Reformulation~\eqref{problem:P_reformulated}, in particular those with a \textit{\covering step}~\cite{AuBoBo24Covering} which ensures convergence to local solutions.
Let us formalize the \covering operator in Definition~\ref{definition:covering_operator}.
We provide only one concise version, but~\cite{AuBoBo24Covering} discusses more tractable variants.

\begin{definition}
    \label{definition:covering_operator}
    The \textit{\covering operator} is the set-valued function~$\covering: \Xbb \times 2^{\Xbb} \to 2^{\Xbb}$ defined, for all~$(x,\Hcal) \in \Xbb \times 2^{\Xbb}$, by~$\covering(x,\Hcal) \defequal \argmax_{\textnorm{d} \leq 1} \dist(x+d,\Hcal)$, where~$\dist(x,\Hcal) \defequal \inf_{h \in \Hcal} \textnorm{x-h}$ for all~$\Hcal \neq \emptyset$ and, by convention,~$\dist(x,\emptyset) \defequal +\infty$.
\end{definition}

Given a point~$x \in \Xbb$ and a set~$\Hcal \subseteq \Xbb$ acting as a history of past points, the \covering operator seeks for a direction~$d$ that yields a new point~$x+d$ as far as possible from~$\Hcal$ within a ball of unitary radius centred at~$x$.
The \covering operator may be embedded into any \dfo algorithm as an additional \textit{\covering step} that applies, at all iterations~$k \in \Nbb$, the \covering operator at the current incumbent solution~$x^k$ and the current history~$\Hcal^k$ of past trial points.
We formalize this framework as Algorithm~\ref{algo:covering_framework}.
This ensures that a dense set of points is evaluated near all accumulation points~$x^* \in \acc((x^k)_{k \in \Nbb})$, which asserts their local optimality.
We state this claim in Proposition~\ref{proposition:cDSM_applied_to_Phi}, a rewording of~\cite[Theorem~2]{AuBoBo24Covering}.

\begin{algorithm}[!ht]
    \caption{Framework~\cite[Algorithm~2]{AuBoBo24Covering} to fit a \covering operator into any iterative algorithm.}
    \label{algo:covering_framework}
    \begin{algorithmic}
    \IState{1} \textbf{Initialization}:
        \IState{2}
            select~$x^0 \in X$ the initial incumbent solution;
            set~$\Hcal^0 \defequal \{x^0\}$ the initial trial points history;

        \IState{2}
            select an iterative optimization Algorithm~$\Acal$;

    \For{$k \in \Nbb$}
        \IState{2} \covering \textbf{step}:
            \IState{3}
                set~$\DcalC^k \defequal \covering(x^k,\Hcal^k)$;
                set~$\TcalC^k \defequal \pointplusset{x^k}{\DcalC^k}$;
                select~$\tC^k \in \argmin \Phi(\TcalC^k)$;

            \IState{3}
                if~$\Phi(\tC^k) < \Phi(x^k)$, then set~$t^k \defequal \tC^k$ and~$\TcalA^k \defequal \emptyset$ and go to the \update step;

        \IState{2} \textbf{core algorithm} \textbf{step}:
            \IState{3}
                select~$\TcalA^k \subseteq \Xbb$ a finite set of trial points according to the iterative scheme of Algorithm $\Acal$;

            \IState{3}
                select~$\tA^k \in \argmin \Phi(\TcalA^k)$;
                if~$\Phi(\tA^k) < \Phi(x^k)$, then set~$t^k \defequal \tA^k$, otherwise set~$t^k \defequal x^k$;

        \IState{2} \update \textbf{step}:
            \IState{3}
                set~$x^{k+1} \defequal t^k$;
                set~$\Hcal^{k+1} \defequal \Hcal^k \cup \TcalC^k \cup \TcalA^k$.

    \EndFor
    \end{algorithmic}
\end{algorithm}

\begin{proposition}[convergence analysis of Algorithm~\ref{algo:covering_framework}]
    \label{proposition:cDSM_applied_to_Phi}
    Assume that~$\Phi$ is bounded below with bounded sublevel sets and that~$X$ admits a partition into~$N \in \Nbbinf^*$ ample continuity sets of~$\Phi$.
    Then Algorithm~\ref{algo:covering_framework} solving Reformulation~\eqref{problem:P_reformulated} generates a sequence~$(x^k)_{k \in \Nbb}$ such that all~$x^* \in \acc((x^k)_{k \in \Nbb}) \neq \emptyset$ are generalized local solutions to Reformulation~\eqref{problem:P_reformulated}.
    Moreover, for all~$x^* \in \acc((x^k)_{k \in \Nbb})$, if~$x^* \in X$ and~$\Phi$ is lower semicontinuous at~$x^*$, then~$x^*$ is a local solution to Reformulation~\eqref{problem:P_reformulated}.
\end{proposition}

We now state our second result.
Theorem~\ref{theorem:solving_P} formalizes the \pom to solve Problem~\eqref{problem:P_initial}, by solving Reformulation~\eqref{problem:P_reformulated} (exactly) using any instance of Algorithm~\ref{algo:covering_framework} and applying Theorem~\ref{theorem:connections_solutions} to the returned solution.
This result relies on the additional Assumption~\ref{assumption:pof_additional} that slightly refines Assumption~\ref{assumption:pof_minimal}, as discussed in Remark~\ref{remark:discussion_assumption_additional}.
The proof of Theorem~\ref{theorem:solving_P} is given in Section~\ref{section:results/proofs}.

\begin{assumption}
    \label{assumption:pof_additional}
    In Framework~\ref{framework:PO};~$\chi(\Ycal)$ is bounded for all~$\Ycal \subseteq \Ybb$ bounded; and~$\varphi_{|\gamma(X)}$ is bounded below and has its sublevel sets bounded; and for all~$i \in \llb1,N\rrb$, the partition set~$X_i$ from Assumption~\ref{assumption:pof_minimal} is ample and such that~$\varphi_{|\gamma(X_i)}$ is continuous.
\end{assumption}

\begin{theorem}
    \label{theorem:solving_P}
    In Framework~\ref{framework:PO} under Assumptions~\ref{assumption:pof_minimal} and~\ref{assumption:pof_additional}, Algorithm~\ref{algo:covering_framework} solving Reformulation~\eqref{problem:P_reformulated} generates a sequence~$(x^k)_{k \in \Nbb}$ such that, for all~$x^* \in \acc((x^k)_{k \in \Nbb}) \neq \emptyset$, the set~$\ACC(\gamma;x^*)$ contains a generalized local solution to Problem~\eqref{problem:P_initial}.
    Moreover, for all~$x^* \in \acc((x^k)_{k \in \Nbb})$, if~$x^* \in X$ and~$\varphi$ is lower semicontinuous at all~$y \in \ACC(\gamma;x^*)$ and~$\ACC(\gamma;x^*) \subseteq \Omega$, then~$\gamma(x^*)$ is a local solution to Problem~\eqref{problem:P_initial}.
\end{theorem}

As Theorem~\ref{theorem:solving_P} shows, Algorithm~\ref{algo:covering_framework} solves only Reformulation~\eqref{problem:P_reformulated}, a problem it is well suited for.
Moreover, most of the literature on \dfo is applicable to Reformulation~\eqref{problem:P_reformulated}.
Examples are discussed in Remark~\ref{remark:grey_box_structure}.
Note also that the \pom is oblivious to whether the oracle function~$\gamma$ returns exactly global solutions as requested in Framework~\ref{framework:PO}.
Theorem~\ref{theorem:solving_P} describes only the case where Framework~\ref{framework:PO} is valid, but Sections~\ref{section:applications} and~\ref{section:applications_heavy} show that the \pom remains efficient when~$\gamma$ returns an approximation of a local solution.
We also foresee in Section~\ref{section:discussion/perspectives} a future work that will formally re-define the \pof with a less stringent definition of~$\gamma$.
Finally, remark that Theorem~\ref{theorem:solving_P} requires for Algorithm~\ref{algo:covering_framework} to run \textit{ad infinitum}.
This is a usual drawback in \dfo that, to our best knowledge, cannot be avoided.
In practice, we interrupt Algorithm~\ref{algo:covering_framework} according to some stopping criterion, at some iteration index~$\kappa \in \Nbb$, and we return the incumbent solution~$x^\kappa$ as an approximation of a solution to Reformulation~\eqref{problem:P_reformulated}.
Then the \pom returns~$\gamma(x^\kappa)$ as an approximation of a solution to Problem~\eqref{problem:P_initial}.

We conclude this section with the practical instance of Algorithm~\ref{algo:covering_framework} that we use in our numerical experiments in Sections~\ref{section:applications} and~\ref{section:applications_heavy}.
We instantiate~$\Acal$ as the popular \textit{direct search method} (\dsm)~\cite{AuDe2006,AuHa2017}.
The resulting \textit{\covering direct search method} (\cdsm)~\cite[Algorithm~1]{AuBoBo24Covering} is therefore a fit of the \covering step into the \dsm.
We formalize the \cdsm in Algorithm~\ref{algo:cdsm}, and discuss its parameters in Remark~\ref{remark:algo_parameters}.

\begin{algorithm}[!ht]
    \caption{\cdsm~\cite[Algorithm~1]{AuBoBo24Covering} instance of Algorithm~\ref{algo:covering_framework} considered in Sections~\ref{section:applications} and~\ref{section:applications_heavy}.}
    \label{algo:cdsm}
    \begin{algorithmic}
    \IState{1} \textbf{Initialization}:
        \IState{2}
            select~$(x^0,\delta^0) \in X \times \Rbb_+^*$ the initial (incumbent solution, poll radius) couple;
        \IState{2}
            set~$\Hcal^0 \defequal \{x^0\}$ the initial trial points history;
        \IState{2}
            select~$(\lambda,\upsilon) \in {]}0,1{[} \times [1,+\infty{[}$ the poll radius shrinking and expanding parameters;
    \For{$k \in \Nbb$}
        \IState{2} \covering \textbf{step}:
            \IState{3}
                set~$\DcalC^k \defequal \covering(x^k,\Hcal^k)$;
                set~$\TcalC^k \defequal \pointplusset{x^k}{\DcalC^k}$;
                select~$\tC^k \in \argmin \Phi(\TcalC^k)$;
            \IState{3}
                if~$\Phi(\tC^k) < \Phi(x^k)$, then set~$t^k \defequal \tC^k$ and~$\TcalS^k = \TcalP^k\defequal \emptyset$ and go to the \update step;
        \IState{2} \search \textbf{step}:
            \IState{3}
                select~$\DcalS^k \subseteq \Xbb$ empty or finite;
                if~$\TcalS^k \defequal \pointplusset{x^k}{\DcalS^k}$ is nonempty, then select~$\tS^k \in \argmin \Phi(\TcalS^k)$;
            \IState{3}
                if moreover~$\Phi(\tS^k) < \Phi(x^k)$, then set~$t^k \defequal \tS^k$ and~$\TcalP^k \defequal \emptyset$ and go to the \update step;
        \IState{2} \poll \textbf{step}:
            \IState{3}
                select~$\DcalP^k \subseteq \Xbb$ a positive basis of length~$\delta^k$;
                set~$\TcalP^k \defequal \pointplusset{x^k}{\DcalP^k}$;
                select~$\tP^k \in \argmin \Phi(\TcalP^k)$;
            \IState{3}
                if~$\Phi(\tP^k) < \Phi(x^k)$, then set~$t^k \defequal \tP^k$, otherwise set~$t^k \defequal x^k$;
        \IState{2} \update \textbf{step}:
            \IState{3}
                set~$x^{k+1} \defequal t^k$;
                set~$\delta^{k+1} \defequal \upsilon\delta^k$ if~$t^k \neq x^k$ or~$\delta^{k+1} \defequal \lambda\delta^k$ if~$t^k = x^k$;
                set~$\Hcal^{k+1} \defequal \Hcal^k \cup \TcalC^k \cup \TcalS^k \cup \TcalP^k$.
    \EndFor
    \end{algorithmic}
\end{algorithm}

\begin{remark}
    \label{remark:discussion_assumption_additional}
    Assumption~\ref{assumption:pof_additional} builds additional mild requirements to Assumption~\ref{assumption:pof_minimal}.
    First, in practice, a partition of~$\Ybb$ that makes the index function~$\chi$ continuous usually yields its additional requirements as well.
    For example, it is the case in all problems from the class we study in Sections~\ref{section:applications} and~\ref{section:applications_heavy}.
    Second, the requirements on~$\varphi_{|\gamma(X)}$ are standard, and moreover they are stated on~$\varphi_{|\gamma(X)}$ instead of~$\varphi$ directly (so even contexts where~$\varphi$ has unbounded sublevel sets may be compatible).
    Finally, the additional requirements on the partition sets~$(X_i)_{i=1}^{N}$ inherited from Assumption~\ref{assumption:pof_minimal} are a minimal safeguard against pathological behaviours~\cite{AuBoBo24Covering} that are unlikely to appear in practice.
\end{remark}

\begin{remark}
    \label{remark:grey_box_structure}
    The \pom is compatible with the notion of \textit{surrogate}~\cite{BoDeFrSeToTr99a} from \dfo.
    When~$\gamma(x)$ is obtained, for all~$x \in X$, by a numerical method solving~\eqref{problem:P_subproblem}, a surrogate is obtained by relaxing the precision of the method.
    It is also possible to order trial points according to a sensitivity analysis of Subproblem~\eqref{problem:P_subproblem} with respect to~$x$ via the generalized implicit function theorem~\cite{Ro91Implicit}.
    This echoes concerns in parametric optimization about how the parameter influences the solution~\cite{FeKaKr21Param,St18Parametric}.
\end{remark}

\begin{remark}
    \label{remark:algo_parameters}
    Sections~\ref{section:applications} and~\ref{section:applications_heavy} instantiate Algorithm~\ref{algo:cdsm} as follows.
    The point~$x^0$ depends on the instance, while we set~$\delta^0 \defequal 1$.
    We choose~$(\lambda,\upsilon) \defequal (\frac{1}{2},1)$ in Sections~\ref{section:applications/monovariable},~\ref{section:applications/radial},~\ref{section:applications/nonlinear},~\ref{section:applications_heavy/monovariable_heavy} and~\ref{section:applications_heavy/radial_heavy}, and~$(\lambda,\upsilon) \defequal (\frac{1}{2},2)$ in Section~\ref{section:applications_heavy/nonlinear_heavy}, and~$(\lambda,\upsilon) \defequal (\frac{3}{4},2)$ in Sections~\ref{section:applications/dim2} and~\ref{section:applications_heavy/dim2_heavy}.
    The \search step is skipped ($\DcalS^k \defequal \emptyset$) and the \poll step selects random orthogonal positive bases.
    The algorithm iterates until~$\delta^k < 1E^{-10}$.
\end{remark}

\subsection{Proofs of Theorems~\ref{theorem:connections_solutions} and~\ref{theorem:solving_P}}
\label{section:results/proofs}

First, Section~\ref{section:results/proof_theorems/lemmas} states some lemmas related to the \pof and Assumption~\ref{assumption:pof_minimal}.
Second, Section~\ref{section:results/proof_theorems/connections_solutions} proves Theorem~\ref{theorem:connections_solutions}.
Third, Section~\ref{section:results/proof/application_theorem} proves that the convergence analysis of Algorithm~\ref{algo:covering_framework} from Proposition~\ref{proposition:cDSM_applied_to_Phi} is applicable when Assumption~\ref{assumption:pof_additional} holds.
Finally, Section~\ref{section:results/proof/solving_P} proves Theorem~\ref{theorem:solving_P}.

\clearpage\newpage
\subsubsection{Technical lemmas on the problem, the framework and its minimal assumption}
\label{section:results/proof_theorems/lemmas}

This section formalizes a collection of lemmas about Problem~\eqref{problem:P_initial}, the \pof, and Assumption~\ref{assumption:pof_minimal}.

First, Lemma~\ref{lemma:inf_neighbourhood_geq_inf_ACC} is related to~$\barrier{\varphi}{\Omega}$, the extreme barrier variant of~$\varphi$ with respect to~$\Omega$, and~$\ACC(\barrier{\varphi}{\Omega};y)$, the set of accumulation values of~$\barrier{\varphi}{\Omega}$ at any point~$y \in \Ybb$.
This lemma states that the infimum of~$\barrier{\varphi}{\Omega}$ over any neighbourhood of any~$y \in \Ybb$ is below the lowest accumulation value of~$\barrier{\varphi}{\Omega}$ at~$y$.

\begin{lemma}
    \label{lemma:inf_neighbourhood_geq_inf_ACC}
    For all~$y \in \Ybb$ and all~$\Ycal \subseteq \Ybb$ neighbourhood of~$y$, it holds that~$\inf \barrier{\varphi}{\Omega}(\Ycal) \leq \inf \ACC(\barrier{\varphi}{\Omega};y)$.
\end{lemma}
\begin{proof}
    Let~$y \in \Ybb$ and let~$\Ycal \subseteq \Ybb$ be a neighbourhood of~$y$.
    Let~$z \in \ACC(\barrier{\varphi}{\Omega};y)$.
    Let~$(y^k)_{k \in \Nbb} \in \Ycal^\Nbb$ converging to~$y$ with~$(\barrier{\varphi}{\Omega}(y^k))_{k \in \Nbb}$ converging to~$z$.
    Then,~$\inf \barrier{\varphi}{\Omega}(\Ycal) \leq \barrier{\varphi}{\Omega}(y^k)$ holds for all~$k \in \Nbb$, so at the limit,~$\inf \barrier{\varphi}{\Omega}(\Ycal) \leq z$ follows.
    This proves the claim since~$z$ was arbitrarily chosen in~$\ACC(\barrier{\varphi}{\Omega};y)$.
\end{proof}

Second, Lemma~\ref{lemma:chi_gamma_reciprocal} states a basic property of the \pof: the oracle function~$\gamma$ and the index function~$\chi$ induce reciprocal bijections between~$X$ and~$\gamma(X)$.

\begin{lemma}
    \label{lemma:chi_gamma_reciprocal}
    In Framework~\ref{framework:PO},~$\gamma$ is a bijection from~$X$ to~$\gamma(X)$ and its reciprocal is the restriction~$\chi_{|\gamma(X)}$.
\end{lemma}
\begin{proof}
    For all pairs~$(x_1,x_2) \in X^2$ with~$x_1 \neq x_2$, it holds that~$\gamma(x_1) \in \Ybb(x_1)$ and~$\gamma(x_2) \in \Ybb(x_2)$ by construction of~$\gamma$, and~$\Ybb(x_1) \cap \Ybb(x_2) = \emptyset$ by definition of the partition.
    It follows that~$\gamma(x_1) \neq \gamma(x_2)$.
    Hence,~$\gamma$ is an injection from~$X$ to~$\gamma(X)$, and thus a bijection since the surjectivity is clear.
    Also, for all~$x \in X$, the inclusion~$\gamma(x) \in \Ybb(x)$ raises~$\chi(\gamma(x)) = x$ by definition of~$\chi$.
    Finally, for all~$y \in \gamma(X)$, there exists a unique~$x \in X$ satisfying~$y = \gamma(x)$, so~$\chi(y) = \chi(\gamma(x)) = x$ and then~$\gamma(\chi(y)) = \gamma(x) = y$.
    This proves the mutual reciprocity of~$\gamma$ and~$\chi_{|\gamma(X)}$.
\end{proof}

Third, Lemma~\ref{lemma:chi(neighbourhood)_neighbourhood} proves that the union of all partition sets of~$\Ybb$ associated to indices in a neighbourhood of any~$x \in \Xbb$ forms a neighbourhood of all~$y \in \Ybb(x)$.
We use it, in our proof of Theorem~\ref{theorem:connections_solutions}, to map neighbourhoods of solutions to Reformulation~\eqref{problem:P_reformulated} to neighbourhoods of solutions to Problem~\eqref{problem:P_initial}.

\begin{lemma}
    \label{lemma:chi(neighbourhood)_neighbourhood}
    Under Assumption~\ref{assumption:pof_minimal}, for all~$x \in \Xbb$ and all~$\Xcal \subseteq \Xbb$ neighbourhood of~$x$, the set~$\Ybb(\Xcal)$ is a neighbourhood of all~$y \in \Ybb(x)$.
\end{lemma}
\begin{proof}
    Consider that Assumption~\ref{assumption:pof_minimal} holds.
    Let~$x \in \Xbb$, and~$\Xcal \subseteq \Xbb$ be a neighbourhood of~$x$.
    For all~$y \in \Ybb(x)$, it holds that~$\chi(y) = x \in \Xcal$ by definition of~$\chi$, and by continuity of~$\chi$ there exists an open set~$\Ycal \subseteq \Ybb$ containing~$y$ and such that~$\chi(\Ycal) \subseteq \Xcal$, so~$y \in \Ycal \subseteq \Ybb(\Xcal)$.
    This proves the desired claim.
\end{proof}

Then, Lemma~\ref{lemma:ACC(gamma;x)_nonempty_finite} highlights several key properties resulting from the completeness of~$\Ybb$ and the uniform continuity of~$\gamma$ on each partition set of~$X$ and from the finite number of partition sets of~$X$ near any~$x \in \cl(X)$, required by Assumption~\ref{assumption:pof_minimal}.
We prove that these two properties imply that~$\ACC(\gamma;x)$ has finitely many elements, which all belong to~$\Ybb(x)$, and that any sequence in~$X$ converging to~$x$ has its accumulation values through~$\gamma$ included in~$\ACC(\gamma;x)$.
In our proof of Theorem~\ref{theorem:connections_solutions}, this lemma allows us to handle the discontinuities of~$\gamma$ at any solution~$x^*$ to Reformulation~\eqref{problem:P_reformulated} identified by Algorithm~\ref{algo:covering_framework} as an accumulation value of the sequence~$(x^k)_{k \in \Nbb}$ it generates.

\begin{lemma}
    \label{lemma:ACC(gamma;x)_nonempty_finite}
    Under Assumption~\ref{assumption:pof_minimal}, for all~$x \in \cl(X)$, the set~$\ACC(\gamma;x)$ is a finite subset of~$\Ybb(x)$, and for all~$(x^k)_{k \in \Nbb} \in X^\Nbb$ converging to~$x$, it holds that~$\emptyset \neq \acc((\gamma(x^k))_{k \in \Nbb}) \subseteq \ACC(\gamma;x)$.
\end{lemma}
\begin{proof}
    Consider that Assumption~\ref{assumption:pof_minimal} holds.
    Let~$x \in \cl(X)$ and let~$\Xcal \subseteq \Xbb$ be the neighbourhood of~$x$ given by Assumption~\ref{assumption:pof_minimal}, so~$\Ical \defequal \{i \in \llb1,N\rrb : X_i \cap \Xcal \neq \emptyset\}$ is finite.
    For all~$i \in \Ical$, by completeness of~$\Ybb$ and uniform continuity of~$\gamma_{|X_i}$, the set~$\cup_{x^k \to x, (x^k)_{k \in \Nbb} \in X_i^\Nbb}~ \acc((\gamma(x^k))_{k \in \Nbb})$ is a singleton in~$\Ybb$, and its unique element denoted by~$y_i$ lies in~$\Ybb(x)$ by continuity of~$\chi$.
    Then,~$\ACC(\gamma;x) = \{y_i,~ i \in \Ical\}$ is a finite subset of~$\Ybb(x)$.
    Now, let~$(x^k)_{k \in \Nbb} \in (X \cap \Xcal)^\Nbb$ converging to~$x$ and, for all~$i \in \llb1,N\rrb$, define~$K_i \defequal \{k \in \Nbb: x^k \in X_i\}$.
    Then,~$\Ical_+ \defequal \{i \in \Ical: K_i$ is infinite$\}$ is nonempty since we have~$\Nbb = \partitioncup_{i \in \llb1,N\rrb} K_i = \partitioncup_{i \in \Ical} K_i$.
    We deduce that~$\emptyset \neq \acc((\gamma(x^k))_{k \in \Nbb}) = \{y_i,~ i \in \Ical_+\} \subseteq \{y_i,~ i \in \Ical\} = \ACC(\gamma;x)$, as claimed.
\end{proof}

Now, Lemma~\ref{lemma:equality_acc_phi_gamma} states a technical equality about how discontinuities of~$\gamma$ at any~$x \in \cl(X)$ influence the possible accumulation values of~$\varphi \circ \gamma$ at~$x$.
It is an intermediate result to prove Lemma~\ref{lemma:inf_ACC(Phi,x)_geq_inf_ACC(phi,y)} below.

\begin{lemma}
    \label{lemma:equality_acc_phi_gamma}
    Under Assumption~\ref{assumption:pof_minimal}, for all~$x \in \cl(X)$, it holds that
    \begin{equation*}
        {\bigcup}_{x^k \to x, (x^k)_{k \in \Nbb} \in X^\Nbb} ~ \acc((\varphi(\gamma(x^k))_{k \in \Nbb}))
        \quad = \quad
        {\bigcup}_{y \in \ACC(\gamma;x)} ~ {\bigcup}_{y^k \to y, (y^k)_{k \in \Nbb} \in \gamma(X)^\Nbb} ~ \acc((\varphi(y^k))_{k \in \Nbb}).
    \end{equation*}
\end{lemma}
\begin{proof}
    Consider that Assumption~\ref{assumption:pof_minimal} holds, and let~$x \in \cl(X)$.
    We first prove the direct inclusion.
    Let~$z \in \cup_{x^k \to x, (x^k)_{k \in \Nbb} \in X^\Nbb} ~ \acc((\varphi(\gamma(x^k))_{k \in \Nbb})$ and~$(x^k)_{k \in \Nbb} \in X^\Nbb$ converging to~$x$ with~$(\varphi(\gamma(x^k)))_{k \in \Nbb}$ converging to~$z$.
    By Lemma~\ref{lemma:ACC(gamma;x)_nonempty_finite},~$\acc((\gamma(x^k))_{k \in \Nbb}) \neq \emptyset$ so there exists~$K \subseteq \Nbb$ infinite such that~$(\gamma(x^k))_{k \in K}$ converges to a limit~$y \in \ACC(\gamma;x)$.
    By re-labelling this sequence by~$(y^k)_{k \in \Nbb} \defequal (\gamma(x^k))_{k \in K} \in \gamma(X)^\Nbb$, we observe that~$z \in \acc((\varphi(y^k))_{k \in \Nbb}) \subseteq \cup_{y \in \ACC(\gamma;x)} \cup_{y^k \to y, (y^k)_{k \in \Nbb} \in \gamma(X)^\Nbb} \acc((\varphi(y^k))_{k \in \Nbb})$ as claimed.
    We now prove the reciprocal inclusion.
    Let~$y \in \ACC(\gamma;x)$ and let~$z \in {\bigcup}_{y^k \to y, (y^k)_{k \in \Nbb} \in \gamma(X)^\Nbb} ~ \acc((\varphi(y^k))_{k \in \Nbb})$.
    There exists~$(y^k)_{k \in \Nbb} \in \gamma(X)^\Nbb$ converging to~$y$ with~$(\varphi(y^k))_{k \in \Nbb}$ converging to~$z$.
    By continuity of~$\chi$ and Lemma~\ref{lemma:ACC(gamma;x)_nonempty_finite},~$(x^k)_{k \in \Nbb} \defequal (\chi(y^k))_{k \in \Nbb} \in X^\Nbb$ converges to~$\chi(y) = x$.
    Then,~$(\varphi(\gamma(x^k)))_{k \in \Nbb} = (\varphi(y^k))_{k \in \Nbb}$ converges to~$z$ by Lemma~\ref{lemma:chi_gamma_reciprocal}, so~$z \in {\bigcup}_{x^k \to x, (x^k)_{k \in \Nbb} \in X^\Nbb} ~ \acc((\varphi(\gamma(x^k))_{k \in \Nbb})$ as claimed.
\end{proof}

Finally, Lemma~\ref{lemma:inf_ACC(Phi,x)_geq_inf_ACC(phi,y)} is a key result of our analysis to highlight which element of~$\ACC(\gamma;x^*)$ is a generalized local solution to Problem~\eqref{problem:P_initial} in our proof of Theorem~\refbis{theorem:connections_solutions}{generalized}.

\begin{lemma}
    \label{lemma:inf_ACC(Phi,x)_geq_inf_ACC(phi,y)}
    Under Assumption~\ref{assumption:pof_minimal},~$\inf \ACC(\Phi;x) \geq \min_{y \in \ACC(\gamma;x)} \inf \ACC(\barrier{\varphi}{\Omega};y)$ holds for all~$x \in \cl(X)$.
\end{lemma}
\begin{proof}
    Consider that Assumption~\ref{assumption:pof_minimal} holds and let~$x \in \cl(X)$.
    It holds that
    \begin{equation*}\begin{array}{rcllll}
        \ACC(\Phi_{|X};x)
            &     =     & \multicolumn{3}{l}{\ACC(\varphi \circ \gamma;x)} & \text{(definition of~$\Phi$)} \\
            &     =     & \multicolumn{3}{l}{{\bigcup}_{x^k \to x,~ (x^k)_{k \in \Nbb} \in X^\Nbb} \quad \acc((\varphi(\gamma(x^k)))_{k \in \Nbb})} & \text{(definition of~$\ACC$)} \\
            &     =     & {\bigcup}_{y \in \ACC(\gamma;x)} & {\bigcup}_{y^k \to y,~ (y^k)_{k \in \Nbb} \in \gamma(X)^\Nbb} & \acc((\varphi(y^k))_{k \in \Nbb}) & \text{(Lemma~\ref{lemma:equality_acc_phi_gamma})} \\
            & \subseteq & {\bigcup}_{y \in \ACC(\gamma;x)} & {\bigcup}_{y^k \to y,~ (y^k)_{k \in \Nbb} \in (Y \cap \Omega)^\Nbb} & \acc((\varphi(y^k))_{k \in \Nbb}) & (\gamma(X) \subset Y \cap \Omega) \\
            & =  & {\bigcup}_{y \in \ACC(\gamma;x)} & {\bigcup}_{y^k \to y,~ (y^k)_{k \in \Nbb} \in (Y \cap \Omega)^\Nbb} & \acc((\barrier{\varphi}{\Omega}(y^k))_{k \in \Nbb}) & (\varphi = \barrier{\varphi}{\Omega} ~\text{over}~ \Omega) \\
            & \subseteq & {\bigcup}_{y \in \ACC(\gamma;x)} & {\bigcup}_{y^k \to y,~ (y^k)_{k \in \Nbb} \in \Ybb^\Nbb} & \acc((\barrier{\varphi}{\Omega}(y^k))_{k \in \Nbb}) & (Y \cap \Omega \subseteq \Ybb) \\
            &     =     & {\bigcup}_{y \in \ACC(\gamma;x)} & \multicolumn{2}{l}{\ACC(\barrier{\varphi}{\Omega};y)} & \text{(definition of~$\ACC$)}.
    \end{array}\end{equation*}
    It follows that~$\inf \ACC(\Phi_{|X};x) \geq \inf_{y \in \ACC(\gamma;x)} \inf \ACC(\barrier{\varphi}{\Omega};y)$, and~$\inf \ACC(\Phi;x) = \inf \ACC(\Phi_{|X};x)$ since~$\Phi(x') = +\infty$ for all~$x' \notin X$.
    The proof is completed by finiteness of~$\ACC(\gamma;x)$ from Lemma~\ref{lemma:ACC(gamma;x)_nonempty_finite}, which ensures that~$\inf_{y \in \ACC(\gamma;x)} \inf \ACC(\barrier{\varphi}{\Omega};y) = \min_{y \in \ACC(\gamma;x)} \inf \ACC(\barrier{\varphi}{\Omega};y)$.
\end{proof}

\subsubsection{Proof of Theorem~\ref{theorem:connections_solutions}}
\label{section:results/proof_theorems/connections_solutions}

We now prove Theorem~\ref{theorem:connections_solutions}.
The first two parts verify that given a usual (global or local) solution~$x^*$ to Reformulation~\eqref{problem:P_reformulated}, the element~$y^* \defequal \gamma(x^*)$ is a usual (global or local) solution to Problem~\eqref{problem:P_initial}.
The global case is direct, while the local case uses Lemma~\ref{lemma:chi(neighbourhood)_neighbourhood} to find a neighbourhood of~$y^*$ compatible with the definition of local optimality.
The case where~$x^*$ is a generalized local solution to Reformulation~\eqref{problem:P_reformulated} works similarly to the usual local case, except that it involves one more step to define~$y^*$.
We use Lemma~\ref{lemma:inf_ACC(Phi,x)_geq_inf_ACC(phi,y)} to find~$y^* \in \ACC(\gamma;x^*)$ that is a generalized local solution to Problem~\eqref{problem:P_initial}.

\begin{proof}
    Consider Framework~\ref{framework:PO} under Assumption~\ref{assumption:pof_minimal}.

    Let~$x^*$ be a global solution to Reformulation~\eqref{problem:P_reformulated} (so~$x^* \in X$ so~$\gamma(x^*) \in \Omega$ exists).
    For all~$y \in \Omega$, it holds that~$\varphi(y) \geq \Phi(\chi(y)) \geq \Phi(x^*) = \varphi(\gamma(x^*))$.
    Thus,~$\gamma(x^*)$ is a global solution to Problem~\eqref{problem:P_initial}, which proves Theorem~\refbis{theorem:connections_solutions}{global}.

    Let~$x^* \in X$ (so~$\gamma(x^*) \in \Omega$ exists) be a local solution to Reformulation~\eqref{problem:P_reformulated}.
    Then there exists a neighbourhood~$\Xcal^* \subseteq \Xbb$ of~$x^*$ such that~$\Phi(x) \geq \Phi(x^*)$ for all~$x \in \Xcal^*$, and by Lemma~\ref{lemma:chi(neighbourhood)_neighbourhood} the set~$\Ybb(\Xcal^*)$ is a neighbourhood of~$\gamma(x^*)$, which moreover satisfies~$\chi(y) \in \Xcal^*$ for all~$y \in \Ybb(\Xcal^*)$.
    Thus, for all~$y \in \Ybb(\Xcal^*) \cap \Omega$, it follows that~$\varphi(y) \geq \Phi(\chi(y)) \geq \Phi(x^*) = \varphi(\gamma(x^*))$.
    Thus,~$\gamma(x^*)$ is a local solution to Problem~\eqref{problem:P_initial}, which proves Theorem~\refbis{theorem:connections_solutions}{local}.

    Let~$x^* \in \cl(X)$ be a generalized local solution to Reformulation~\eqref{problem:P_reformulated}.
    Then~$\ACC(\gamma;x^*)$ is finite nonempty by Lemma~\ref{lemma:ACC(gamma;x)_nonempty_finite}, so we may select an element~$y^* \in \argmin_{y \in \ACC(\gamma;x^*)} \inf \ACC(\barrier{\varphi}{\Omega};y)$.
    First, there exists a neighbourhood~$\Xcal^* \subseteq \Xbb$ of~$x^*$ such that~$\inf \ACC(\Phi;x^*) = \inf \Phi(\Xcal^*)$, and~$\Ybb(\Xcal^*)$ is a neighbourhood of~$y^*$ by Lemma~\ref{lemma:chi(neighbourhood)_neighbourhood}.
    Second,~$y^*$ satisfies~$\inf \ACC(\Phi;x^*) \geq \inf \ACC(\barrier{\varphi}{\Omega};y^*)$ by Lemma~\ref{lemma:inf_ACC(Phi,x)_geq_inf_ACC(phi,y)}.
    Third,~$\inf \barrier{\varphi}{\Omega}(\Ybb(\Xcal^*)) \geq \inf \Phi(\Xcal^*)$ holds since all~$y \in \Ybb(\Xcal^*)$ satisfy~$\barrier{\varphi}{\Omega}(y) \geq \Phi(\chi(y)) \geq \inf \Phi(\Xcal^*)$.
    We deduce that~$\inf \barrier{\varphi}{\Omega}(\Ybb(\Xcal^*)) \geq \inf \ACC(\barrier{\varphi}{\Omega};y^*)$, and then~$\inf \barrier{\varphi}{\Omega}(\Ybb(\Xcal^*)) = \inf \ACC(\barrier{\varphi}{\Omega};y^*)$ via Lemma~\ref{lemma:inf_neighbourhood_geq_inf_ACC}.
    Thus,~$y^*$ is a generalized local solution to Problem~\eqref{problem:P_initial}, which proves Theorem~\refbis{theorem:connections_solutions}{generalized}.
\end{proof}

\subsubsection{Proof of validity of Proposition~\ref{proposition:cDSM_applied_to_Phi} under Assumption~\ref{assumption:pof_additional}}
\label{section:results/proof/application_theorem}

This section shows that the existing convergence analysis of Algorithm~\ref{algo:covering_framework} (from~\cite[Theorem~2]{AuBoBo24Covering}, and recalled in Proposition~\ref{proposition:cDSM_applied_to_Phi} above) is applicable in the context of Reformulation~\eqref{problem:P_reformulated}.
First, Proposition~\ref{proposition:requirements_OK} shows that all assumptions involved in Proposition~\ref{proposition:cDSM_applied_to_Phi} are met in our context.

\begin{proposition}
    \label{proposition:requirements_OK}
    If Assumptions~\ref{assumption:pof_minimal} and~\ref{assumption:pof_additional} hold, then~$\Phi$ is bounded below and has its sublevel sets bounded, and~$X_i$ is an ample continuity set of~$\Phi$ for each~$i \in \llb1,N\rrb$.
\end{proposition}
\begin{proof}
    Consider that Assumptions~\ref{assumption:pof_minimal} and~\ref{assumption:pof_additional} hold.
    First,~$\inf \Phi(\Xbb) = \inf \Phi(X) = \inf \varphi(\gamma(X)) > \minusinfty$ holds by assumption, so~$\Phi$ is bounded below.
    Second, let~$\alpha \in \Rbb$.
    For all~$x \in \Xbb$, the inclusion that~$x$ lies in the sublevel set of~$\Phi$ of level~$\alpha$ reads as~$x \in \Phi^{-1}({]}\minusinfty,\alpha])$, which is equivalent to~$\Phi(x) \leq \alpha$, and then to~$\varphi(\gamma(x)) \leq \alpha$, and then to~$\gamma(x) \in \varphi_{|\gamma(X)}^{-1}(]\minusinfty,\alpha])$, and then to~$x \in \chi(\varphi_{|\gamma(X)}^{-1}(]\minusinfty,\alpha]))$ by Lemma~\ref{lemma:chi_gamma_reciprocal}.
    Thus,~$\Phi^{-1}({]}\minusinfty,\alpha])$ is bounded, since by assumption~$\varphi_{|\gamma(X)}^{-1}(]\minusinfty,\alpha])$ is bounded and~$\chi(\Ycal)$ is bounded for all~$\Ycal \subseteq \Ybb$ bounded.
    Last, Assumption~\ref{assumption:pof_additional} yields~$X = \partitioncup_{i=1}^{N} X_i$ such that, for all~$i \in \llb1,N\rrb$,~$X_i$ is ample and~$\Phi_{|X_i} = (\varphi_{|\gamma(X_i)}) \circ (\gamma_{|X_i})$ is continuous as a composition of continuous functions.
\end{proof}

Then, Proposition~\ref{proposition:semicontinuity_Phi} highlights some conditions to ensure the requirement of lower semicontinuity of~$\Phi$ (involved in the last part of Proposition~\ref{proposition:cDSM_applied_to_Phi}).

\begin{proposition}
    \label{proposition:semicontinuity_Phi}
    Under Assumption~\ref{assumption:pof_minimal}, for all~$x \in X$, if~$\varphi$ is lower semicontinuous at all~$y \in \ACC(\gamma;x)$ and~$\ACC(\gamma;x) \subseteq \Omega$, then~$\Phi$ is lower semicontinuous at~$x$.
\end{proposition}
\begin{proof}
    Consider that Assumption~\ref{assumption:pof_minimal} holds, let~$x \in X$ and assume that~$\varphi$ is lower semicontinuous at all~$y \in \ACC(\gamma;x) \subseteq \Omega$.
    Lemma~\ref{lemma:inf_ACC(Phi,x)_geq_inf_ACC(phi,y)} claims that~$\inf \ACC(\Phi;x) \geq \min_{y \in \ACC(\gamma;x)} \inf \ACC(\barrier{\varphi}{\Omega};y)$.
    Moreover, for all~$y \in \ACC(\gamma;x)$, it holds that~$\inf \ACC(\barrier{\varphi}{\Omega};y) \geq \inf \ACC(\varphi;y) \geq \varphi(y)$ by lower semicontinuity of~$\varphi$ at~$y$, and that~$\varphi(y) \geq \varphi(\gamma(\chi(y))) = \Phi(x)$ since~$y \in \Omega$ and~$\chi(y) = x \in X$ so~$\gamma(\chi(y))$ exists.
    Then, we deduce that~$\inf \ACC(\Phi;x) \geq \Phi(x)$, which proves the claim.
\end{proof}

\subsubsection{Proof of Theorem~\ref{theorem:solving_P}}
\label{section:results/proof/solving_P}

Finally, we prove Theorem~\ref{theorem:solving_P}.
We apply Proposition~\ref{proposition:cDSM_applied_to_Phi} to gather a solution~$x^*$ to Reformulation~\eqref{problem:P_reformulated}, and then we derive a solution to Problem~\eqref{problem:P_initial} via Theorem~\ref{theorem:connections_solutions}.

\begin{proof}
    Consider Framework~\ref{framework:PO} under Assumptions~\ref{assumption:pof_minimal} and~\ref{assumption:pof_additional}.
    Let~$(x^k)_{k \in \Nbb}$ be generated by Algorithm~\ref{algo:covering_framework} solving Reformulation~\eqref{problem:P_reformulated}.
    Proposition~\ref{proposition:cDSM_applied_to_Phi} is applicable by Proposition~\ref{proposition:requirements_OK}, so all~$x^* \in \acc((x^k)_{k \in \Nbb}) \neq \emptyset$ are generalized local solutions to Reformulation~\eqref{problem:P_reformulated}, so~$\ACC(\gamma;x^*)$ contains a generalized local solution to Problem~\eqref{problem:P_initial} by Theorem~\refbis{theorem:connections_solutions}{generalized}.
    This proves the first part of Theorem~\ref{theorem:solving_P}.
    If moreover~$x^* \in X$ and~$\varphi$ is lower semicontinuous at all~$y \in \ACC(\gamma;x^*)$ and~$\ACC(\gamma;x^*) \subseteq \Omega$, then~$\Phi$ is lower semicontinuous at~$x^*$ by Proposition~\ref{proposition:semicontinuity_Phi}, so~$x^*$ is a local solution to Reformulation~\eqref{problem:P_reformulated} by Proposition~\ref{proposition:cDSM_applied_to_Phi}, so~$\gamma(x^*)$ is a local solution to Problem~\eqref{problem:P_initial} by Theorem~\refbis{theorem:connections_solutions}{local}.
    This completes the proof of Theorem~\ref{theorem:solving_P}.
\end{proof}

\clearpage\newpage
\section{Application of the \pof to an infinite-dimensional problem}
\label{section:example_optimal_control}

Section~\ref{section:intro/motivation} mentions that the design of the \pof allows addressing infinite-dimensional problems.
This section proposes a simple infinite-dimensional example for illustration.
Precisely, we investigate a discontinuous version of a classical optimal control problem that challenges standard theoretical and numerical techniques found in the literature.
In contrast, the \pof allows reformulating this problem so that it may be solved analytically.
This section aims for a concise analytical illustration only, using an example from the PhD thesis~\cite[Chapter~7.3.1]{BouchetPhD}.
We refer to~\cite[Chapter~7]{BouchetPhD} and our technical report~\cite{BoAuBo21PWMayerCost} for numerical experiments applying the \pof to optimal control problems and using the \pom to solve them.

We study a discontinuous version of the textbook \textit{double integrator} problem~\cite{SethiTompson00Controle,Vinter10Controle}.
This problem considers a point moving along a line, with position~$p(t)$, velocity~$v(t)$, and (controlled) acceleration~$u(t)$, for all time~$t$ in the interval~$[0,5]$.
The problem seeks to minimize the \textit{ceiling}~$\ceil{\cdot}$ of~$p(5)$ plus the integral of~$\frac{1}{2}u^2$.
The initial state~$(p,v)(0) \defequal (0,0)$ and the final velocity~$v(5) \defequal 0 $ are fixed.
Let us denote also by~$\Acal\Ccal$ the set of absolutely continuous functions from~$[0,5]$ to~$\Rbb$, endowed with the standard norm~$\textnorm{\cdot}_{\Acal\Ccal}$ of the Sobolev space~$W^{1,1}([0,5]) \cong \Acal\Ccal$, and by~$\Lcal^\infty$ the set of essentially bounded functions from~$[0,5]$ to~$\Rbb$, endowed with the standard norm~$\textnorm{\cdot}_{\Lcal^\infty}$.
The problem is

\begin{equation}
    \label{problem:integrator}
    \tag{$\Pbf^{\texttt{DI}}_{\texttt{ini}}$}
    \problemoptim
        {\minimize}
        {(p,v,u) \in \mathcal{AC} \times \mathcal{AC} \times \Lcal^{\infty}}
        {\ceil{p(5)} + \dinteg{0}{5} \dfrac{u(t)^2}{2}~ dt}
        {
            \dot{p}(t) = v(t), & \mbox{a.e.}~ t \in [0,5], \\
            \dot{v}(t) = u(t), & \mbox{a.e.}~ t \in [0,5], \\
            (p,v)(0) = (0,0), \\
            v(5) = 0.
            }
\end{equation}

If Problem~\eqref{problem:integrator} was written without the ceiling~$\textceil{p(5)}$, then it would fit a usual linear-quadratic structure that could be solved using the acknowledged \textit{Pontryagin's maximum principle} (\pmp)~\cite{Pontryagin62PMP}.
Yet, due to the discontinuities introduced by~$\textceil{p(5)}$, neither the \pmp nor its adaptations~\cite{BaPf19TimeCrisis,Clarke13Book} are applicable.
Other usual techniques such as \textit{direct and indirect numerical methods}~\cite{Rao10SurveyNumControl} also fail on this problem.
We refer to our technical report~\cite{BoAuBo21PWMayerCost} for a more detailed discussion on this fact.

Nevertheless, as we now show, the \pof allows solving Problem~\eqref{problem:integrator} analytically.
To fit the \pof Framework~\ref{framework:PO}, we define~$\Ybb \defequal \mathcal{AC} \times \mathcal{AC} \times \Lcal^{\infty}$, and the norm~$\textnorm{(p,v,u)}_{\Ybb} \defequal \textnorm{p}_{\Acal\Ccal} + \textnorm{v}_{\Acal\Ccal} + \textnorm{u}_{\Lcal^\infty}$ for all~$(p,v,u) \in \Ybb$.
Then we define the cost functional~$\varphi : y \defequal (p,v,u) \in \Ybb \mapsto \ceil{p(5)} + \frac{1}{2}\integ_{0}^{5} u(t)^2~ dt$.
To partition~$\Ybb$, we define~$\Xbb \defequal \Rbb$ and for all~$x \in \Xbb$, we set~$\Ybb(x) \defequal \{(p,v,u) \in \Ybb : p(5) = x\}$.
It follows that~$\chi$ is defined as~$\chi(p,v,u) \defequal p(5)$ for all~$(p,v,u) \in \Ybb$.
Then we apply the \pof as described in Section~\ref{section:results/framework}.
For all~$x \in \Xbb$, the subproblem of minimizing~$\varphi$ over~$\Ybb(x)$ reads as

\begin{equation}
    \label{problem:integrator_subproblem}
    \tag{$\Pbf^{\texttt{DI}}_{\mathrm{sub}}(x)$}
    \problemoptim
        {\minimize}
        {(p,v,u) \in \mathcal{AC} \times \mathcal{AC} \times \Lcal^{\infty}}
        {\ceil{p(5)} + \dinteg{0}{5} \dfrac{u(t)^2}{2}~ dt}
        {  \dot{p}(t) = v(t), & \mbox{a.e.}~ t \in [0,5], \\
           \dot{v}(t) = u(t), & \mbox{a.e.}~ t \in [0,5], \\
           (p,v)(0) = (0,0), \\
           (p,v)(5) = (x,0).
        }
\end{equation}

For all~$x \in \Xbb$, Subproblem~\eqref{problem:integrator_subproblem} has a classical linear-quadratic structure, since~$\textceil{p(5)} = \textceil{x}$ is constant.
Thus, Subproblem~\eqref{problem:integrator_subproblem} admits a solution and the \pmp is applicable.
This latter provides the necessary optimality condition that, if~$(p^*,v^*,u^*) \in \mathcal{AC} \times \mathcal{AC} \times \Lcal^{\infty}$ is a solution to Subproblem~\eqref{problem:integrator_subproblem}, then there exists a nontrivial triplet~$(q_0^*,q_p^*,q_v^*) \in \{ -1, 0 \} \times \mathcal{AC} \times \mathcal{AC}$ such that
\begin{equation*}
    \label{system:PMP}
    \tag{\pmp}
    \left\{\begin{array}{llll}
        p^*(0) = 0, & p^*(5) = x, & \dot{p}^*(t) = v^*(t),     & \mbox{ a.e.}~ t \in [0,5], \\
        v^*(0) = 0, & v^*(5) = 0, & \dot{v}^*(t) = u^*(t),     & \mbox{ a.e.}~ t \in [0,5], \\
                    &             & \dot{q}_p^*(t) = 0,        & \mbox{ a.e.}~ t \in [0,5], \\
                    &             & \dot{q}_v^*(t) = q_p^*(t), & \mbox{ a.e.}~ t \in [0,5], \\
        \multicolumn{3}{l}{ u^*(t) \in \argmax_{\omega} q_p^*(t)v^*(t) + q_v^*(t)\omega + q_0^*\frac{\omega^2}{2}}, & \mbox{ a.e.}~ t \in [0,5].
    \end{array}\right.
\end{equation*}

Then, for all~$x \in \Xbb$, elementary calculations ensure that System~\eqref{system:PMP} has a unique solution given by

\begin{equation*}
    (p^*_x,v^*_x,u^*_x) (t) = \dfrac{x}{125} \left(t^2(15-2t) , t(30-6t) , 30-12t) \right) \quad \mbox{a.e.}~ t \in [0,5],
\end{equation*}
and thus, Subproblem~\eqref{problem:integrator_subproblem} has a unique global solution given by
\begin{equation*}
    \gamma(x) \defequal (p^*_x,v^*_x,u^*_x),
    \quad \text{with} \quad
    \varphi\left(\gamma(x)\right) = \ceil{x} + \dfrac{12}{125}x^2.
\end{equation*}
Hence, in this context, Reformulation~\eqref{problem:P_reformulated} reads as
\begin{equation}
    \label{problem:integrator_reformulated}
    \tag{$\Pbf^{\texttt{DI}}_\texttt{ref}$}
    \problemoptimfree{\minimize}{x \in \Xbb}{\Phi(x)},
    \quad \text{where} \quad
    \Phi \defequal \varphi \circ \gamma.
\end{equation}

Now that Problem~\eqref{problem:integrator} fits into the \pof Framework~\ref{framework:PO}, observe that Assumption~\ref{assumption:pof_minimal} holds.
Indeed,~$\chi$ is continuous, and~$\gamma$ is uniformly continuous so Assumption~\ref{assumption:pof_minimal} is nonstringent with respect to the partition of~$\Xbb$.
A natural admissible partition of~$\Xbb$ is given by~$\Xbb = \partitioncup_{n \in \Zbb} ]n,n+1]$, since the ceiling~$\textceil{x}$ involved in~$\varphi(\gamma(x))$ is constant over each of these intervals.
Consequently, Theorem~\ref{theorem:connections_solutions} is applicable.
By Theorem~\refbis{theorem:connections_solutions}{global}, any global solution~$x^*$ to~Reformulation~\eqref{problem:integrator_reformulated} leads to~$\gamma(x^*)$ as a global solution to Problem~\eqref{problem:integrator}.
Moreover, the minimizer of~$\Phi$ is tractable as~$x^* \defequal -5$.
Thus, by Theorem~\refbis{theorem:connections_solutions}{global},

\begin{equation*}
    \gamma(x^*),
    \quad \text{with}~ x^* \defequal -5, \quad
    \text{is a global solution to Problem~\eqref{problem:integrator}}.
\end{equation*}

This proves that the discontinuous Problem~\eqref{problem:integrator}, which challenges standard techniques from the literature on optimal control theory, is solved analytically using the \pof.
Remark that solving a naive alteration of Problem~\eqref{problem:integrator} that consists in replacing~$\textceil{p(5)}$ by~$p(5)$ in the objective function fails to recover a global solution to the true problem.
Indeed, this altered problem is tractable using the \pmp, but its solution equals~$\gamma(x^\sharp)$ with~$x^\sharp \defequal \frac{-125}{24} \approx -5.20 \neq x^*$.
This example illustrates how the \pof may be used to tackle some infinite-dimensional problems that admit difficulties depending on a few variables.
Moreover, we stress that we need neither a discretization scheme to make the dimension finite nor a \dfo algorithm to solve Reformulation~\eqref{problem:integrator_reformulated}, since we obtain its solution analytically.
Thus, this examples relies only on the \pof and not on the \pom, so we do not need to check Assumption~\ref{assumption:pof_additional}.

\section{Application of the \pof to composite problems}
\label{section:applications}

In this section, we focus on a class of \textit{composite optimization problems} (see~\cite{LaMe24structureAware,LaMeZh21ManifoldComposite} and many references therein, or~\cite[Section~5]{LaMeWi2019}) where~$\Omega$ is described by smooth constraints and~$\varphi$ is defined by
\begin{equation*}
    \fct{\varphi}{y}{\Ybb}{\widetilde{\varphi}(y,\sigma(f(y)))+\varepsilon(f(y)),}{\Rbbinf}
\end{equation*}
where~$\widetilde{\varphi}: \Ybb \times \Rbb \to \Rbb$ is smooth and easy to minimize and has an explicit definition, and~$f: \Ybb \to \Xbb$ is smooth with explicit definition (and~$\dim(\Xbb) \in \Nbb^*$ small), and~$\sigma: \Xbb \to \Rbb$ and~$\varepsilon: \Xbb \to \Rbbinf$ are \textit{blackboxes} (that is, their expressions are both unknown).
We say that~$\varphi$ is a \textit{greybox}, since its explicit expression depends on blackbox components.
Such \textit{composite greybox problems} are compatible with the \pof.

In this context, adding a constraint~$f(y) = x$, with any~$x \in \Xbb$ fixed, simplifies Problem~\eqref{problem:P_initial}.
Indeed,~$f(y) = x$ is a smooth equality constraint and~$\sigma(f(y)) = \sigma(x)$ and~$\varepsilon(f(y)) = \varepsilon(x)$ are constant.
Accordingly, the \pof partitions~$\Ybb = \partitioncup_{x \in \Xbb} \Ybb(x)$ with~$\Ybb(x) \defequal \{y \in \Ybb: f(y) = x\}$ for all~$x \in \Xbb$ (note that this partition yields~$\chi \defequal f$ so the continuity of~$\chi$ naturally holds).
Then, either Subproblem~\eqref{problem:P_subproblem} is infeasible or, by assumption, a global solution~$\gamma(x)$ may be computed.
Thus, Reformulation~\eqref{problem:P_reformulated} is a low-dimensional blackbox problem where, given any~$x \in \Xbb$, we try to compute~$\gamma(x)$ and~$\sigma(x)$ and~$\varepsilon(x)$ on-the-fly, and we return either~$\Phi(x) \defequal \widetilde{\varphi}(\gamma(x),\sigma(x))+\varepsilon(x)$ if~$\gamma(x)$ exists or~$\Phi(x) \defequal +\infty$ otherwise.
We then use the \pom to solve Reformulation~\eqref{problem:P_reformulated} and deduce a (possibly, generalized) local solution to Problem~\eqref{problem:P_initial}.

This section solves four examples using the \pom.
In Section~\ref{section:applications/monovariable}, we solve the example from Section~\ref{section:intro/illustrative_example}.
In Section~\ref{section:applications/radial},~$f$ returns the radius of the polar coordinates and~$\gamma$ is easily tractable.
In Section~\ref{section:applications/nonlinear},~$f$ returns a nonlinear combination of all variables and the expression of~$\gamma$ may be calculated.
In Section~\ref{section:applications/dim2},~$f$ returns two nonlinear combinations of all variables, and~$\gamma$ cannot be expressed in closed form but Subproblem~\eqref{problem:P_subproblem} is solvable with a dedicated algorithmic method.
These examples show the variety of what Theorem~\ref{theorem:solving_P} may guarantee.
In Sections~\ref{section:applications/monovariable} and~\ref{section:applications/nonlinear}, Theorem~\ref{theorem:solving_P} ensures the optimality of~$\gamma(x^*)$ only in the generalized local sense, although~$\gamma(x^*)$ is a usual local solution in Section~\ref{section:applications/monovariable}.
In Sections~\ref{section:applications/radial} and~\ref{section:applications/dim2}, Theorem~\ref{theorem:solving_P} successfully claims the usual local optimality of~$\gamma(x^*)$.
In all cases, we implement the \pom with the \cdsm as described in Remark~\ref{remark:algo_parameters}, so we solve Reformulation~\eqref{problem:P_reformulated} numerically up to some stopping criterion and the \pom outputs an approximation of a solution to Problem~\eqref{problem:P_initial}.

\subsection{Mono-variable noise}
\label{section:applications/monovariable}

Consider again the example from Section~\ref{section:intro/illustrative_example}, defined via~$\Ybb = Y \defequal \Rbb^2$, and~$\Omega \defequal \Ybb$, and
\begin{equation*}
    \fct{\varphi}{(y_1,y_2)}{\Ybb}
    {\left(y_2-\sigma(y_1)\right)^2 + \varepsilon(y_1),}{\Rbb}
\end{equation*}
where we define~$\sigma$ and~$\varepsilon$ as follows, but the \pof may be applied without this information.
For all~$x \in \Rbb$, define~$\sigma(x) \defequal 2\floorceil{x}$, and~$\varepsilon(x) \defequal \abs{x}(1+\sin(\frac{2\pi}{x})^2)^{\frac{1}{2}}+\abs{\floorceil{x}}$ if~$x \in \Rbb^*$ and~$\varepsilon(0) \defequal 0$, where~$\floorceil{x} \defequal \floor{x}$ if~$x \in \Rbb_-$ and~$\floorceil{x} \defequal \ceil{x}-1$ if~$x \in \Rbb_+^*$.
Then, the unique minimizer of~$\varphi$ is~$y^* \defequal (0,0)$, with~$\varphi(y^*) = 0$.
In this context, we define~$\Ybb(x) \defequal \{x\}\times\Rbb$ for all~$x \in \Xbb = X \defequal \Rbb$, so the unique minimizer of~$\varphi_{|\Ybb(x)}$ is
\begin{equation*}
    \gamma(x) \defequal (x,\sigma(x)),
    \quad \mbox{with} \quad
    \varphi(\gamma(x)) = \varepsilon(x),
\end{equation*}
pictured in Figure~\ref{figure:application_monovariable_phi} (right) and Figure~\ref{figure:application_monovariable_Phi} (left).
It follows that Reformulation~\eqref{problem:P_reformulated} minimizes
\begin{equation*}
    \fct{\Phi}{x}{\Xbb}{\varepsilon(x),}{\Rbb}
\end{equation*}
shown on Figure~\ref{figure:application_monovariable_Phi} (right).
Its global minimizer is~$x^* \defequal 0$, with~$\Phi(x^*) = 0$, and~$\Phi$ is left-discontinuous at~$x^*$.
Assumptions~\ref{assumption:pof_minimal} and~\ref{assumption:pof_additional} hold.
Table~\ref{table:application_monovariable_results} shows that the \cdsm approaches~$x^*$ closely and from the side avoiding the discontinuity.
All instances return~$\widetilde{x}^* \in [0,2E^{-10}]$.
Hence, they all provide~$\widetilde{y}^* \defequal \gamma(\widetilde{x}^*)$ such that~$\norm{\widetilde{y}^* - y^*} \leq 2E^{-10}$ and~$\abs{\varphi(\widetilde{y}^*) - \varphi(y^*)} \leq 2E^{-10}$.
This empirically validates Theorem~\ref{theorem:solving_P}.
We also remark that~$\gamma(x^*)$ is a local solution to Problem~\eqref{problem:P_initial}, even though the requirement in Theorem~\ref{theorem:solving_P} that~$\varphi$ is lower semicontinuous at all points in~$\ACC(\gamma;x^*)$ does not hold (it fails at~$(0,-1) \in \ACC(\gamma;x^*)$).

\begin{figure}[!ht]
    \centering
    \includegraphics[width=0.43\linewidth, trim = 0 20 0 40,clip]{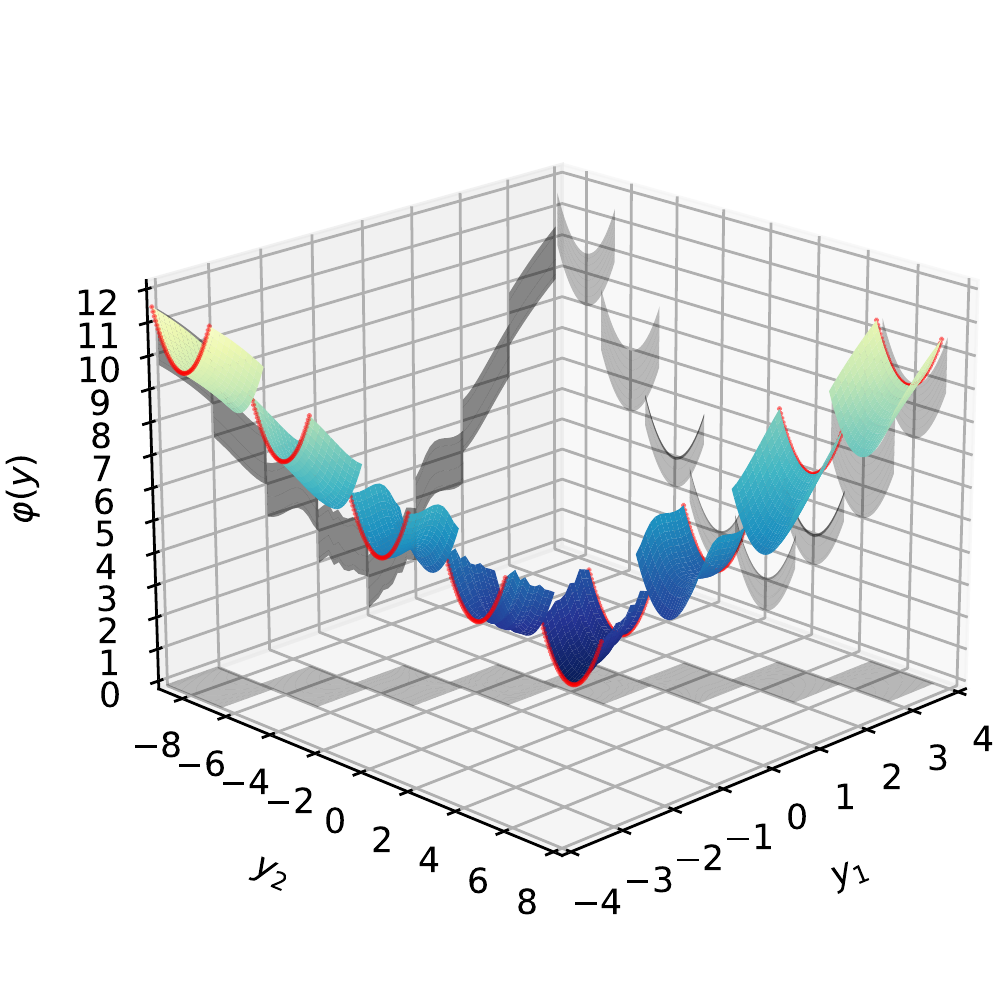}
    \hfill
    \includegraphics[width=0.41\linewidth, trim = 5 5 5 5,clip]{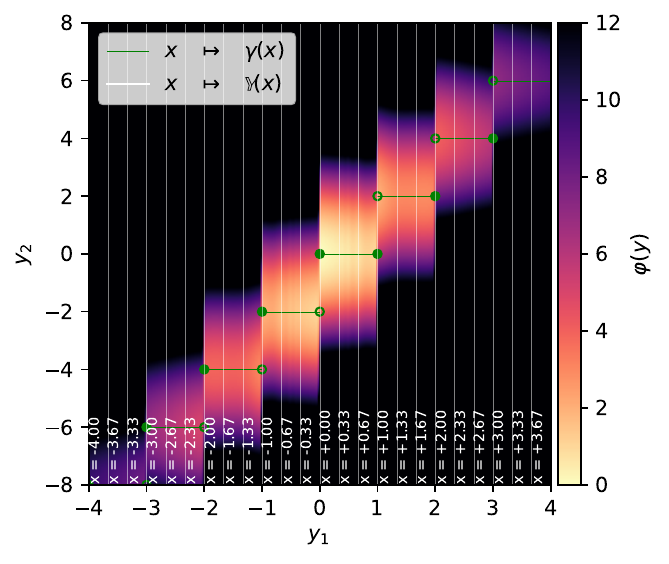}
    \caption{Function~$\varphi$ in Section~\ref{section:applications/monovariable}, (left) 3D view, (right) 2D view, partition of~$\Ybb$ and locus of~$\gamma$. For each~$y_1 \in \Rbb$, we plot~$\varphi_{|\{y_1\}\times\Rbb}$ for the values~$y_2 \in [\sigma(y_1)-1.3,\ \sigma(y_1)+1.3]$ only.}
    \label{figure:application_monovariable_phi}
\end{figure}

\begin{figure}[!ht]
    \centering
    \includegraphics[width=0.49\linewidth, trim = 8 10 7 7,clip]{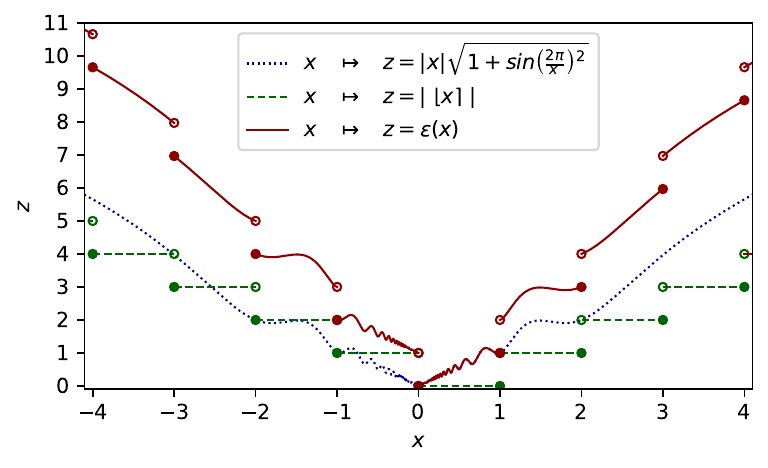}
    \hfill
    \includegraphics[width=0.49\linewidth, trim = 8 10 7 7,clip]{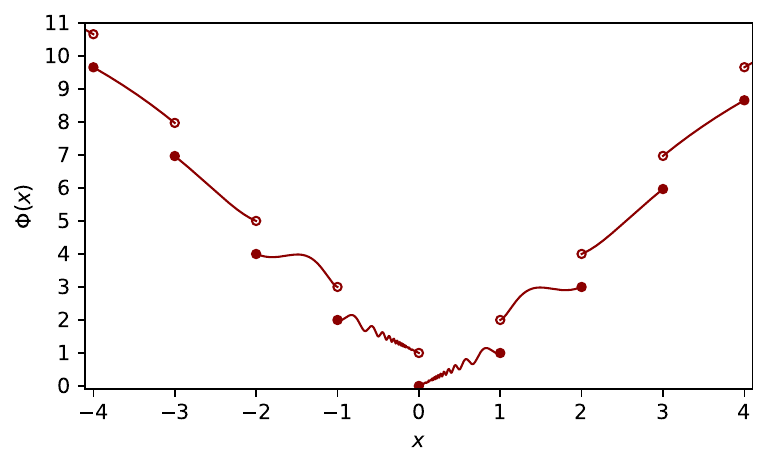}
    \caption{Functions~$\varepsilon$ and~$\Phi$ in Section~\ref{section:applications/monovariable}, (left) function~$\varepsilon$ and its components, (right) function~$\Phi$. In this section,~$\Phi = \varepsilon$.}
    \label{figure:application_monovariable_Phi}
\end{figure}

\begin{table}[!ht]
    \centering
    \input{Figures/FirstCoord_2Vars_table.tex}
    \caption{Results for eight instances of the \cdsm in the context of Section~\ref{section:applications/monovariable}. In this table, $\hat{x}^k \defequal x^k$.}
    \label{table:application_monovariable_results}
\end{table}

\subsection{Radial noise}
\label{section:applications/radial}

Consider an unconstrained ($\Omega \defequal \Ybb$) case in polar coordinates ($\Ybb = Y \defequal \Rbb_+\times[0,2\pi[$), involving
\begin{equation*}
    \fct{\varphi}{(r,\theta) }{\Ybb}
    {\sqrt{r}\sin\left(\dfrac{\theta-\sigma(r)}{2}\right)^2 + \varepsilon(r),}{\Rbb}
\end{equation*}
where the blackboxes~$\sigma$ and~$\varepsilon$ are defined as follows (but the \pof has no access to this information).
For all~$x \in \Rbb_+$, define~$\sigma(x) \defequal \pi-2\pi\log_2(x)$ if~$x > 0$ and~$\sigma(0) \defequal 0$, and~$\varepsilon(x) \defequal \sqrt{\textabs{x^2-2}} + \frac{\sin(10\pi(x-\sqrt{2}))^2}{10}$.
Then, the global minimizer of~$\varphi$ is~$y^* \defequal (\sqrt{2},0)$, with~$\varphi(y^*) = 0$.
We define~$\Ybb(x) \defequal \{x\}\times[0,2\pi[$ for all~$x \in \Xbb = X \defequal \Rbb_+$, so the unique global minimizer of~$\varphi_{|\Ybb(x)}$ is
\begin{equation*}
    \gamma(x) \defequal (x,\sigma(x)\sim2\pi),
    \quad \mbox{with} \quad
    \varphi(\gamma(x)) = \varepsilon(x),
\end{equation*}
where~$\sigma(x)\sim2\pi$ denotes the residual of~$\sigma(x)$ modulo~$2\pi$.
Figure~\ref{figure:application_radial_phi} (left) draws~$\varphi$, Figure~\ref{figure:application_radial_phi} (right) shows the locus of~$\gamma$, and Figure~\ref{figure:application_radial_Phi} (left) pictures~$\varepsilon$.
Then the function~$\Phi$ equals
\begin{equation*}
    \fct{\Phi}{x}{\Xbb}{\varepsilon(x),}{\Rbb}
\end{equation*}
shown on Figure~\ref{figure:application_radial_Phi} (right).
Its minimizer is~$x^* \defequal \sqrt{2}$, with~$\Phi(x^*) = 0$.
Assumptions~\ref{assumption:pof_minimal} and~\ref{assumption:pof_additional} hold.
Table~\ref{table:application_radial_results} shows the convergence of the \cdsm towards~$x^*$.
All instances return~$\widetilde{x}^* \in [x^* \pm 6E^{-11}]$, so~$\widetilde{y}^* \defequal \gamma(\widetilde{x}^*)$ satisfies~$\norm{\widetilde{y}^*-y^*} \leq 6E^{-11}$ and~$\abs{\varphi(\widetilde{y}^*)-\varphi(y^*)} \leq 3E^{-5}$.
This agrees with Theorem~\ref{theorem:solving_P}.

\begin{table}[!ht]
    \centering
    \input{Figures/Polar_2Vars_table.tex}
    \caption{Results for eight instances of the \cdsm in the context of Section~\ref{section:applications/radial}. In this table,~$\hat{x}^k \defequal x^k-\sqrt{2}$.}
    \label{table:application_radial_results}
\end{table}

\begin{figure}[!ht]
    \centering
    \hfill
    \includegraphics[width=0.40\linewidth, trim=0 0 0 20, clip]{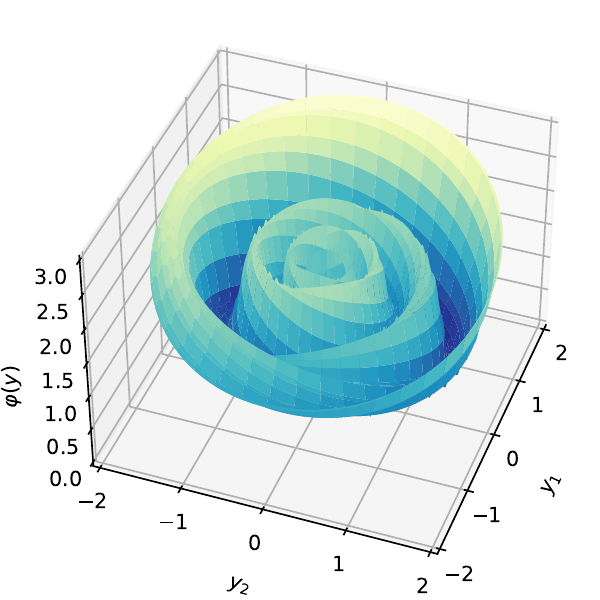}
    \hfill
    \includegraphics[width=0.45\linewidth, trim = 5 10 5 5,clip]{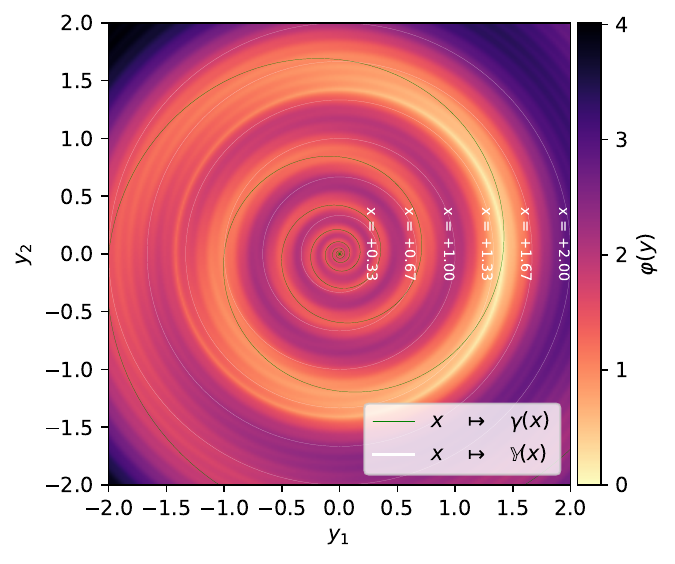}
    \caption{Function~$\varphi$ in Section~\ref{section:applications/radial}, (left) 3D view, (right) 2D view, partition of~$\Ybb$ and locus of~$\gamma$.}
    \label{figure:application_radial_phi}
\end{figure}

\begin{figure}[!ht]
    \centering
    \includegraphics[width=0.49\linewidth, trim = 8 10 7 7,clip]{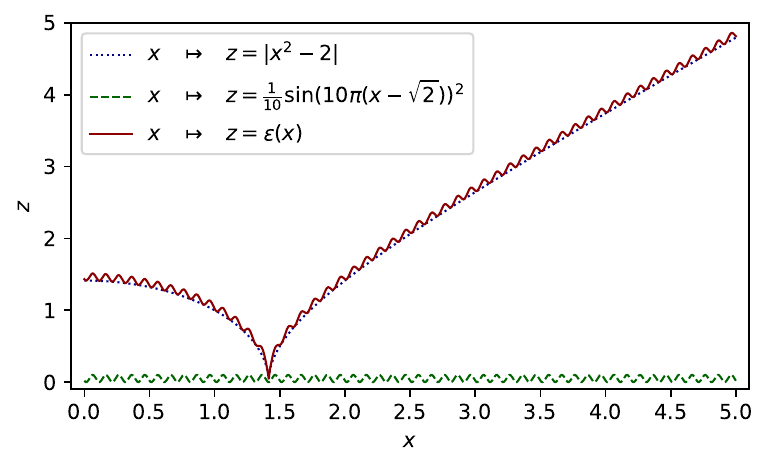}
    \hfill
    \includegraphics[width=0.49\linewidth, trim = 8 10 7 7,clip]{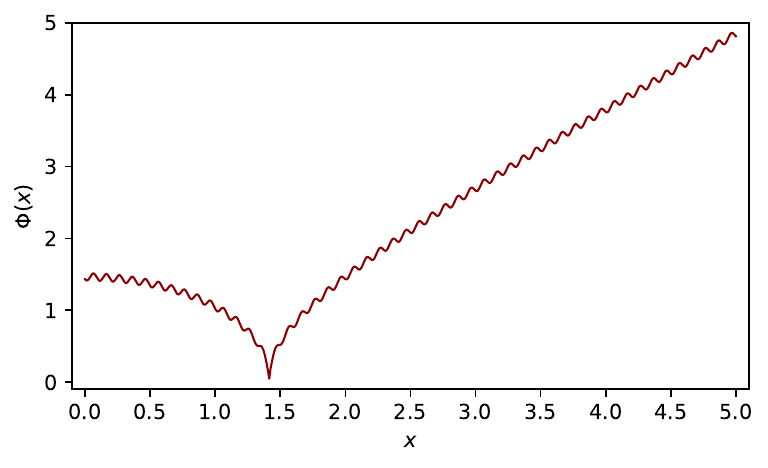}
    \caption{Functions~$\varepsilon$ and~$\Phi$ in Section~\ref{section:applications/radial}, (left) function~$\varepsilon$ and its components, (right) function~$\Phi$. In this section,~$\Phi = \varepsilon$.}
    \label{figure:application_radial_Phi}
\end{figure}

\subsection{Noise affected by a nonlinear combination of all variables}
\label{section:applications/nonlinear}

Consider the variables space~$\Ybb \defequal \Rbb^2$, the feasible set~$\Omega \defequal \Rbb_+\times\Rbb$ and the function
\begin{equation*}
    \fct{\varphi}{y \defequal (y_1,y_2)}{\Ybb}
    {\ln\left(1+\left(\dfrac{y_1^2}{y_2^2+1}-1\right)^2\right) + \varepsilon(f(y)),}{\Rbbinf}
\end{equation*}
where~$f: y \in \Ybb \mapsto y_1y_2 \in \Rbb$ is explicitly known, while the blackbox~$\varepsilon: \Rbb \to \Rbbinf$ is defined via~$\varepsilon(4) \defequal +\infty$ and~$\varepsilon(x) \defequal \exp((x-4)^{-1}) + \frac{1}{5}\sqrt{\textabs{x-4}}$ if~$x \neq 4$.
Then~$\varphi$ has no global minimizer, but it has a generalized global minimizer~$y^* \approx (2.13,1.88)$, since~$\inf \varphi(\Omega) = 0$ and~$\varphi(y) \to 0$ as~$y \to y^*$ from some direction, even though~$\varphi(y^*) = +\infty$.
The exact value of~$y^*$ and the direction are given below.
We define~$\Ybb(x) \defequal \{y \in \Ybb : f(y) = x\}$ for all~$x \in \Xbb = X \defequal \Rbb$, so~$\varphi_{|\Ybb(x)}$ is smooth (since~$f(y) = x$ so~$\varepsilon(f(y))$ is constant) and its minimizer is computable (by solving~$y_1^2/(y_2^2+1) = 1$ subject to~$y \in \Omega$ and~$f(y) = x$).
Then, for all~$x \in \Xbb$, the minimizer of~$\varphi_{|\Ybb(x)}$ is
\begin{equation*}
    \gamma(x) \defequal \left(
        \sqrt{\frac{1+\sqrt{1+4x^2}}{2}}
        ,
        x\sqrt{\frac{2}{1+\sqrt{1+4x^2}}}
    \right),
    \quad \mbox{with} \quad
    \varphi(\gamma(x)) = \varepsilon(x).
\end{equation*}
Figure~\ref{figure:application_nonlinear_phi} shows~$\varphi$, the partition of~$\Ybb$ and the locus of~$\gamma$, and Figure~\ref{figure:application_nonlinear_Phi} (left) shows~$\varepsilon$.
Remark that~$\gamma$ is continuous, even though~$\varphi$ is not.
In this context,~$\Phi$ equals
\begin{equation*}
    \fct{\Phi}{x}{\Xbb}{\varepsilon(x),}{\Rbb}
\end{equation*}
shown on Figure~\ref{figure:application_nonlinear_Phi} (right).
Then~$\Phi$ has no minimum but~$x^* \defequal 4$ is a generalized global minimizer, since~$\lim_{x \nearrow x^*} \Phi(x) = 0 = \inf \Phi(\Xbb)$ despite~$\lim_{x \searrow x^*} \Phi(x) = +\infty = \Phi(x^*)$.
Thus,~$y^* \defequal \gamma(x^*)$ is a generalized global minimizer of~$\varphi$ and~$\lim_{y = \gamma(x \nearrow x^*)} \varphi(y) = 0$.
Assumptions~\ref{assumption:pof_minimal} and~\ref{assumption:pof_additional} hold.

We test eight instances of the \cdsm minimizing~$\Phi$.
Table~\ref{table:application_nonlinear_results} shows their convergence towards~$x^*$.
All instances starting with~$x^0 < x^*$ return~$\widetilde{x}^* \in [x^*-2E^{-10},x^*]$, which approaches~$x^*$ closely and from the direction minimizing~$\Phi$.
Hence, they provide~$\widetilde{y}^* \defequal \gamma(\widetilde{x}^*) \approx y^*$ with~$\varphi(\widetilde{y}^*) \approx \inf \varphi(Y)$, which agrees with Theorem~\ref{theorem:solving_P}.
The instance starting with~$x^0 \defequal 3\sqrt{2} \approx 4.24 > x^*$ behaves similarly.
The last three instances, which start with~$x^0 \gg 4$, converge instead to~$\widetilde{x}^* \approx 9.27$, but this is consistent with Theorem~\ref{theorem:solving_P} since~$\Phi$ has a local minimizer around~$9.27$ so~$\varphi$ has a local minimizer around~$\gamma(9.27)$.

\begin{table}[!ht]
    \centering
    \input{Figures/NonlinComb_2Vars_table.tex}
    \caption{Results for eight instances of Algorithm~\ref{algo:cdsm} in the context of Section~\ref{section:applications/nonlinear}. In this table,~$\hat{x}^k \defequal x^k-4$. For each of the last three instances, we arbitrarily return~$x^{201} \approx 9.26779505$ since~$(x^k)_{k \geq 40}$ has its first~$9$ digits constant in all cases.}
    \label{table:application_nonlinear_results}
\end{table}

\begin{figure}[!ht]
    \centering
    \hfill
    \includegraphics[width=0.43\linewidth, trim=0 15 15 40, clip]{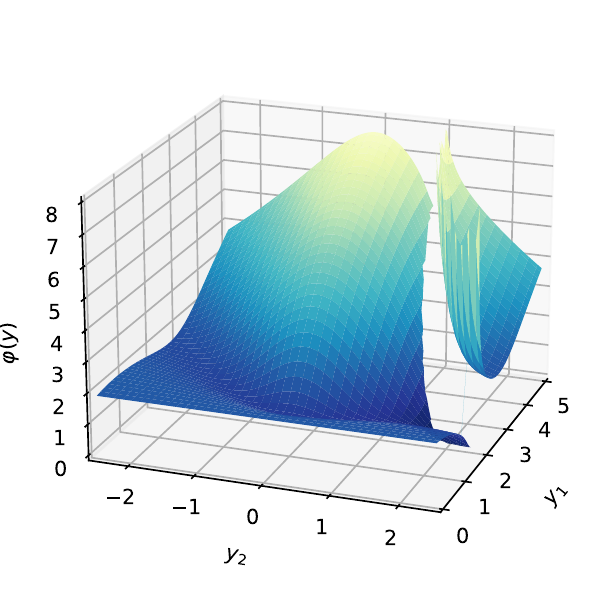}
    \hfill
    \includegraphics[width=0.43\linewidth, trim = 5 5 5 5,clip]{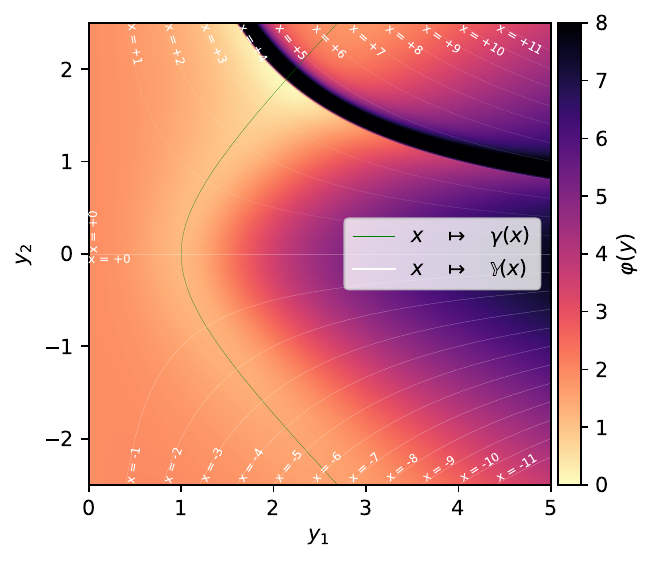}
    \caption{Function~$\varphi$ in Section~\ref{section:applications/nonlinear}, (left) 3D view, (right) 2D view, partition of~$\Ybb$ and the locus of~$\gamma$.}
    \label{figure:application_nonlinear_phi}
\end{figure}

\begin{figure}[!ht]
    \centering
    \includegraphics[width=0.49\linewidth, trim = 8 10 7 7,clip]{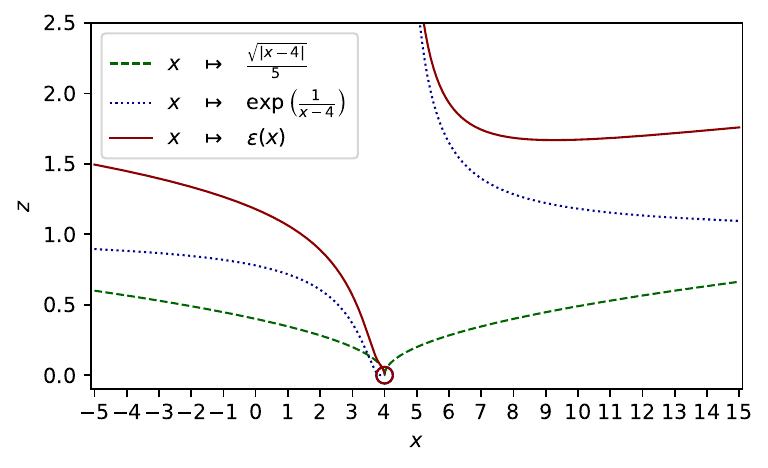}
    \hfill
    \includegraphics[width=0.49\linewidth, trim = 8 10 7 7,clip]{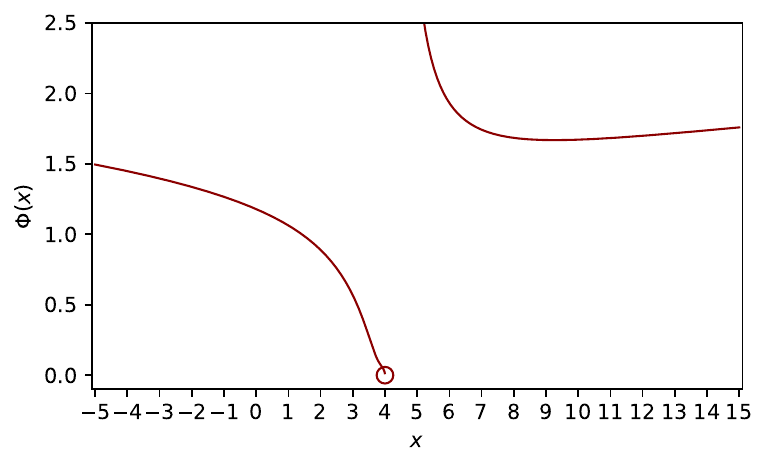}
    \caption{Functions~$\varepsilon$ and~$\Phi$ in Section~\ref{section:applications/nonlinear}, (left) function~$\varepsilon$ and its components, (right) function~$\Phi$. In this section,~$\Phi = \varepsilon$. Note that neither $\varepsilon$ nor $\Phi$ have a global minimum.}
    \label{figure:application_nonlinear_Phi}
\end{figure}

\subsection{Bi-dimensional noise and oracle provided by an algorithmic method}
\label{section:applications/dim2}

Consider~$\Ybb = Y \defequal \Rbb^3$, and~$\Omega \defequal \Ybb$, and the function
\begin{equation*}
    \fct{\varphi}{y}{\Ybb}
    {\norm{y}_\infty + \varepsilon\left(f(y)\right),}{\Rbb}
\end{equation*}
where~$f$, defined by~$f(y) \defequal (y_2-y_1^3,y_1-y_3^3)$ for all~$y \defequal (y_1,y_2,y_3) \in \Ybb$, is explicitly known while the blackbox~$\varepsilon$ is defined by~$\varepsilon(x) \defequal (\frac{\sin(10\pi(x_2-x_1^3))}{5} + \frac{\sin(6\pi(x_2-e^{-x_1}+1))}{7} + \frac{\sin(12\pi(x_1^2+x_2^2)^{\frac{1}{2}})}{11})^2$ for all~$x \in \Rbb^2$.
The global minimizer of~$\varphi$ is~$y^* \defequal (0,0,0)$, with~$\varphi(y^*) = 0$.
In this context, we define the partition sets of~$\Ybb$ as~$\Ybb(x) \defequal \{y \in \Ybb : f(y) = x\}$ for all~$x \in \Xbb = X \defequal \Rbb^2$, so~$\varepsilon_{|\Ybb(x)}$ is constant.
For all~$x \defequal (x_1,x_2) \in \Xbb$, some calculation allows us to rewrite~$\Ybb(x) = \{(t,t^3+x_1,\sqrt[3]{t-x_2}),~ t \in \Rbb\}$, so the unique minimizer of~$\varphi_{|\Ybb(x)}$ is the solution to the problem
\begin{equation*}
    \problemoptimfree{\minimize}{y \in \Ybb(x)}{\varepsilon(x) + \max\left\{\abs{t_x(y)}, \abs{t_x(y)^3+x_1}, \abs{\sqrt[3]{t_x(y)-x_2}}\right\}},
\end{equation*}
where, for any~$y \in \Ybb(x)$,~$t_x(y)$ denotes the unique~$t \in \Rbb$ such that~$s_x(t) \defequal (t,t^3+x_1,\sqrt[3]{t-x_2}) = y$.
This solution is hardly tractable.
However, the problem admits a smooth exact reformulation given by
\begin{equation*}
    \problemoptim{\minimize}{M \geq 0,~ T \in \Rbb}{\varepsilon(x) + M}
    {T \in I_x^1(M) \defequal \left[-M,M\right], \\
     T \in I_x^2(M) \defequal \left[\sqrt[3]{-M-x_1},\sqrt[3]{M-x_1}\right], \\
     T \in I_x^3(M) \defequal \left[-M^3+x_2,M^3+x_2\right].}
\end{equation*}
Its unique solution, denoted by~$(M(x),T(x))$, provides
\begin{equation*}
    \gamma(x) \defequal s_x(T(x)),
    \quad \mbox{with} \quad
    \varphi(\gamma(x)) = \varepsilon(x) + M(x).
\end{equation*}
Figure~\ref{figure:application_dim2_phi} shows~$\varphi$, Figure~\ref{figure:application_dim2_partition} draws the partition of~$\Ybb$ and the locus of~$\gamma$, and Figure~\ref{figure:application_dim2_Phi} (left) and (centre) represent respectively~$\varepsilon$ and~$M$.
In this context,~$\Phi$ equals
\begin{equation*}
    \fct{\Phi}{x}{\Xbb}{\varepsilon(x)+M(x),}{\Rbb}
\end{equation*}
pictured on Figure~\ref{figure:application_dim2_Phi} (right).
Its minimizer is~$x^* \defequal (0,0)$, with~$\Phi(x^*) = 0$.
Assumptions~\ref{assumption:pof_minimal} and~\ref{assumption:pof_additional} hold.

Note that~$(M(x),T(x))$ remains intractable analytically.
Nevertheless, we observe that
\begin{equation*}
    \begin{array}{ccl}
    M(x) & = &
        \min\left\{M \in \Rbb_+ \ : \ I_x(M) \defequal I_x^1(M) \cap I_x^2(M) \cap I_x^3(M) \neq \emptyset\right\}, \\
    \left\{ T(x) \right\} & = &
        I_x\left(M(x)\right),
    \end{array}
\end{equation*}
so~$M(x)$ is obtainable by dichotomic search and the singleton~$\{T(x)\}$ follows.
In practice, running sufficiently many iterations of the dichotomic search provides an accurate approximation of~$M(x)$.
We define~$M^0_\mathrm{inf} \defequal \floor{M(x)}$ and~$M^0_\mathrm{sup} \defequal M^0_\mathrm{inf}+1$, and~$\ell \defequal 0$, and we iterate the dichotomic search on~$[M^\ell_\mathrm{inf},M^\ell_\mathrm{sup}]$ while~$M^\ell_\mathrm{sup}-M^\ell_\mathrm{inf} > 2^{-30}$.
This allows to define~$\hat{M}(x) \defequal \frac{1}{2}(M^\ell_\mathrm{inf}+M^\ell_\mathrm{sup}) \approx M(x)$, and then~$\hat{T}(x) \defequal \frac{1}{2}(\min I_x(\hat{M}(x)) + \max I_x(\hat{M}(x))) \approx T(x)$, and finally~$\hat{\gamma}(x) \defequal s_x(\hat{T}(x)) \approx \gamma(x)$.

A total of eight instances of the \cdsm are run, minimizing~$\hat{\Phi} \defequal \varepsilon + \hat{M}$ since we approximate~$M \approx \hat{M}$ for tractability.
Table~\ref{table:application_dim2_results} shows the convergence towards~$x^*$.
All instances return~$\norm{\widetilde{x}^*-x^*} \leq 9E^{-7}$, hence they provide~$\widetilde{y}^* \defequal \hat{\gamma}(\widetilde{x}^*)$ such that~$\norm{\widetilde{y}^*-y^*} \leq 9E^{-7}$ and~$\abs{\varphi(\widetilde{y}^*)-\varphi(y^*)} \leq 9E^{-7}$.
Five instances even return~$\norm{\widetilde{x}^*-x^*} \leq 5E^{-10}$.
These behaviours agree with Theorem~\ref{theorem:solving_P}, even though the problem we actually solve slightly differs from the true problem.

{\renewcommand\arraystretch{1.3}
\begin{table}[!ht]
    \centering
    \input{Figures/Dim2_3Vars_table.tex}
    \caption{Results for eight instances of Algorithm~\ref{algo:cdsm} in the context of Section~\ref{section:applications/dim2}. In this table,~$\hat{x}^k \defequal \textnorm{x^k}_\infty$.}
    \label{table:application_dim2_results}
\end{table}
}

\begin{figure}[!ht]
    \centering
    \includegraphics[width=0.8\linewidth]{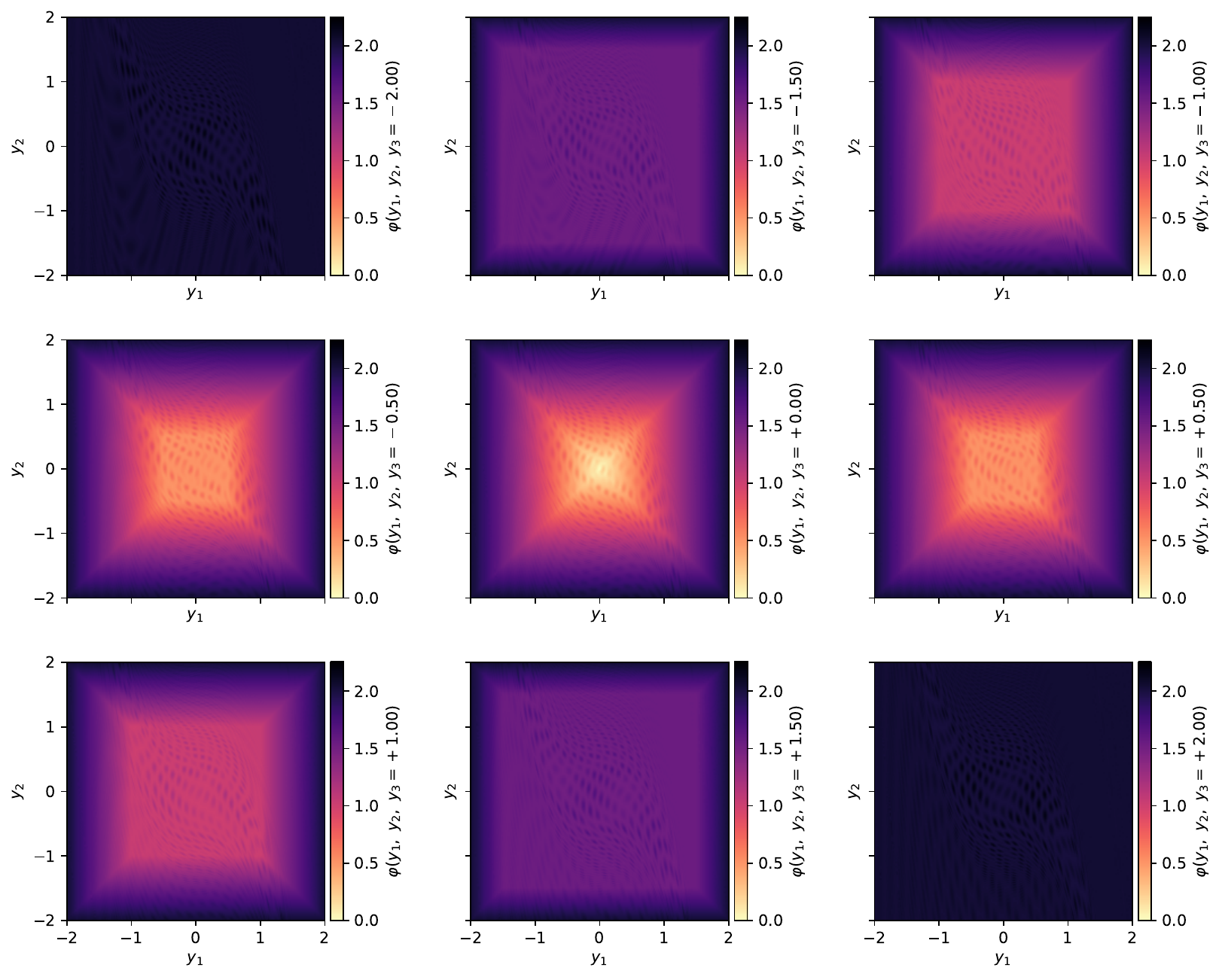}
    \caption{Function~$\varphi$ in Section~\ref{section:applications/dim2}, restriction over planes with constant third variable.}
    \label{figure:application_dim2_phi}
\end{figure}

\begin{figure}[!ht]
    \centering
    \includegraphics[width=0.8\linewidth, trim = 0 130 0 150,clip]{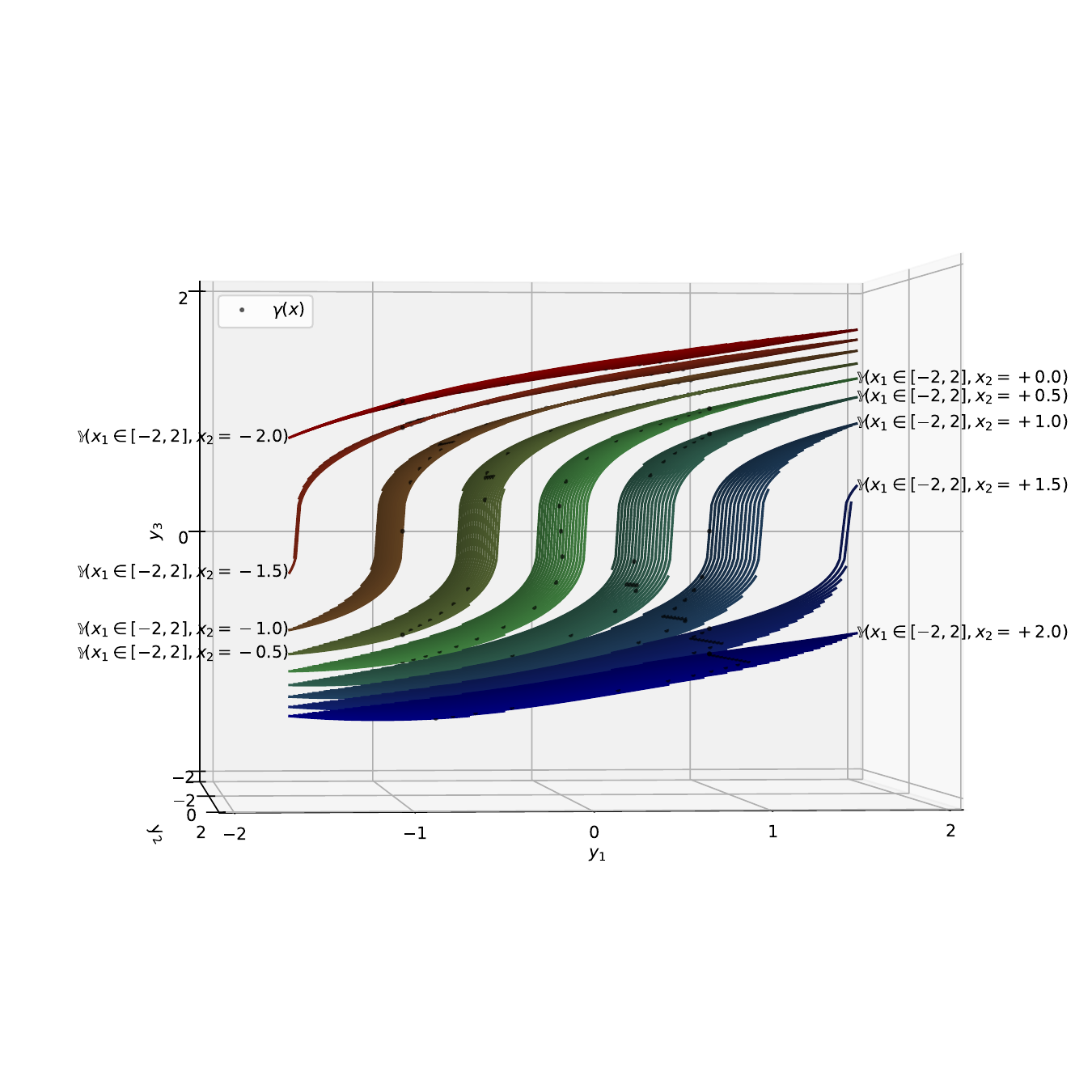}
    \caption{Partition sets of~$\Ybb(x)$ and locus of~$\gamma$ in the context of Section~\ref{section:applications/dim2}.}
    \label{figure:application_dim2_partition}
\end{figure}

\begin{figure}[!ht]
    \centering
    \includegraphics[width=0.32\linewidth, trim= 5 5 10 5, clip]{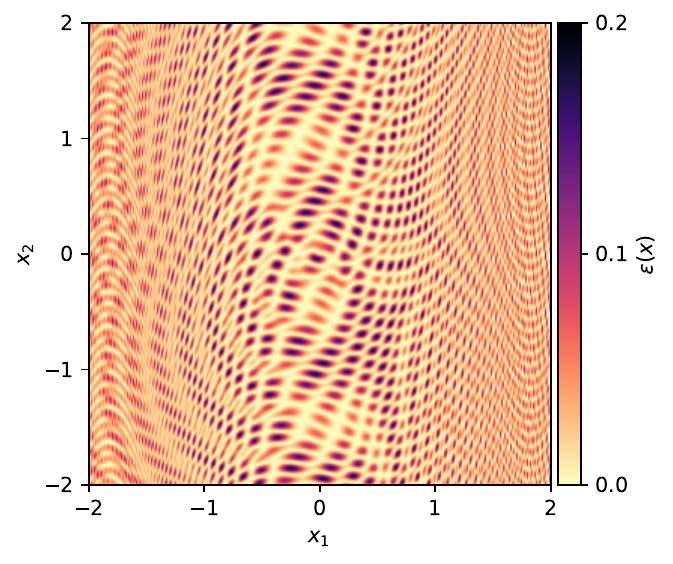}
    \hfill
    \includegraphics[width=0.32\linewidth, trim= 5 5 10 5, clip]{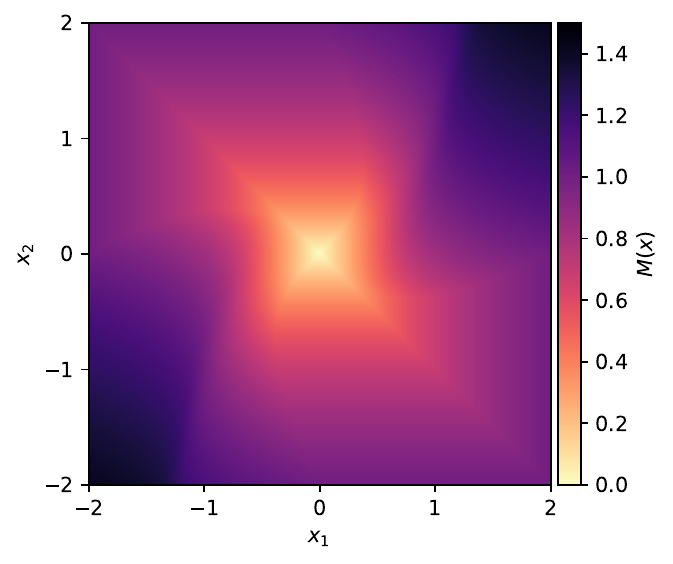}
    \hfill
    \includegraphics[width=0.32\linewidth, trim= 5 5 10 5, clip]{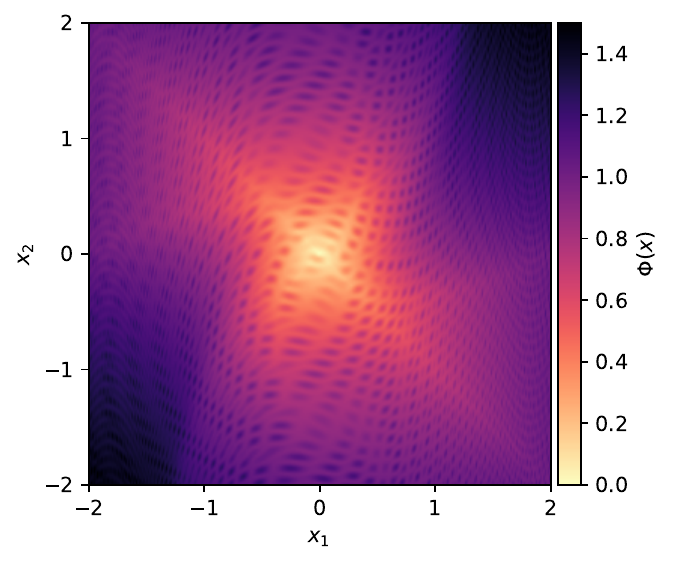}
    \caption{Functions (left)~$\varepsilon$, (centre)~$M$ and (right)~$\Phi$ in Section~\ref{section:applications/dim2}. In this section~$\Phi = \varepsilon + M$.}
    \label{figure:application_dim2_Phi}
\end{figure}

\section{Gain of performance provided by the \pom}
\label{section:applications_heavy}

In this section, we tackle composite greybox problems where the \pof allows for Reformulation~\eqref{problem:P_reformulated} with~$\dim(\Xbb)$ small (equal to the number of inputs of the blackboxes~$\sigma$ and~$\varepsilon$) even though~$\dim(\Ybb)$ is possibly large.
We show that the \pom is efficient in this context.
We alter the problems from Section~\ref{section:applications} so that~$\dim(\Ybb) \approx 100$ while~$\dim(\Xbb) \leq 10$.
Let us stress that our choice to consider~$\dim(\Ybb) \approx 100$ at most stands only for limited computational effort:~$\dim(\Ybb) \approx 100$ is already larger than what many \dfo algorithms tackle, and our examples could be adapted to match any~$0 < \dim(\Xbb) \leq \dim(\Ybb) \leq +\infty$ (we even have~$\dim(\Xbb) \leq 4 \ll \dim(\Ybb) = +\infty$ in~\cite{BoAuBo21PWMayerCost}).

In our examples below, computing~$\varphi(y)$,~$\gamma(x)$ and~$\Phi(x)$ is very fast for all~$y \in \Ybb$ and~$x \in \Xbb$.
This likely differs from real cases, where~$\varphi(y)$ may be computed by an expensive computer program and~$\gamma(x)$ may be obtained by numerically solving an optimization subproblem.
Thus, to establish a realistic comparison between the solvers, we proceed as follows.
For each problem, we consider that the cost to compute~$\varphi(y)$ is~$1$ unit, identical for all~$y \in \Ybb$, and that the computation of~$\Phi(x)$ costs~$1+\tau$ units, identical for all~$x \in \Xbb$ and where~$\tau \geq 0$ is fixed.
The value~$\tau$ captures the relative cost to evaluate~$\gamma$ versus those to evaluate~$\varphi$, so the computational cost of~$\Phi(x)$ models the additional cost required to first compute~$\gamma(x)$ before computing~$\varphi(\gamma(x))$.
Then, for each method, we track the sequence of all evaluated points.
This consists in evaluations of~$\varphi$ for solvers tackling Problem~\eqref{problem:P_initial} directly, and of~$\Phi$ for our method focusing on Reformulation~\eqref{problem:P_reformulated} instead.
Finally, for each method we plot the graph representing the lowest objective value found within a budget of~$c$ units, for all~$c \in \Rbb_+$.

We compare the \pom against the \dfo solvers \nomad~\cite{nomad4paper} and \prima~\cite{PRIMA_solver} solving Problem~\eqref{problem:P_initial} directly.
We observe that the \pom significantly outperforms both \nomad and \prima, unless~$\tau$ is prohibitively large.
Note that this does not discredit these solvers: as we discuss in Section~\ref{section:literature}, they are not designed for contexts suited for the \pof, and other \dfo techniques~\cite{GrRoViZh2015,LaMe24structureAware,LaMeWi2019,LaMeZh21ManifoldComposite,RoRo23ReducedSpaces} may also perform poorly in these contexts.
Similarly, our results do not cast light on our instance of the \cdsm in the \pom.
What makes the \pom efficient is the sheer dimension reduction provided by the \pof.

We follow the \textit{extreme barrier} approach~\cite{AuDe2006} that minimizes~$\barrier{\varphi}{\Omega}$ over~$\Ybb$, with relevant box constraints.
We run \nomad with default parameters, \prima with the \textsc{BObyQA} algorithm, and the \pom as in Remark~\ref{remark:algo_parameters}.
We give to \nomad and \prima six random starting points~$(y^0_\ell)_{\ell \in \llb1,6\rrb}$, and we give to our \cdsm the starting points~$(\chi(y^0_\ell))_{\ell \in \llb1,6\rrb}$.
In this section, we denote by~$\1 \defequal (1, \dots, 1)$.

\subsection{101-dimensional problem with mono-variable noise}
\label{section:applications_heavy/monovariable_heavy}

Consider an alteration of the example from Section~\ref{section:applications/monovariable} given by~$\Ybb = Y \defequal \Rbb^{101}$, and~$\Omega \defequal \Ybb$, and
\begin{equation*}
    \fct{\varphi}{y \defequal (y_i)_{i=0}^{100}}{\Ybb}
    {\norm{y-\sigma(y_0)}^2 + \varepsilon(y_0),}{\Rbb}
\end{equation*}
where~$\varepsilon(x) \defequal \abs{x}(1+\sin(\frac{2\pi}{x})^2)^{\frac{1}{2}}+\abs{\floorceil{x}}$ if~$x \in \Rbb^*$ and~$\varepsilon(0) \defequal 0$ as in Section~\ref{section:applications/monovariable}, and where
\begin{equation*}
    \fct{\sigma}{x}{\Rbb}{\left(~
        x ~,~
        \left(2\left(1+\frac{i-1}{5}\right)\floorceil{\frac{x}{i}}\right)_{i=1}^{25} ~,~
        \left(25\sin\left(\frac{i-25}{5}\pi x\right)\right)_{i=26}^{50} ~,~
        \left(x-\frac{10}{i}\right)_{i=51}^{75} ~,~
        \left(\frac{i}{10}\right)_{i=76}^{100}
    ~\right).}{\Rbb^{101}}
\end{equation*}
We consider~$\sigma$ and~$\varepsilon$ as blackboxes, so their expression is assumed to be unknown to all methods, but we remark that the global minimizer of~$\varphi$ is~$y^* \defequal \sigma(0)$, with~$\varphi(y^*) = 0$.
We define~$\Ybb(x) \defequal \{x\}\times\Rbb^{100}$ for all~$x \in \Xbb \defequal \Rbb$, since therefore~$\varphi_{|\Ybb(x)}$ admits an easily tractable global minimizer as
\begin{equation*}
    \gamma(x) \defequal \sigma(x), \quad \mbox{with} \quad \varphi(\gamma(x)) = \varepsilon(y_0).
\end{equation*}
Thus, Assumptions~\ref{assumption:pof_minimal} and~\ref{assumption:pof_additional} hold and~$\Phi = \varepsilon$ is as in Section~\ref{section:applications/monovariable}.
Figure~\ref{figure:application_monovariable_heavy} compares the three methods.

The \pom outperforms \nomad and \prima tackling Problem~\eqref{problem:P_initial} directly.
Our \cdsm closely approaches the global minimizer while \nomad and \prima remain in an exploratory phase that returns poor results, and the computational budget is comparable when~$\tau = 10$.

\begin{figure}[!b]
    \centering
    \includegraphics[width=0.9\linewidth]{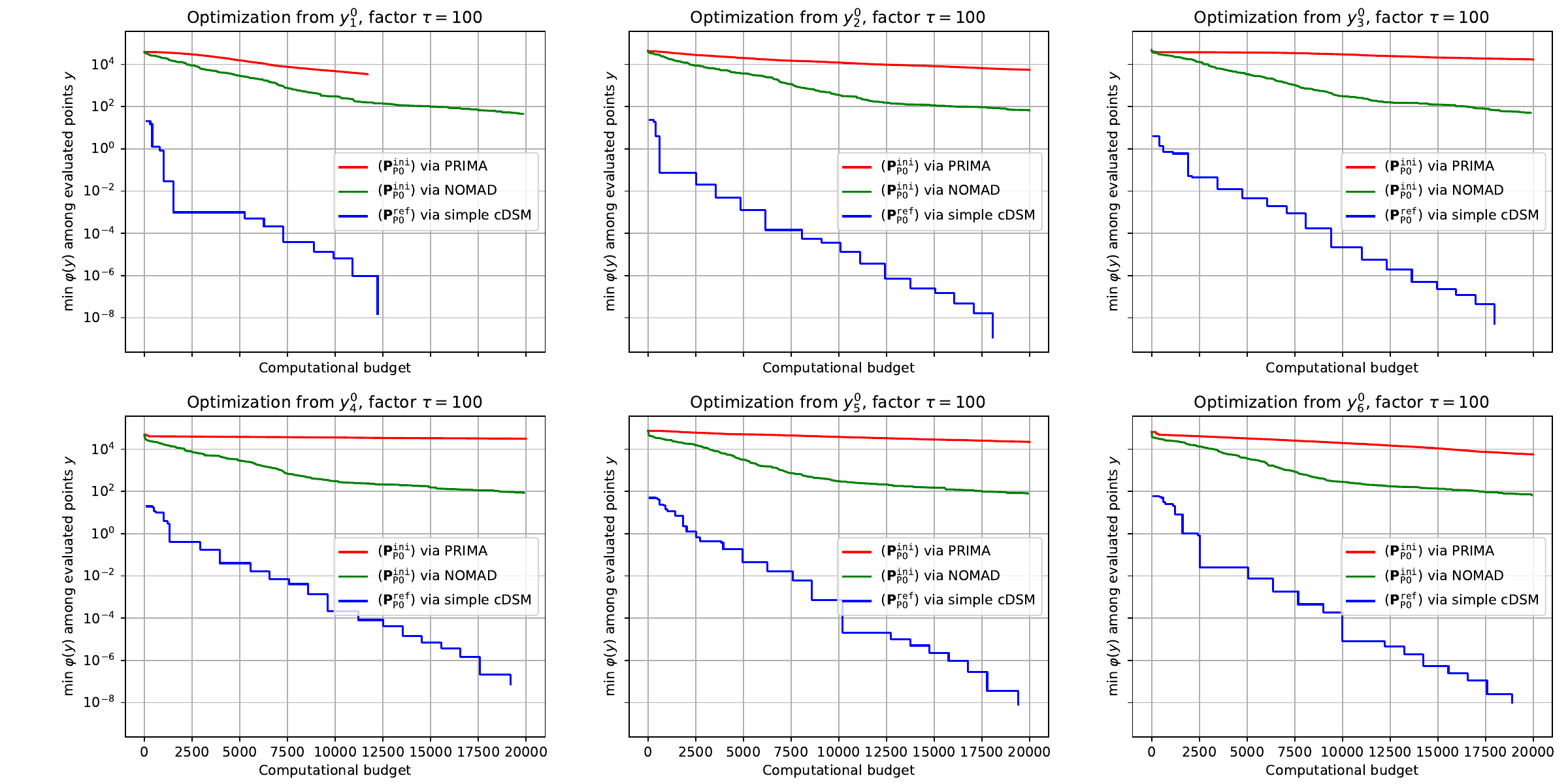}
    \caption{Comparison of the best solution found by each solver depending on the computational budget spent. The six starting points~$y^0_i$,~$i \in \llb1,6\rrb$, are chosen as $6$ observations of the random uniform independent distribution over~$[-30,30]^{101}$.}
    \label{figure:application_monovariable_heavy}
\end{figure}

\subsection{Radial noise in a 101-dimensional problem}
\label{section:applications_heavy/radial_heavy}

Consider an alteration of the problem from Section~\ref{section:applications/radial}, with~$\Ybb = Y = \Omega \defequal \Rbb_+\times[0,2\pi[^{100}$ and
\begin{equation*}
    \fct{\varphi}{\left(r,(\theta_i)_{i=1}^{100}\right) }{\Ybb}
    {\dfrac{\sqrt{r}}{100}\dsum{i=1}{100}\sin\left(\dfrac{\theta_i-\sigma_i(r)}{2}\right)^2 + \varepsilon(r),}{\Rbb}
\end{equation*}
where~$\sigma(x) \defequal (2\pi\log_{i+1}(x))_{i=1}^{100}$ for all~$x > 0$ and~$\sigma(0) \defequal 0\1$, and~$\varepsilon(x) \defequal \sqrt{\textabs{x^2-2}} + \frac{\sin(10\pi(x-\sqrt{2}))^2}{10}$ for all~$x \in \Rbb_+$, are two blackboxes.
The global minimizer of~$\varphi$ is~$y^* \defequal (\sqrt{2},\sigma(\sqrt{2}))$, with~$\varphi(y^*) = 0$.
We define~$\Ybb(x) \defequal \{x\}\times[0,2\pi[^{100}$ for all~$x \in \Xbb = X \defequal \Rbb_+$.
Then,~$\varphi_{|\Ybb(x)}$ is a smooth function of~$\theta$ while~$r = x$ is fixed, and its global minimizer is
\begin{equation*}
    \gamma(x) \defequal (x,\sigma(x)\sim2\pi),
    \quad \mbox{with} \quad
    \varphi(\gamma(x)) = \varepsilon(x),
\end{equation*}
where~$\sigma(x) \sim 2\pi$ denotes the component-wise residual of~$\sigma(x)$ modulo~$2\pi$.
Thus, Assumptions~\ref{assumption:pof_minimal} and~\ref{assumption:pof_additional} hold and~$\Phi = \varepsilon$ is as in Section~\ref{section:applications/radial}.
The comparison between the strategies is provided in Figure~\ref{figure:application_radial_heavy}.

The results are similar to those in Section~\ref{section:applications_heavy/monovariable_heavy}.
Even by considering~$\tau = 100$, the \pof provides a gain of performance.
Within a comparable budget, the \cdsm in our \pom converges towards the global minimizer, while \nomad and \prima solving the original problem remain in an exploratory phase that returns irrelevant points.
The very early stop of \nomad in three cases even shows that \nomad sometimes interrupts itself early on its poor incumbent solution.

\begin{figure}[!h]
    \centering
    \includegraphics[width=\linewidth]{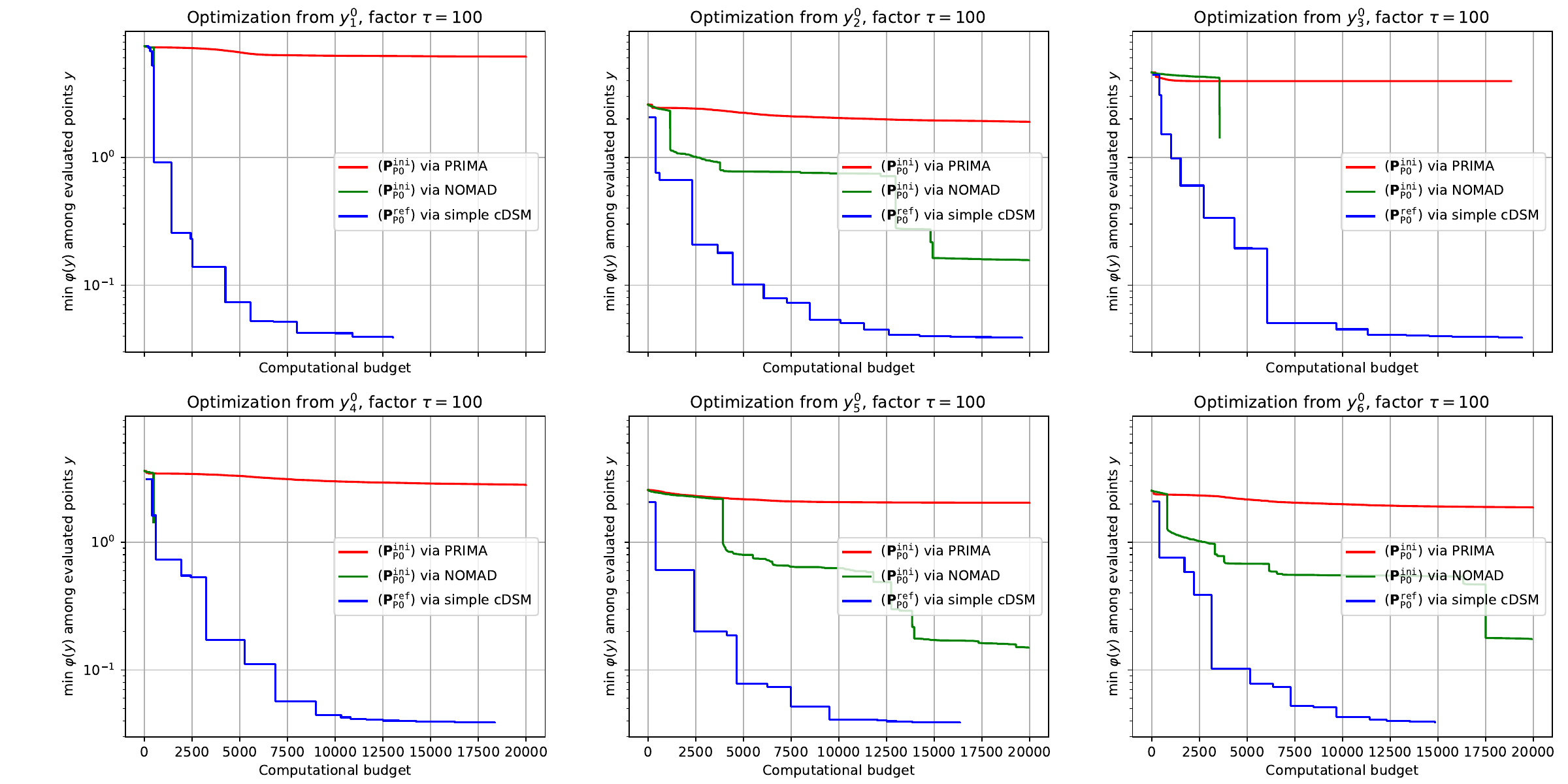}
    \caption{Comparison of the best solution found by each solver depending on the computational cost spent. The six starting points~$y^0_i$,~$i \in \llb1,6\rrb$, are chosen as $6$ observations of the random uniform independent distribution over~$[0,2\pi]^{101}$.}
    \label{figure:application_radial_heavy}
\end{figure}

\subsection{Noise affected by a nonlinear combination of all 100 variables}
\label{section:applications_heavy/nonlinear_heavy}

Consider~$\Ybb = Y \defequal \Rbb^{100}$, and~$\Omega \defequal \{y \in \Ybb : \forall i \in \llb1,100\rrb,~ y_{i+1} \geq y_i > 0\}$, and
\begin{equation*}
    \fct{\varphi}{y \defequal (y_i)_{i=1}^{100}}{\Ybb}
    {\dsum{(\ell,p)\in\llb0,19\rrb^2}{}\ln\left(1+\left(\dfrac{\pi_\ell(y)}{\pi_p(y)}-1\right)^2\right) + \varepsilon\left(f(y)\right),}{\Rbb}
\end{equation*}
where for all~$y \in \Ybb$,~$\pi_\ell(y) \defequal \prod_{i=5\ell+1}^{5(\ell+1)}y_i$ for each~$\ell \in \llb0,19\rrb$ and~$f(y) \defequal \frac{1}{5}\sum_{\ell=0}^{19}\pi_\ell(y)$ have explicit definition, and where the blackbox~$\varepsilon$ satisfies~$\varepsilon(x) \defequal \exp(\frac{1}{z-4}) + \frac{1}{5}\sqrt{\textabs{z-4}}$ if~$z \neq 4$ and~$\varepsilon(4) \defequal +\infty$.
Here,~$\varphi$ has no global minimizer, but~$\inf \varphi(\Omega) = 0$ and~$y^* \defequal \1$ is a generalized global solution (even though~$\varphi(y^*) = +\infty$).
For all~$x \in \Xbb = X \defequal \Rbb$, we define~$\Ybb(x) \defequal \{y \in \Ybb : f(y) = x\}$.
Then,~$\varphi_{|\Ybb(x)}$ is a smooth function since~$\varepsilon(f(y))$ is constant.
Moreover, its minimizer is analytically obtainable as follows.
For all~$x \in \Xbb$,~$\varphi_{|\Ybb(x)}(y) \geq 0$ and equals~$0$ when all~$(\pi_\ell(y))_{\ell \in \llb0,19\rrb}$ are equal.
Under the constraint~$y \in \Omega$, equalling all~$(\pi_\ell(y))_{\ell \in \llb0,19\rrb}$ implies that~$y = \alpha\1$ for some~$\alpha \in \Rbb$.
Then,~$y \in \Ybb(x)$ leads to~$4\alpha^5 = x$.
The value of~$\alpha$ follow, and we get that
\begin{equation*}
    \gamma(x) \defequal \sqrtp{5}{\frac{x}{4}}\1,
    \quad \mbox{with} \quad
    \varphi(\gamma(x)) = \varepsilon(x).
\end{equation*}
Thus,~$\Phi = \varepsilon$ as in Section~\ref{section:applications/nonlinear}, and Assumptions~\ref{assumption:pof_minimal} and~\ref{assumption:pof_additional} hold.
Figure~\ref{figure:application_nonlinear_heavy} compares our strategies.

Similarly to Sections~\ref{section:applications_heavy/monovariable_heavy} and~\ref{section:applications_heavy/radial_heavy}, Reformulation~\eqref{problem:P_reformulated} is a \dfo problem fixing the value of a single combination of all the variables, and the \pof fixes the~$100$ variables of the original problem accordingly.
Even with~$\tau = 100$, the \pof provides a gain of performance.
\nomad and \prima solving the original problem both fail to significantly improve their initial incumbent solution and, in the second case, both solvers behave as if their initial incumbent solution cannot be improved.
Within a comparable budget, the \cdsm in our \pom converges towards the global minimizer on all but one case.
Yet, in this case, first we observe that this solution remains~$1000$ times better than those returned by \nomad and \prima, and second we claim that the poor performance results from our naive implementation: even if Reformulation~\eqref{problem:P_reformulated} is $1$-dimensional, it is preferable to use a globalization strategy because the interval of possibly relevant values is large.
Our instance lacks one since it lacks a \search step.

\begin{figure}[!h]
    \centering
    \includegraphics[width=\linewidth]{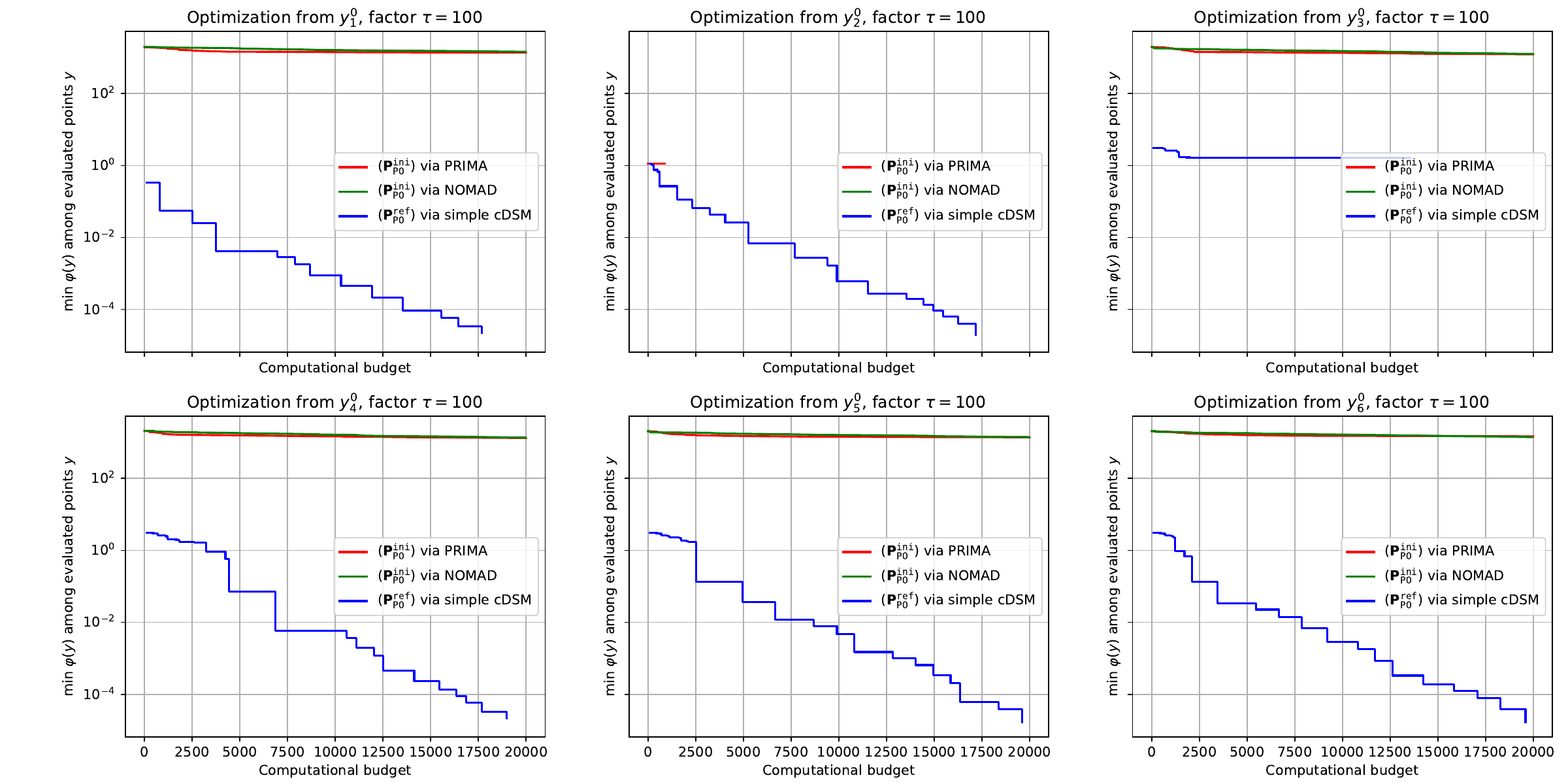}
    \caption{Comparison of the best solution found by each solver depending on the computational cost spent. The starting points are chosen as~$y^0_1 \defequal (i/100)_{i=1}^{100}$,~$y^0_2 \defequal 0.5\1$,~$y^0_3 \defequal (\frac{2i}{100})_{i=1}^{100}$, and~$y^0_i$, for each~$i \in \llb4,6\rrb$, is chosen randomly but such that for each~$j \in \llb1,100\rrb$, the $j^{th}$ component of~$y^0_i$ lies in~$[\frac{j-1}{100},\frac{j}{100}]$.}
    \label{figure:application_nonlinear_heavy}
\end{figure}

\subsection{Ten-dimensional noise with non-analytical oracle and 100 variables}
\label{section:applications_heavy/dim2_heavy}

Consider~$\Ybb = Y \defequal \Rbb^{100}$, and~$\Omega \defequal \Ybb$, and the function
\begin{equation*}
    \fct{\varphi}{y}{\Rbb^{100}}
    {\norm{y}_1 + \varepsilon(f(y)),}{\Rbb}
\end{equation*}
where the function~$f$ (with definition assumed accessible) is defined as follows, with~$\Xbb = X \defequal \Rbb^{10}$,
\begin{equation*}
    \fct{f}{(y_i)_{i=1}^{100}}{\Ybb}
    {\left(g_1(y_{10})-\dsum{i=1}{9}y_i,~ \dots,~ g_{10}(y_{100})-\dsum{i=91}{99}y_i\right),}{\Xbb}
\end{equation*}
with~$g_j(z) \defequal (z+(1+\frac{j}{10})^z-1)$ for all~$z \in \Rbb$ and all~$j \in \Nbb$, and where the blackbox~$\varepsilon$ is defined through
\begin{equation*}
    \fct{\varepsilon}{z = (z_j)_{j=1}^{10}}{\Xbb}
    {\left(
        \frac{\sin5\pi\left(z_2-z_1^3\right)}{5} +
        \frac{\sin6\pi\left(z_4-e^{-z_2-z_3}+1\right)}{7} +
        \frac{\sin7\pi\sqrt{z_5^2+z_6^2+z_7^2}}{11} +
        \frac{\sin8\pi z_8z_9z_{10}}{13}
     \right)^2.}
     {\Rbb}
\end{equation*}
The global minimizer of~$\varphi$ is~$y^* \defequal 0\1$, with~$\varphi(y^*) = 0$.
We partition~$\Ybb$ via~$\Ybb(x) \defequal \{y \in \Ybb : f(y) = x\}$.
Then, for all~$x \in \Xbb$, the constraint~$y \in \Ybb(x)$ fixes all variables~$(y_{10j})_{j=1}^{10}$ to~$g_j^{-1}(x_j+\sum_{i=10(j-1)+1}^{10(j-1)+9}y_i)$, and~$\varepsilon(f(y))$ to the constant~$\varepsilon(x)$, so~$\gamma(x)$ minimizing~$\varphi_{|\Ybb(x)}$ is the solution to the problem
\begin{equation*}
    \problemoptimfree{\minimize}{y \in \Ybb(x)}{\varepsilon(x) + \dsum{j=1}{10}\left(\dsum{i=10(j-1)+1}{10(j-1)+9}\abs{y_i} + \abs{g_j^{-1}\left(x_j+\dsum{i=10(j-1)+1}{10(j-1)+9}y_i\right)}\right).}
\end{equation*}
Then, observing that~$\frac{\mathrm{d}g_j^{-1}}{\mathrm{d}z}(z) \in {]}0,1{[}$ for all~$(j,z) \in \Nbb^* \times \Rbb$, a sensitivity analysis with respect to each sum~$\sum_{i=10(j-1)+1}^{10(j-1)+9}\abs{y_i}$ shows that
\begin{equation*}
    \gamma(x) \defequal (\gamma_i(x))_{i=1}^{100}
    \quad \mbox{where} \quad
    \gamma_i(x) \defequal \left\{
        \begin{array}{ll}
        0             & \mbox{if}~ i \notin 10\Nbb, \\
        g_j^{-1}(x_j) & \mbox{if}~ i = 10j,~ j \in \llb1,10\rrb,
        \end{array}
    \right.
\end{equation*}
with~$\varphi(\gamma(x)) = \varepsilon(x)+\dsum{j=1}{10}\abs{g_j^{-1}(x_j)}$.
As a result,~$\Phi$ equals
\begin{equation*}
    \fct{\Phi}{x \defequal (x_j)_{j=1}^{10}}{\Xbb}
    {\varepsilon(x)+\dsum{j=1}{10}\abs{g_j^{-1}(x_j)}.}{\Rbb}
\end{equation*}
Its global minimizer is~$x^* \defequal 0\1$, with~$\Phi(x^*) = 0$.
Assumptions~\ref{assumption:pof_minimal} and~\ref{assumption:pof_additional} hold.

Unfortunately,~$\gamma$ admits no analytical expression since, for all~$j \in \llb1,10\rrb$,~$g_j^{-1}$ cannot be expressed with elementary functions.
However, since~$g_j^{-1}(x_j)$ solves the equation~$(z+(1+\frac{j}{10})^z-1)-x_j = 0$ (with variable~$z \in \Rbb$), we approximately solve this equation using a dichotomic search to obtain an approximate solution~$\hat{g_j}^{-1}(x_j) \approx g_j^{-1}(x_j)$.
As a result, for all~$x \in X$, we approximate
\begin{equation*}
    \gamma(x) = (\gamma_i(x))_{i=1}^{100}\approx \hat{\gamma}(x) \defequal (\hat{\gamma_i}(x))_{i=1}^{100}
    \quad \mbox{where} \quad
    \hat{\gamma_i}(x) \defequal \left\{
        \begin{array}{ll}
        0                     & \mbox{if}~ i \notin 10\Nbb, \\
        \hat{g_j}^{-1}(x_j)   & \mbox{if}~ i = 10j,~ j \in \llb1,10\rrb,
        \end{array}
    \right.
\end{equation*}
and
\begin{equation*}
    \Phi \approx \hat{\Phi}
    \quad \mbox{where} \quad
    \fct{\hat{\Phi}}{x \defequal (x_j)_{j=1}^{10}}{\Xbb}
    {\varepsilon(x)+\dsum{j=1}{10}\abs{\hat{g_j}^{-1}(x_j)}.}{\Rbb}
\end{equation*}

Figure~\ref{figure:application_dim2_heavy} shows a comparison between the different strategies, where for tractability we actually minimize~$\hat{\Phi}$ instead of~$\Phi$.
even though the problem we actually solve slightly differs from the true problem, we observe results similar to those in Sections~\ref{section:applications_heavy/monovariable_heavy},~\ref{section:applications_heavy/radial_heavy} and~\ref{section:applications_heavy/nonlinear_heavy}.
Reformulation~\eqref{problem:P_reformulated} is a $10$-dimensional \dfo problem fixing the value of the important combinations of all the variables, and the \pof highlights how to fix the~$100$ variables of Problem~\eqref{problem:P_initial} accordingly.
Using our \cdsm, the \pom returns solutions~$100$ times better than those returned by \nomad and~$10^7$ times better than those returned by \prima, and the computational budget is similar if we allow~$\tau \approx 10$ at most.
The value~$\tau = 10$ is low, but we could consider a higher value by using a better algorithm.
For example, a test with \nomad solving the reformulated problem (instead of our naive instance of \cdsm) provides similar graphs with~$\tau = 30$, represented on Figure~\ref{figure:application_dim2_heavy_nomad}.

\begin{figure}[!ht]
    \centering
    \includegraphics[width=\linewidth]{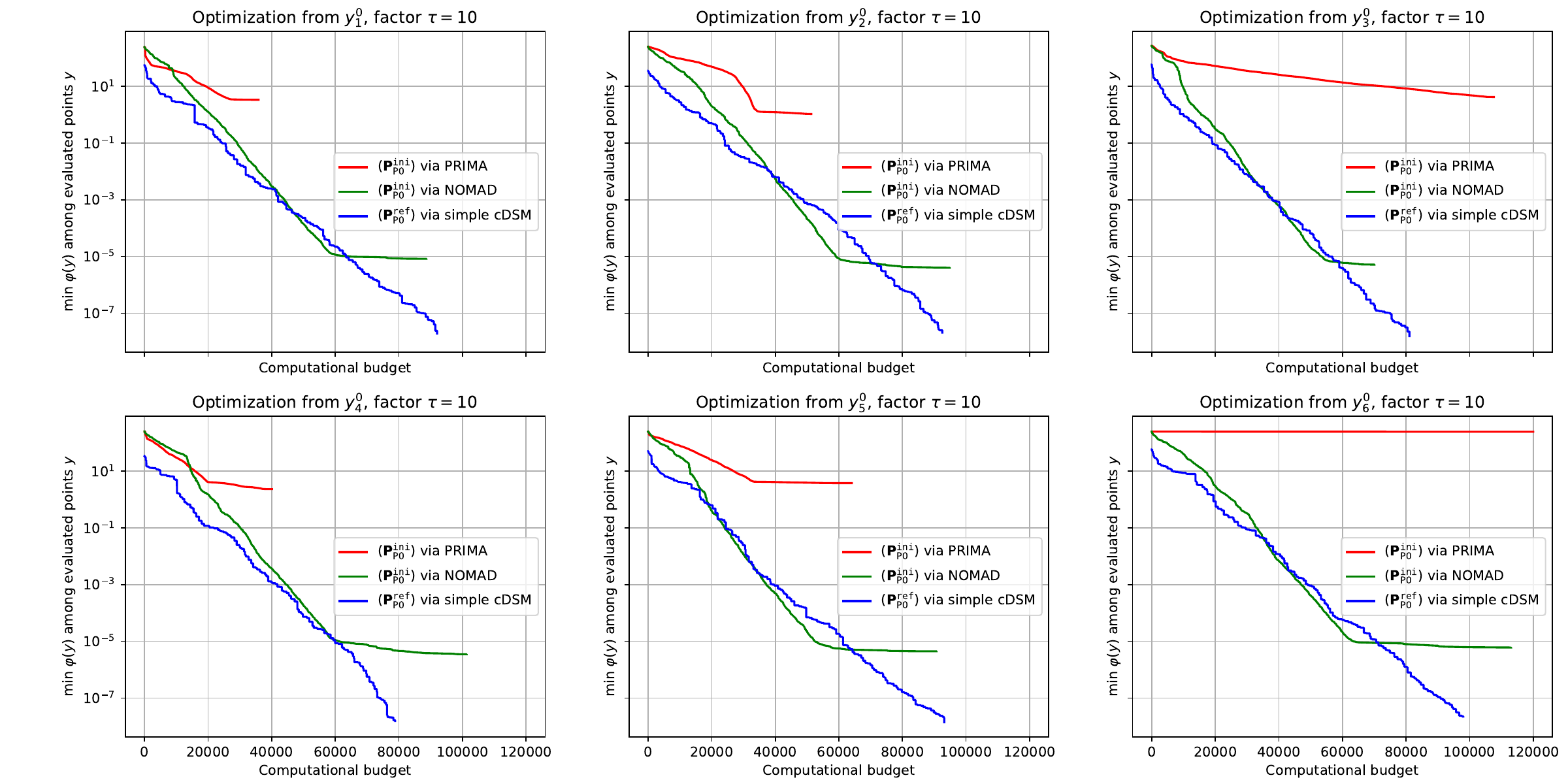}
    \caption{Comparison of the best solution found by each solver depending on the computational cost spent. The six starting points~$y^0_i$,~$i \in \llb1,6\rrb$, are chosen as $6$ observations of the random uniform independent distribution over~$[-5,5]^{100}$.}
    \label{figure:application_dim2_heavy}
\end{figure}

\begin{figure}[!ht]
    \centering
    \includegraphics[width=\linewidth]{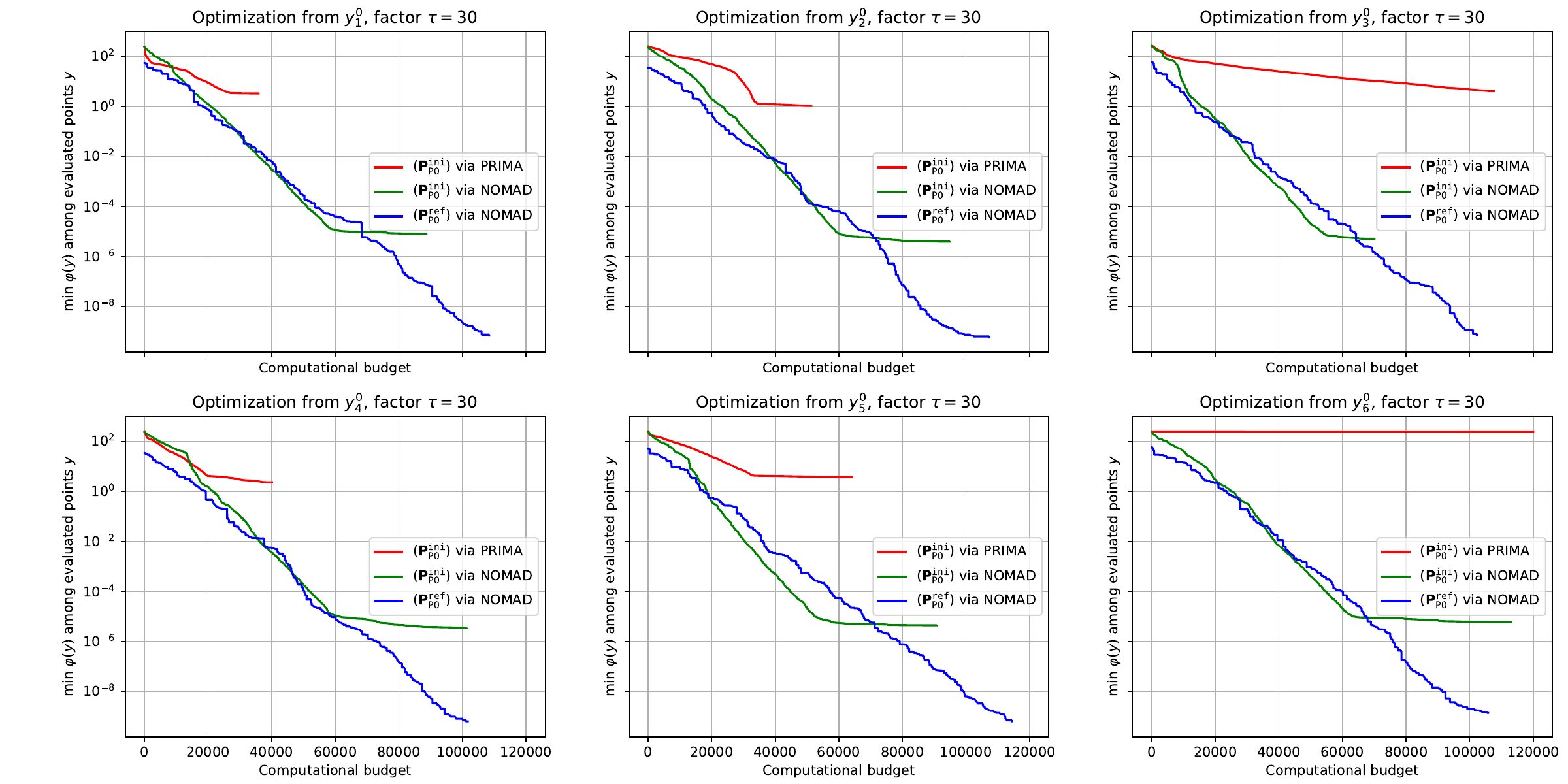}
    \caption{Additional gain provided by solving the reformulated problem using a more efficient solver. Despite the higher value for~$\tau$ compared to the results on Figure~\ref{figure:application_dim2_heavy}, the graphs are similar.}
    \label{figure:application_dim2_heavy_nomad}
\end{figure}

\section{General discussion}
\label{section:discussion}

We conclude this work with a discussion on our contribution.
Section~\ref{section:discussion/comments} summarizes the most important aspects of the \pof, and Section~\ref{section:discussion/perspectives} lists some routes for improvements.

\subsection{Comments on the \pof and the \pom}
\label{section:discussion/comments}

The \pof is a reformulation framework that applies to several classes of problems, including infinite-dimensional, nonconvex, discontinuous and unbounded problems.
The idea is to pinpoint the difficult aspect in Problem~\eqref{problem:P_initial}, and to partition the variables space so that this difficulty is transparent in each partition set.
Then, the \pof defines the oracle function~$\gamma$ (assumed computable) that maps any partition set index to a solution to the (simpler) problem restricted to the partition set.
Thus, the \pof is not a \dfo framework by itself, but rather a way to reformulate a problem with localized difficulties so that Reformulation~\eqref{problem:P_reformulated} tackles only these difficulties.
Then, Reformulation~\eqref{problem:P_reformulated} is manageable using \dfo algorithms, as we formalize in the \pom.
This method is theoretically grounded, as formalized by Theorem~\ref{theorem:solving_P}, since Theorem~\ref{theorem:connections_solutions} ensures that solutions to Reformulation~\eqref{problem:P_reformulated} lead to solutions to Problem~\eqref{problem:P_initial}, and Proposition~\ref{proposition:cDSM_applied_to_Phi} ensures that Algorithm~\ref{algo:covering_framework} indeed solves Reformulation~\eqref{problem:P_reformulated}.

A stringent practical aspect of the \pof is the oracle function~$\gamma$.
For all~$x \in X$, Subproblem~\eqref{problem:P_subproblem} must have a computable global solution.
Ensuring this requirement is challenging and likely problem-dependent, but some applications already fit in this framework (as discussed in Sections~\ref{section:intro/motivation},~\ref{section:example_optimal_control},~\ref{section:applications} and~\ref{section:applications_heavy}), and we are working on possible ways to significantly relax this requirement (see Section~\ref{section:discussion/perspectives}).
In some cases,~$\gamma$ is analytically tractable, but the \pof does not need that~$\gamma$ has a closed-form expression.

For theoretical considerations, our main contributions are Theorems~\ref{theorem:connections_solutions} and~\ref{theorem:solving_P}, and Assumptions~\ref{assumption:pof_minimal} and~\ref{assumption:pof_additional} ensuring their validity.
Yet, the proof of Theorem~\ref{theorem:connections_solutions} actually shows that our analysis is not tight.
First, Theorem~\refbis{theorem:connections_solutions}{global} holds under no assumption except the applicability of the \pof, and Theorem~\refbis{theorem:connections_solutions}{local} holds as soon as~$\chi$ is continuous.
Consequently, neither Assumptions~\ref{assumption:pof_minimal} and~\ref{assumption:pof_additional} nor Algorithm~\ref{algo:covering_framework} are required when we may compute analytically a (global or local) solution to Reformulation~\eqref{problem:P_reformulated}.
This scenario arises in some nontrivial cases, as in Section~\ref{section:example_optimal_control}.
Second, we require that~$\Ybb$ is a Banach space in the whole paper, but actually our analysis requires only for~$\Ybb$ to be complete, and this is needed only for Theorems~\refbis{theorem:connections_solutions}{generalized} and~\ref{theorem:solving_P}.
As a consequence, it is even not mandatory to ensure the completeness of~$\Ybb$ in Theorems~\refbis{theorem:connections_solutions}{global} and~\refbis{theorem:connections_solutions}{local}.
Hence, in Section~\ref{section:example_optimal_control}, we could use any norm on~$\Acal\Ccal$.
Third, Theorem~\ref{theorem:solving_P} is stronger if we moreover assume that~$\varphi$ is lower semicontinuous and~$\Omega$ is closed.
In this case, Theorem~\ref{theorem:solving_P} ensures that~$\gamma(x^*)$ is a local solution to Problem~\eqref{problem:P_initial} for all~$x^*$.
In our analysis where this additional assumption does not hold, Theorem~\ref{theorem:solving_P} only ensures that~$\gamma(x^*)$ is a generalized local solution, as the subsequent claim ensuring that~$\gamma(x^*)$ is a (usual) local solution relies on a sufficient but not necessary condition.
Sections~\ref{section:applications/monovariable} and~\ref{section:applications/nonlinear} illustrate cases where~$\varphi$ is not semicontinuous and~$\gamma(x^*)$ is, respectively, a local solution nevertheless, or only a generalized local solution as claimed.

We did not find generic rules that ensure Assumptions~\ref{assumption:pof_minimal} and~\ref{assumption:pof_additional}.
Yet, they are likely met in practice and not very demanding, as we discuss in Remarks~\ref{remark:discussion_assumption_minimal} and~\ref{remark:discussion_assumption_additional}.
Moreover, Assumptions~\ref{assumption:pof_minimal} and~\ref{assumption:pof_additional} become tractable in some cases, as in Section~\ref{section:applications}.
Indeed, they hold if~$\Xbb$ admits a partition into ample continuity sets of both~$\sigma$ and~$\varepsilon$ (a light assumption similar to~\cite{AuBoBo24Covering}), and~$\widetilde{\varphi}$ is piecewise uniformly continuous and bounded below with bounded level sets.
Finally, if the assumptions do not hold, the \pom becomes heuristic, but none of our experiments in~\cite{BouchetPhD,BoAuBo21PWMayerCost} exhibits a drop in the performance of the method.

Theorems~\ref{theorem:connections_solutions} and~\ref{theorem:solving_P} assume that~$\gamma$ returns exact global solutions to Subproblem~\eqref{problem:P_subproblem} for all~$x \in \Xbb$ and Algorithm~\ref{algo:covering_framework} runs \textit{ad infinitum} to solve Reformulation~\eqref{problem:P_reformulated} exactly.
None of these assumptions hold true in practice.
Nevertheless, our numerical experiments show that the \pom remains efficient when~$\gamma$ returns approximations of global solutions and Algorithm~\ref{algo:covering_framework} is interrupted early.
In Sections~\ref{section:applications/dim2} and~\ref{section:applications_heavy/dim2_heavy},~$\gamma$ is defined as the output of a numerical method solving Subproblem~\eqref{problem:P_subproblem}, for all~$x \in X$, with high confidence to approximate a global solution.
We also observe in~\cite{BoAuBo21PWMayerCost} that the \pom remains effective when~$\gamma(x)$ consists in the output of some solver tackling Subproblem~\eqref{problem:P_subproblem}, even though the solver possesses only guarantees to identify a point satisfying necessary optimality conditions.
Then, such approximations are sufficient in practice to reap the benefits of the \pof, while our theory justifies the methodology when~$\gamma$ returns global solutions exactly and Algorithm~\ref{algo:covering_framework} runs \textit{ad infinitum}.

Our \pom solves Reformulation~\eqref{problem:P_reformulated} using a \dfo algorithm with a \covering step.
We consider the \cdsm (Algorithm~\ref{algo:cdsm}) in our numerical experiments, but any \dfo algorithm may be enhanced with a \covering step, as discussed in Section~\ref{section:results/algorithm} and~\cite{AuBoBo24Covering}.
The choice of the most appropriate \dfo algorithm is likely to depend on~$\Phi$.
In general, Reformulation~\eqref{problem:P_reformulated} has little specificity that differentiates it from a generic blackbox optimization problem, so there is no reason to favour a \dsm over other classes of methods.
In particular, if~$\dim(\Xbb)$ is too large, a \dsm may not be the most suited method to solve Reformulation~\eqref{problem:P_reformulated}.
However, to our best knowledge, no popular \dfo solver implements algorithms with a \covering step yet.
Then, all fast solvers we may consider are only guaranteed to return a point satisfying necessary optimality conditions (instead of local optimality), so they are not compatible with Theorem~\ref{theorem:solving_P} yet.
Nevertheless,~\cite{AuBoBo24Covering} remarks that the \covering step is likely optional in practice.
We second this observation in Section~\ref{section:applications_heavy/dim2_heavy}, as solving Reformulation~\eqref{problem:P_reformulated} using \nomad provides better results than our naive \cdsm, even though \nomad does not implement a \covering step.

The usual case of application of the \pom is when~$0 < \dim(\Xbb) \ll \dim(\Ybb)$.
Moreover,~$\dim(\Xbb) \leq 50$ is desirable when solving Reformulation~\eqref{problem:P_reformulated} with the \cdsm, for performance reason~\cite[Section~1.4]{AuHa2017}.
We stress that the performance of the \pom depends only on~$\dim(\Xbb)$ and on the computation time related to~$\gamma$.
Hence,~$\dim(\Ybb)$ does not directly matter, so it may be possibly infinite as in~\cite{BoAuBo21PWMayerCost}.
Moreover, the more efficient the \dfo solver solving Reformulation~\eqref{problem:P_reformulated} is, the larger the constant~$\tau$ from Section~\ref{section:applications_heavy} (the relative cost to evaluate~$\gamma$ versus those to evaluate~$\varphi$) can be.
As a rule of thumb, if we can estimate the number~$\Nbf^\texttt{ref}$ of points in~$\Xbb$ required by the \dfo solver to solve Reformulation~\eqref{problem:P_reformulated}, and~$\Nbf^\texttt{ini}$ of points in~$\Ybb$ required by a dedicated solver to solve Problem~\eqref{problem:P_initial} directly, then the \pom is useful when~$\Nbf^\texttt{ref}\tau < \Nbf^\texttt{ini}$.
We also remark that the way the partition is indexed influences the shape of~$\Phi$.
The solving process may be substantially simplified by a re-parametrization of~$\Xbb$ that avoid concentrations of discontinuities or mitigate local variations of~$\Phi$.

\subsection{Perspectives for future work}
\label{section:discussion/perspectives}

Two lines of improvements appear for extending the \pof.
The first concerns the oracle function~$\gamma$.
It is unrealistic to assume that a global solution to each subproblem is computable.
A future work will define~$\gamma$ as, for all~$x \in X$, an approximation of a local solution to Subproblem~\eqref{problem:P_subproblem}.
This would be representative of most practical cases.
We test this idea in~\cite[Chapter~7]{BouchetPhD} in a context of optimal control problems, by defining, for all~$x \in X$,~$\gamma(x)$ as the output of a well-suited numerical method (chosen according to acknowledged guidelines from the literature) solving Subproblem~\eqref{problem:P_subproblem}.
This approach seems promising and works well in our experiments, although the theory is technical so we are still working to simplify it.
We may also determine if, given~$x^* \in X$ satisfying necessary optimality conditions for Reformulation~\eqref{problem:P_reformulated},~$\gamma(x^*)$ satisfies necessary optimality conditions for Problem~\eqref{problem:P_initial}.
This would allow tackling Reformulation~\eqref{problem:P_reformulated} with any \dfo algorithm that returns points satisfying only necessary optimality conditions, instead of imposing the \covering step to ensure local optimality.

Another possible improvement concerns the solving process of Reformulation~\eqref{problem:P_reformulated}.
Currently, Reformulation~\eqref{problem:P_reformulated} is defined according to an \textit{extreme barrier}~\cite{AuDe2006}, since~$\Phi(x) \defequal +\infty$ for all~$x \in \Xbb$ such that Subproblem~\eqref{problem:P_subproblem} is infeasible.
We observed in~\cite[Chapter~7]{BouchetPhD} that an extreme barrier may lead to an important number of evaluated points returning the value~$+\infty$.
A \textit{progressive barrier}~\cite{AuDe09a} may be more efficient.
We have initiated some research to associate an \textit{infeasibility metric} to all points~$x \in \Xbb$ and use it in the progressive barrier.
For all~$x \in \Xbb$, our infeasibility associated to~$x$ with respect to Reformulation~\eqref{problem:P_reformulated} is the infimum over all~$y \in \Ybb(x)$ of the infeasibility of~$y$ with respect to Subproblem~\eqref{problem:P_subproblem}.
In all the problems we tested in~\cite[Chapter~7]{BouchetPhD}, solving Reformulation~\eqref{problem:P_reformulated} is significantly easier with such a progressive barrier than with an extreme barrier.

Beyond these two main lines of improvement, we also plan to develop more advanced algorithms for the \pom.
This may be seen as a prototypical hybrid method using either a \dfo algorithm (to solve Reformulation~\eqref{problem:P_reformulated}) and other classes of algorithms (to solve Subproblem~\eqref{problem:P_subproblem} for any~$x \in X$).
Roughly speaking, we may solve Problem~\eqref{problem:P_initial} by optimizing jointly the partition set to consider and the point to consider into the set.
We may also learn or dynamically adapt the partition and its indexation to improve the behaviour of~$\Phi$ and the computational cost of~$\gamma$.

Finally, Section~\ref{section:applications} does not enclose all problems where the \pof is useful.
As the theory is now stated, applying the \pof requires only to follow Section~\ref{section:results}.
We plan to reword our work~\cite{BoAuBo21PWMayerCost} about discontinuous optimal control problems accordingly.
Another direction is related to parametric optimization~\cite{FeKaKr21Param,St18Parametric}, since our results are directly transferrable to this context.
We also plan to tackle \textit{counterfactuals for contextual optimization} from machine learning~\cite{SaChDeFoFrVi25SurveyContextualCO,VeBoHoHiDiSh22CounterfactualReview,VAFoPaVi24CFOPT}, which may relate well to the \pof.

\bibliography{bibliography.bib}

\end{document}

%% file: Figures/FirstCoord_2Vars_table.tex
\begin{tabular}{||c||c|c|c|c||}\hline
\hline        $x^0$ & $1^{st}~ \hat{x}^k \in \left[\pm5E^{-03}\right]$ & $1^{st}~ \hat{x}^k \in \left[\pm5E^{-06}\right]$ & $1^{st}~ \hat{x}^k \in \left[\pm5E^{-09}\right]$ & returned~$\hat{x}^k$ \\ \hline
\hline    $+9.753$ & $\hat{x}^{23} = +3.37E^{-03}$ & $\hat{x}^{37} = +8.86E^{-07}$ & $\hat{x}^{55} = +3.15E^{-09}$ & $\hat{x}^{64} = +7.41E^{-12}$ \\
\hline       $+\pi$ & $\hat{x}^{18} = +3.12E^{-03}$ & $\hat{x}^{31} = +3.93E^{-06}$ & $\hat{x}^{49} = +1.00E^{-09}$ & $\hat{x}^{56} = +6.89E^{-11}$ \\
\hline  $+\sqrt{2}$ & $\hat{x}^{13} = +4.65E^{-03}$ & $\hat{x}^{29} = +2.79E^{-06}$ & $\hat{x}^{44} = +3.08E^{-09}$ & $\hat{x}^{54} = +5.12E^{-11}$ \\
\hline       $+e+1$ & $\hat{x}^{17} = +2.08E^{-03}$ & $\hat{x}^{24} = +3.01E^{-06}$ & $\hat{x}^{45} = +3.29E^{-09}$ & $\hat{x}^{53} = +3.45E^{-11}$ \\
\hline    $-9.753$ & $\hat{x}^{17} = +4.61E^{-04}$ & $\hat{x}^{33} = +3.33E^{-06}$ & $\hat{x}^{51} = +3.96E^{-09}$ & $\hat{x}^{60} = +4.05E^{-12}$ \\
\hline       $-\pi$ & $\hat{x}^{20} = +2.94E^{-03}$ & $\hat{x}^{33} = +3.46E^{-06}$ & $\hat{x}^{43} = +2.33E^{-10}$ & $\hat{x}^{54} = +1.16E^{-10}$ \\
\hline  $-\sqrt{2}$ & $\hat{x}^{06} = +3.45E^{-03}$ & $\hat{x}^{27} = +1.13E^{-06}$ & $\hat{x}^{43} = +1.51E^{-09}$ & $\hat{x}^{52} = +7.67E^{-13}$ \\
\hline       $-e-1$ & $\hat{x}^{21} = +2.74E^{-03}$ & $\hat{x}^{36} = +2.40E^{-06}$ & $\hat{x}^{49} = +4.02E^{-09}$ & $\hat{x}^{57} = +5.70E^{-11}$ \\
\hline\hline
\end{tabular}

%% file: Figures/Polar_2Vars_table.tex
\begin{tabular}{||c||c|c|c|c||}\hline
\hline        $x^0$ & $1^{st}~ \hat{x}^k \in \left[\pm5E^{-03}\right]$ & $1^{st}~ \hat{x}^k \in \left[\pm5E^{-06}\right]$ & $1^{st}~ \hat{x}^k \in \left[\pm5E^{-09}\right]$ & returned~$\hat{x}^k$ \\ \hline
\hline          $0$ & $\hat{x}^{09} = +4.39E^{-03}$ & $\hat{x}^{29} = -8.47E^{-07}$ & $\hat{x}^{41} = +2.73E^{-09}$ & $\hat{x}^{52} = +5.39E^{-11}$ \\
\hline     $2^{-5}$ & $\hat{x}^{09} = +5.08E^{-04}$ & $\hat{x}^{27} = +4.42E^{-06}$ & $\hat{x}^{44} = -3.57E^{-09}$ & $\hat{x}^{54} = +3.83E^{-11}$ \\
\hline  $3\sqrt{2}$ & $\hat{x}^{17} = -1.14E^{-03}$ & $\hat{x}^{35} = -1.51E^{-06}$ & $\hat{x}^{51} = -1.95E^{-09}$ & $\hat{x}^{60} = +2.94E^{-11}$ \\
\hline       $4\pi$ & $\hat{x}^{25} = +1.37E^{-03}$ & $\hat{x}^{39} = +1.50E^{-06}$ & $\hat{x}^{57} = +2.05E^{-09}$ & $\hat{x}^{66} = -4.99E^{-11}$ \\
\hline          $5$ & $\hat{x}^{08} = +6.72E^{-05}$ & $\hat{x}^{27} = -1.45E^{-06}$ & $\hat{x}^{41} = -1.81E^{-09}$ & $\hat{x}^{48} = +5.71E^{-11}$ \\
\hline          $e$ & $\hat{x}^{08} = -1.75E^{-03}$ & $\hat{x}^{28} = +9.13E^{-07}$ & $\hat{x}^{41} = +4.04E^{-09}$ & $\hat{x}^{53} = -3.28E^{-11}$ \\
\hline        $e^2$ & $\hat{x}^{11} = -5.45E^{-04}$ & $\hat{x}^{25} = +4.01E^{-06}$ & $\hat{x}^{43} = +5.27E^{-10}$ & $\hat{x}^{53} = -5.47E^{-11}$ \\
\hline        $e^3$ & $\hat{x}^{30} = +3.83E^{-03}$ & $\hat{x}^{44} = +4.66E^{-06}$ & $\hat{x}^{59} = -4.19E^{-09}$ & $\hat{x}^{69} = +3.72E^{-12}$ \\
\hline\hline
\end{tabular}

%% file: Figures/NonlinComb_2Vars_table.tex
\begin{tabular}{||c||c|c|c|c||}\hline
\hline        $x^0$ & $1^{st}~ \hat{x}^k \in \left[\pm5E^{-03}\right]$ & $1^{st}~ \hat{x}^k \in \left[\pm5E^{-06}\right]$ & $1^{st}~ \hat{x}^k \in \left[\pm5E^{-09}\right]$ & returned~$\hat{x}^k$ \\ \hline
\hline       $-e^2$ & $\hat{x}^{107} = -4.17E^{-03}$ & $\hat{x}^{120} = -2.20E^{-06}$ & $\hat{x}^{138} = -1.68E^{-09}$ & $\hat{x}^{146} = -4.65E^{-11}$ \\
\hline       $-\pi$ & $\hat{x}^{026} = -1.76E^{-03}$ & $\hat{x}^{043} = -1.70E^{-06}$ & $\hat{x}^{057} = -3.88E^{-09}$ & $\hat{x}^{065} = -4.14E^{-11}$ \\
\hline  $-\sqrt{2}$ & $\hat{x}^{024} = -4.25E^{-04}$ & $\hat{x}^{040} = -4.95E^{-06}$ & $\hat{x}^{051} = -1.40E^{-09}$ & $\hat{x}^{062} = -8.01E^{-12}$ \\
\hline         $+e$ & $\hat{x}^{009} = -4.68E^{-04}$ & $\hat{x}^{026} = -2.78E^{-06}$ & $\hat{x}^{043} = -3.13E^{-09}$ & $\hat{x}^{051} = -1.06E^{-10}$ \\
\hline $+3\sqrt{2}$ & $\hat{x}^{017} = -2.66E^{-03}$ & $\hat{x}^{030} = -1.33E^{-07}$ & $\hat{x}^{046} = -2.41E^{-09}$ & $\hat{x}^{053} = -8.40E^{-11}$ \\
\hline      $+2e^2$ &               /                &               /                &               /                & $\hat{x}^{201} = +5.27E^{+00}$ \\
\hline      $+4\pi$ &               /                &               /                &               /                & $\hat{x}^{201} = +5.27E^{+00}$ \\
\hline       $+e^3$ &               /                &               /                &               /                & $\hat{x}^{201} = +5.27E^{+00}$ \\
\hline\hline
\end{tabular}

%% file: Figures/Dim2_3Vars_table.tex
\begin{tabular}{||c||c|c|c|c||}\hline
\hline        $x^0$ & $1^{st}~ \hat{x}^k \leq 5E^{-02}$ & $1^{st}~ \hat{x}^k \leq 5E^{-05}$ & $1^{st}~ \hat{x}^k \leq 5E^{-08}$ & returned~$\hat{x}^k$ \\ \hline
\hline     $(-2,2)$ & $\hat{x}^{032} = 3.52E^{-02}$ & $\hat{x}^{110} = 4.03E^{-05}$ & $\hat{x}^{194} = 4.79E^{-08}$ & $\hat{x}^{220} = 1.10E^{-08}$ \\
\hline$(\frac{-1}{100},e^2)$ & $\hat{x}^{037} = 2.02E^{-02}$ & $\hat{x}^{100} = 3.99E^{-05}$ & $\hat{x}^{176} = 1.10E^{-08}$ & $\hat{x}^{198} = 9.03E^{-10}$ \\
\hline$(\frac{-\pi}{2},\frac{7}{4})$ & $\hat{x}^{027} = 6.34E^{-03}$ & $\hat{x}^{095} = 4.60E^{-05}$ & $\hat{x}^{177} = 1.23E^{-08}$ & $\hat{x}^{195} = 5.51E^{-09}$ \\
\hline$(\frac{-\pi}{4},e^{\frac{1}{2}})$ & $\hat{x}^{040} = 3.66E^{-02}$ & $\hat{x}^{126} = 1.46E^{-05}$ & $\hat{x}^{235} = 3.45E^{-08}$ & $\hat{x}^{253} = 4.54E^{-09}$ \\
\hline$(\frac{1}{4},\frac{1}{4})$ & $\hat{x}^{022} = 2.85E^{-02}$ & $\hat{x}^{119} = 2.90E^{-05}$ & $\hat{x}^{193} = 1.84E^{-08}$ & $\hat{x}^{212} = 2.88E^{-09}$ \\
\hline$(\frac{3\pi}{2},\frac{1}{\sqrt{8}})$ & $\hat{x}^{039} = 3.29E^{-02}$ & $\hat{x}^{124} = 1.96E^{-05}$ & $\hat{x}^{186} = 4.86E^{-08}$ & $\hat{x}^{226} = 5.22E^{-09}$ \\
\hline $(e^2,2\pi)$ & $\hat{x}^{036} = 3.43E^{-02}$ & $\hat{x}^{109} = 3.00E^{-05}$ & $\hat{x}^{171} = 4.40E^{-08}$ & $\hat{x}^{211} = 8.19E^{-09}$ \\
\hline$(e^2,\frac{-1}{11})$ & $\hat{x}^{029} = 1.48E^{-02}$ & $\hat{x}^{093} = 4.77E^{-05}$ & $\hat{x}^{148} = 2.48E^{-08}$ & $\hat{x}^{172} = 1.48E^{-09}$ \\
\hline\hline
\end{tabular}